\documentclass[a4paper]{amsart}
\usepackage[utf8]{inputenc}
\usepackage[T1]{fontenc}

\usepackage{hyperref}
\usepackage[backrefs,lite]{amsrefs}

\usepackage{floatrow}
\usepackage{subcaption}
\usepackage{amssymb}              
\usepackage{longtable}

\usepackage{tikz,tkz-euclide}
\usepackage[all]{xy}
\usepackage{dsfont}
\usepackage{pdflscape}
\usepackage{bigdelim}
\usepackage{multirow}
\usepackage{multicol}
\usepackage{cancel}
\usepackage{array}
\usepackage{booktabs}
\usepackage{mathtools}
\usepackage{paralist}
\usepackage{cellspace}
\usepackage{setspace}

\usepackage{tcolorbox}

\CompileMatrices

\newcolumntype{C}{>{$}c<{$}}
\newcolumntype{L}{>{$}l<{$}}
\newcolumntype{R}{>{$}r<{$}}

\usepackage{hyperref}
\usepackage[capitalise]{cleveref}

\newtheorem{theorem}{Theorem}[section]

\newtheorem{lemma}[theorem]{Lemma}
\newtheorem{proposition}[theorem]{Proposition}
\newtheorem{corollary}[theorem]{Corollary}

\theoremstyle{definition}

\newtheorem{definition}[theorem]{Definition}
\newtheorem{construction}[theorem]{Construction}

\newtheorem{example}[theorem]{Example}
\newtheorem{remark}[theorem]{Remark}

\newtheorem{setting}[theorem]{Setting}

\numberwithin{equation}{subsection}

\def\vector2#1#2{\left(\begin{array}{c} #1 \\ #2 \end{array}\right)}
\def\Cl{{\rm Cl}}

\def\CC{{\mathbb C}}
\def\KK{{\mathbb K}}

\def\ZZ{{\mathbb Z}}

\def\QQ{{\mathbb Q}}
\def\PP{{\mathbb P}}

\def\KKK{\mathcal{K}}

\def\OOO{{\mathcal O}}
\def\RRR{{\mathcal R}}

\DeclareMathOperator{\Eff}{Eff}
\DeclareMathOperator{\Mov}{Mov}
\DeclareMathOperator{\SAmple}{SAmple}
\DeclareMathOperator{\Ample}{Ample}
\DeclareMathOperator{\rk}{rank}

\def\grad{{\rm grad}}

\def\quot{/\!\!/}

\def\diag{{\rm diag}}

\def\bangle#1{{\langle #1 \rangle}}

\def\Aut{{\rm Aut}}

\def\GL{{\rm GL}}

\def\GL{{\rm GL}}

\def\Spec{{\rm Spec}}

\def\cone{{\rm cone}}
\def\conv{{\rm conv}}

\def\pr{{\rm pr}}

\def\rlv{{\rm rlv}}

\usepackage[backrefs,lite]{amsrefs}

\title[On smooth Fano fourfolds of Picard number two]%
{On smooth Fano fourfolds of Picard number two}

\thanks{The second author has been 
partially supported by Proyecto 
FONDECYT Regular N. 1190777.}

\author[J.~Hausen, A.~Laface and C.~Mauz]{J\"urgen~Hausen, Antonio Laface and Christian Mauz}

\address{Mathematisches Institut, Universit\"at T\"ubingen,
Auf der Morgenstelle 10, 72076 T\"ubingen, Germany}
\email{juergen.hausen@uni-tuebingen.de}

\address{Departamento de Matem\'atica, Universidad de Concepci\'on,
Casilla 160-C, Concepci\'on, Chile}
\email{alaface@udec.cl}

\address{Mathematisches Institut, Universit\"at T\"ubingen,
Auf der Morgenstelle 10, 72076 T\"ubingen, Germany}
\email{mauz@math.uni-tuebingen.de}


\begin{document}

\begin{abstract}
We classify the smooth Fano 4-folds of Picard number 
two that have a general hypersurface Cox ring.
\end{abstract}

\maketitle

\section{Introduction}

By a Fano variety, we mean a normal projective 
complex variety with an ample anticanonical 
divisor.
Our aim is to contribute to the explicit 
classification of smooth Fano varieties.
In dimension two, these are the well known
smooth del Pezzo surfaces.
The smooth Fano threefolds have been classified 
by Iskovskikh~\cites{Is1,Is2}
and Mori/Mukai~\cites{MM1,MM2}.
From dimension four on the classification 
problem is widely open in general, 
but there are trendsetting partial results, 
such as Batyrev's classification of the 
smooth toric Fano fourfoulds~\cite{Ba}.

In the present article, we focus on the 
case of Picard number at most two.
In this situation, all smooth Fano varieties 
coming with a torus action of complexity 
at most one are known~\cite{FaHaNi} and 
in~\cite{HaHiWr} one finds a natural 
extension to complexity two.
As in~\cites{FaHaNi,HaHiWr},
our approach goes via the Cox ring.
Recall that for any normal projective 
variety $X$ with finitely generated 
divisor class group $\Cl(X)$,
the Cox ring is defined as 
$$ 
\RRR(X) 
\ =  \ 
\bigoplus_{\Cl(X)} \Gamma(X,\OOO_X(D)).
$$
In case of a smooth Fano variety $X$,
the Cox ring is known to be a finitely 
generated $\CC$-algebra~\cite{BCHM}.
We restrict our attention to simply 
structured Cox rings:
We say that a variety $X$ 
with divisor class group $\Cl(X) = K$ 
has a \emph{hypersurface Cox ring}
if we have a $K$-graded presentation
$$ 
\RRR(X)
\ = \
R_g
\ = \ 
\CC[T_1,\ldots,T_r] / \bangle{g},
$$ 
where $g$ is homogeneous of degree
$\mu \in K$ and $T_1,\ldots,T_r$
define a minimal system of $K$-homogeneous 
generators.
In this situation, we call $\RRR(X)$
\emph{spread} if each monomial of degree
$\mu$ is a convex combination
of monomials of $g$.
Moreover, we call $\RRR(X)$
\emph{general (smooth, Fano)} if $g$ 
admits an open neighbourhood 
$U$ in the vector space of all
$\mu$-homogeneous polynomials
such that every $h \in U$ yields a
hypersurface Cox ring~$R_{h}$ of a
normal (smooth, Fano) variety~$X_h$ with
divisor class group~$K$;
see also Definition~\ref{def:genhypcr}.

Among the del Pezzo surfaces,
there are no smooth ones
with a hypersurface Cox ring, 
but in the singular case, we encounter 
many examples~\cite{Der}.
The first examples of smooth Fano varieties 
with a hypersurface Cox ring show
up in dimension three;
see~\cite[Thms.~4.1 and~4.5]{DHHKL},
where, based on the classifications mentioned 
before, the Cox rings of the smooth Fano 
threefolds of Picard numbers one and two 
have been computed.
For the smooth Fano fourfolds of Picard
number one having a hypersurface Cox ring
we refer to numbers 2, 4, 5 and~8 in K\"uchle's 
list~\cite[Prop.~2.2.1]{Ku} in case of Fano
index one and in case of higher Fano index
to the numbers~14, 15, 18, 19, 20 and~22
in the list of Przyjalkowski and
Shramov~\cite[p.~12]{PrSh}.

\goodbreak 

We approach our main result,
concerning Fano fourfolds of Picard number
two. The notation is as follows.
For any hypersurface Cox
ring~$\RRR(X) = R_g$
graded by~$\Cl(X) = K$, we write 
$w_i = \deg(T_i) \in K$
for the generator degrees 
and $\mu = \deg(g) \in K$ for the
degree of the relation.
Moreover, in this setting, the anticanonical 
class of $X$ is given by 
$$ 
\mathcal{-K}
\ = \ 
w_1 + \ldots + w_r - \mu
\ \in \ 
\Cl(X)
\ =  \ 
K.
$$
If $R_g$ is the Cox ring 
of a Fano variety $X$,
then $X$ can be reconstructed 
as the GIT quotient of the set 
of $(-\KKK)$-semistable points 
of $\Spec \, R_g$ by the quasitorus 
$\Spec \, \CC[K]$. 
In this setting, we refer to 
the Cox ring generator degrees
$w_1,  \ldots, w_r \in K$ and the 
relation degree $\mu \in K$ as the
\emph{specifying data of the Fano variety $X$}.
In the case of a smooth Fano fourfould~$X$
of Picard number two, $\Cl(X)$
equals $\ZZ^2$ and thus $\Spec \, \CC[K]$ is
a two-dimensional torus.
Hence the hypersurface Cox ring $R_g$
is of dimension six and has seven
generators.


\begin{theorem} 
\label{thm:candidates}
The following table lists the specifying data 
$w_1, \ldots, w_7$ and $\mu$  in $\Cl(X) = \ZZ^2$,
the anticanonical class $-\KKK$ and $\KKK^4$
for all smooth Fano fourfolds of Picard number
two with a spread hypersurface Cox ring.

\begin{center}
\footnotesize
\setlength\arraycolsep{2pt}
\addtolength\tabcolsep{-.35em}

\newcounter{candidateno}

\newcommand{\mycolumnwidth}{.485\textwidth}

\begin{minipage}[t]{\mycolumnwidth}
\begin{tabular}[t]{>{\refstepcounter{candidateno}\thecandidateno}LCCCR}
\toprule
\multicolumn{1}{c}{No.} & [w_1, \dotsc, w_7] & \deg(g) & - \KKK & \KKK^4 \\
\midrule


\label{cand:deginrelint-ia-1}& \multirow{6}{*}{
	$\tiny \begin{bmatrix}
		1 & 1 & 1 & 1 & 0 & 0 & 0 \\
		0 & 0 & 0 & 0 & 1 & 1 & 1
	\end{bmatrix}$ } &
  (1,1) & (3,2) & 432 \\
\label{cand:deginrelint-ia-2} & & (2,1) & (2,2) & 256 \\
\label{cand:deginrelint-ia-3} & & (3,1) & (1,2) & 80\\
\label{cand:deginrelint-ia-4} & & (1,2) & (3,1) & 270\\
\label{cand:deginrelint-ia-5} & & (2,2) & (2,1) & 112\\
\label{cand:deginrelint-ia-6} & & (3,2) & (1,1) & 26 \\ \midrule

\label{cand:deginrelint-ia-7} & \multirow{4}{*}{
	$\tiny\begin{bmatrix*}[r]
		1 & 1 & 1 & 1 & 0 & 0 & -1 \\
		0 & 0 & 0 & 0 & 1 & 1 &  1
	\end{bmatrix*}$ } &
  (1,1) & (2,2) & 416 \\ 
\label{cand:deginrelint-ia-8} & & (1,2) & (2,1) & 163 \\
\label{cand:deginrelint-ia-9} & & (2,1) & (1,2) & 224 \\
\label{cand:deginrelint-ia-10} & & (2,2) & (1,1) & 52 \\ \midrule

\label{cand:deginrelint-ia-11} & \multirow{2}{*}{
	$\tiny\begin{bmatrix*}[r]
		1 & 1 & 1 & 1 & 0 & 0 & -2 \\
		0 & 0 & 0 & 0 & 1 & 1 &  1
	\end{bmatrix*}$ } &
  (1,1) & (1,2) & 464  \\
\label{cand:deginrelint-ia-12} & & (1,2) & (1,1) & 98 \\ \midrule

\label{cand:deginrelint-ib-1} & \multirow{2}{*}{
	$\tiny\begin{bmatrix}
		1 & 1 & 1 & 1 & 0 & 0 & 0 \\
		0 & 0 & 0 & 1 & 1 & 1 & 1
	\end{bmatrix}$ } &
  (1, 2) & (3, 2) & 352  \\
\label{cand:deginrelint-ib-3} & & (2, 3) & (2, 1) & 65  \\ \midrule

\label{cand:deginrelint-ib-4} & {\tiny\begin{bmatrix*}[r]
		1 & 1 & 1 & 1 & 0 & 0 & -1 \\
		0 & 0 & 0 & 1 & 1 & 1 &  1
	\end{bmatrix*}} &
  (1, 3) & (2, 1) & 83  \\ \midrule

\label{cand:deginrelint-ii-1}& \multirow{2}{*}{
	$\tiny\begin{bmatrix}
		1 & 1 & 1 & 1 & 1 & 0 & 0 \\
		0 & 0 & 0 & 0 & 1 & 1 & 1
	\end{bmatrix}$ } &
  (2, 1) & (3, 2) & 352 \\
\label{cand:deginrelint-ii-2} & & (3, 2) & (2, 1) & 81 \\ \midrule

& \multirow{3}{*}{
	$\tiny\begin{bmatrix*}[r]
		 1 & 1 & 1 & 1 & 0 & 0 & 0 \\
		-1 & 0 & 0 & 0 & 1 & 1 & 1
	\end{bmatrix*}$ } &
  (3, 1) & (1, 1) & 38 \\
& & (2, 1) & (2, 1) & 192 \\
\label{cand:deginrelint-BI1-3} & & (1, 1) & (3, 1) & 432 \\ \midrule

&	{\tiny\begin{bmatrix*}[r]
		 1 & 1 & 1 & 1 & 1 & 0 & 0 \\
		-1 & 0 & 0 & 0 & 1 & 1 & 1
	\end{bmatrix*} } &
 (3, 1) & (2, 1) & 113 \\ \midrule



\label{cand:II-1-1} & \multirow{2}{*}{
	$\tiny\begin{bmatrix*}[r]
		1 & 1 & 1 & 1 & 0 & 0 & 0 \\
		0 & 0 & 1 & 1 & 1 & 1 & 1
	\end{bmatrix*}$ } &
  (2,2) & (2,3) & 272 \\
\label{cand:II-1-2} & & (3,3) & (1,2) & 51 \\ \midrule

\label{cand:II-1-3} & {\tiny\begin{bmatrix*}[r]
		1 & 1 & 1 & 2 & 0 & 0 & 0 \\
		0 & 0 & 1 & 2 & 1 & 1 & 1
	\end{bmatrix*}} &
	(4,4) & (1,2) & 34 \\ \midrule

\label{cand:II-1-4} & {\tiny\begin{bmatrix*}[r]
		1 & 1 & 2 & 3 & 0 & 0 & 0 \\
		0 & 0 & 2 & 3 & 1 & 1 & 1
	\end{bmatrix*}} &
	(6,6) & (1,2) & 17 \\ \midrule

\label{cand:II-2-1} & {\tiny\begin{bmatrix*}[r]
		1 & 1 & 1 & 0 & 0 & 0 & 0 \\
		0 & 0 & 1 & 1 & 1 & 1 & 1
	\end{bmatrix*}} &
	(2,2) & (1,3) & 216 \\ \midrule

\label{cand:II-2-2} &  {\tiny\begin{bmatrix*}[r]
		1 & 1 & 1 & 0 & 0 & 0 & 0 \\
		0 & 0 & 2 & 1 & 1 & 1 & 1
	\end{bmatrix*}} &
	(2,4) & (1,2) & 64 \\ \midrule

\label{cand:II-2-3} &  {\tiny\begin{bmatrix*}[r]
		1 & 1 & 1 & 0 & 0 & 0 & 0 \\
		0 & 0 & 3 & 1 & 1 & 1 & 1
	\end{bmatrix*}} &
	(2,6) & (1,1) & 8 \\ \bottomrule

\end{tabular}
\end{minipage}
\begin{minipage}[t]{\mycolumnwidth}
\begin{tabular}[t]{>{\refstepcounter{candidateno}\thecandidateno}LCCCR}
\toprule
\multicolumn{1}{c}{No.} & [w_1, \dotsc, w_7] & \deg(g) & - \KKK & \KKK^4 \\
\midrule

\label{cand:II-3-1} & \multirow{2}{*}{
	$\tiny\begin{bmatrix*}[r]
		1 & 1 & 1 & 1 & 0 & 0 & 0 \\
		0 & 0 & 0 & 1 & 1 & 1 & 1
	\end{bmatrix*}$ } &
  (2,2) & (2,2) & 192 \\
& & (3,3) & (1,1) & 18 \\ \midrule

\label{cand:II-3-2} & {\tiny\begin{bmatrix*}[r]
		1 & 1 & 1 & 2 & 0 & 0 & 0 \\
		0 & 0 & 0 & 1 & 1 & 1 & 1
	\end{bmatrix*}} &
	(4,2) & (1,2) & 48 \\ \midrule

\label{cand:II-3-4} & {\tiny\begin{bmatrix*}[r]
		1 & 1 & 1 & 2 & 0 & 0 & 0 \\
		0 & 0 & 0 & 2 & 1 & 1 & 1
	\end{bmatrix*}} &
	(4,4) & (1,1) & 12 \\ \midrule


\label{cand:III-0-1} &  {\tiny\begin{bmatrix}
		1 & 1 & 2 & 1 & 0 & 0 & 0 \\
		0 & 1 & 3 & 2 & 1 & 1 & 1
\end{bmatrix}} &
 (4, 6) & (1, 3) & 50  \\ \midrule

\label{cand:III-1-1} \label{cand:1contr-0-1} & \multirow{3}{*}{
	$\tiny\begin{bmatrix}
		1 & 1 & 1 & 1 & 1 & 0 & 0 \\
		0 & 1 & 1 & 1 & 1 & 1 & 1
	\end{bmatrix}$ } &
 (2, 2) & (3, 4) & 378 \\
& & (3, 3) & (2, 3) & 144 \\ 
& & (4, 4) & (1, 2) & 20 \\ \midrule

&  {\tiny\begin{bmatrix}
		1 & 1 & 1 & 1 & 2 & 0 & 0 \\
		0 & 1 & 1 & 1 & 2 & 1 & 1
\end{bmatrix}} &
 (4, 4) & (2, 3) & 96 \\ \midrule

\label{cand:III-1-5} &  {\tiny\begin{bmatrix}
		1 & 1 & 1 & 1 & 3 & 0 & 0 \\

		0 & 1 & 1 & 1 & 3 & 1 & 1
\end{bmatrix}} &
 (6, 6) & (1, 2) & 10 \\ \midrule

\label{cand:III-1-6}&  {\tiny\begin{bmatrix}
		1 & 1 & 1 & 2 & 3 & 0 & 0 \\
		0 & 1 & 1 & 2 & 3 & 1 & 1
\end{bmatrix}} &
 (6, 6) & (2, 3) & 48 \\ \midrule

\label{cand:III-3-1} & \multirow{2}{*}{
	$\tiny\begin{bmatrix}
		1 & 1 & 1 & 1 & 0 & 0 & 0 \\
		0 & 1 & 1 & 1 & 1 & 1 & 1
	\end{bmatrix}$ } &
  (2, 2) & (2, 4) & 352 \\
& & (3, 3) & (1, 3) & 99 \\ \midrule

\label{cand:1contr-1_2_1_1-3} & \multirow{2}{*}{
	$\tiny\begin{bmatrix}
		1 & 1 & 1 & 1 & 0 & 0 & 0 \\
		0 & 2 & 2 & 2 & 1 & 1 & 1
	\end{bmatrix}$ } &
  (2, 4) & (2, 5) & 304 \\
\label{cand:1contr-1_2_1_1-4}  & & (3, 6) & (1, 3) & 54 \\ \midrule

&  {\tiny\begin{bmatrix}
		1 & 1 & 1 & 2 & 0 & 0 & 0 \\
		0 & 1 & 1 & 2 & 1 & 1 & 1
\end{bmatrix}} &
 (4, 4) & (1, 3) & 66 \\ \midrule

\label{cand:III-3-6} \label{cand:1contr-1_2_1_1-6}  &  {\tiny\begin{bmatrix}
		1 & 1 & 1 & 2 & 0 & 0 & 0 \\
		0 & 2 & 2 & 4 & 1 & 1 & 1
\end{bmatrix}} &
 (4, 8) & (1, 3) & 36 \\ \midrule

\label{cand:III-3-7} &  {\tiny\begin{bmatrix}
		1 & 1 & 2 & 3 & 0 & 0 & 0 \\
		0 & 1 & 2 & 3 & 1 & 1 & 1
\end{bmatrix}} &
 (6, 6) & (1, 3) & 33  \\ \midrule

\label{cand:III-3-8} \label{cand:1contr-1_2_1_1-8} &  {\tiny\begin{bmatrix}
		1 & 1 & 2 & 3 & 0 & 0 & 0 \\
		0 & 2 & 4 & 6 & 1 & 1 & 1
\end{bmatrix}} &
 (6, 12) & (1, 3) & 18 \\ \midrule

\label{cand:III-3-9} \label{cand:1contr-1_2_1_2_2-1} &  {\tiny\begin{bmatrix}
		1 & 1 & 1 & 1 & 1 & 0 & 0 \\
		0 & 1 & 1 & 1 & 2 & 1 & 1
\end{bmatrix}} &
 (2, 2) & (3, 5) & 433 \\ \midrule

\label{cand:III-3-10} \label{cand:1contr-1_2_1_2_2-2} & {\tiny\begin{bmatrix}
		1 & 1 & 1 & 1 & 1 & 0 & 0 \\
		0 & 2 & 2 & 2 & 3 & 1 & 1
	\end{bmatrix} } &
(3, 6) & (2, 5) & 145 \\ \midrule

\label{cand:III-4-1} &  {\tiny\begin{bmatrix}
		1 & 1 & 1 & 1 & 0 & 0 & 0 \\
		0 & 1 & 1 & 2 & 1 & 1 & 1
\end{bmatrix}} &
 (2, 4) & (2, 3) & 144 \\ \bottomrule

\end{tabular}
\end{minipage}

\begin{minipage}[t]{\mycolumnwidth}
\begin{tabular}[t]{>{\refstepcounter{candidateno}\thecandidateno}LCCCR}
\toprule
\multicolumn{1}{c}{No.} & [w_1, \dotsc, w_7] & \deg(g) & - \KKK & \KKK^4 \\
\midrule

\label{cand:III-4-2} &  {\tiny\begin{bmatrix}
		1 & 1 & 1 & 2 & 0 & 0 & 0 \\
		0 & 1 & 1 & 3 & 1 & 1 & 1
\end{bmatrix}} &
 (4, 6) & (1, 2) & 22 \\ \midrule

\label{cand:III-4-3} &  {\tiny\begin{bmatrix}
		1 & 1 & 1 & 2 & 1 & 0 & 0 \\
		0 & 1 & 1 & 3 & 2 & 1 & 1
\end{bmatrix}} &
 (4, 6) & (2, 3) & 65  \\ \midrule


\label{cand:IV-1-1} \label{cand:2contr-1} & \multirow{2}{*}{
	$\tiny\begin{bmatrix*}[r]
		 1 & 1 & 1 & 1 & 1 & 1 & 0 \\
		-1 & 0 & 0 & 0 & 0 & 1 & 1
	\end{bmatrix*}$} &
 (2,0) & (4,1) & 431 \\
\label{cand:IV-1-2} & & (4,0) & (2,1) & 62 \\ \midrule

\label{cand:IV-1-3} \label{cand:2contr-2} &  {\tiny\begin{bmatrix*}[r]
		 1 & 1 & 1 & 1 & 1 & 2 & 0 \\
		-1 & 0 & 0 & 0 & 0 & 1 & 1
\end{bmatrix*}} &
 (3,0) & (4,1) & 376  \\ \midrule

\label{cand:IV-1-4} \label{cand:2contr-3} &  {\tiny\begin{bmatrix*}[r]
		 1 & 1 & 1 & 1 & 1 & 3 & 0 \\
		-1 & 0 & 0 & 0 & 0 & 1 & 1
\end{bmatrix*}} &
 (4,0) & (4,1) & 341  \\ \midrule

\label{cand:IV-1-5} &  {\tiny\begin{bmatrix*}[r]
		 1 & 1 & 1 & 1 & 3 & 1 & 0 \\
		-1 & 0 & 0 & 0 & 0 & 1 & 1
\end{bmatrix*}} &
 (6,0) & (2,1) & 31  \\ \midrule

\label{cand:V-1-1} & {\tiny \begin{bmatrix*}[r]
	1 & 1 & 1 & 1 & 3 & 0 & 0 \\ 
	0 & 0 & 0 & 0 & 0 & 1 & 1 \\ 
\end{bmatrix*}} & 
(6, 0) & (1, 2) & 16 \\ \midrule

\label{cand:V-1-2} & {\tiny \begin{bmatrix*}[r]
	1 & 1 & 1 & 2 & 3 & 0 & 0 \\ 
	0 & 0 & 0 & 0 & 0 & 1 & 1 \\ 
\end{bmatrix*}} &
(6,0) & (2, 2) & 64 \\ \bottomrule

\end{tabular}
\end{minipage}
\begin{minipage}[t]{\mycolumnwidth}
\begin{tabular}[t]{>{\refstepcounter{candidateno}\thecandidateno}LCCCR}
\toprule
\multicolumn{1}{c}{No.} & [w_1, \dotsc, w_7] & \deg(g) & - \KKK & \KKK^4 \\
\midrule

\label{cand:V-1-3} & {\tiny \begin{bmatrix*}[r]
	1 & 1 & 1 & 2 & 3 & 1 & 0 \\ 
	0 & 0 & 0 & 0 & 0 & 1 & 1 \\ 
\end{bmatrix*}} &
(6,0) & (3, 2) & 80 \\ \midrule

\label{cand:V-1-4} & {\tiny \begin{bmatrix*}[r]
	1 & 1 & 1 & 1 & 2 & 0 & 0 \\ 
	0 & 0 & 0 & 0 & 0 & 1 & 1 \\ 
\end{bmatrix*}} &
(4,0) & (2, 2) & 128 \\ \midrule

\label{cand:V-1-5} & {\tiny \begin{bmatrix*}[r]
	1 & 1 & 1 & 1 & 2 & 1 & 0 \\ 
	0 & 0 & 0 & 0 & 0 & 1 & 1 \\ 
\end{bmatrix*}} &
(4,0) & (3, 2) & 160 \\ \midrule

\label{cand:V-1-6} & {\tiny \begin{bmatrix*}[r]
	1 & 1 & 1 & 1 & 1 & 0 & 0 \\ 
	0 & 0 & 0 & 0 & 0 & 1 & 1 \\ 
\end{bmatrix*}} &
(3, 0) & (2, 2) & 192 \\ \midrule

\label{cand:V-1-7} & {\tiny \begin{bmatrix*}[r]
	1 & 1 & 1 & 1 & 1 & 1 & 0 \\ 
	0 & 0 & 0 & 0 & 0 & 1 & 1 \\ 
\end{bmatrix*}} &
(3, 0) & (3, 2) & 240 \\ \midrule

\label{cand:V-1-8} & {\tiny \begin{bmatrix*}[r]
	1 & 1 & 1 & 1 & 1 & 0 & 0 \\ 
	0 & 0 & 0 & 0 & 0 & 1 & 1 \\ 
\end{bmatrix*}} & 
(2,0) & (3, 2) & 432 \\ \midrule

\label{cand:V-1-9} & {\tiny \begin{bmatrix*}[r]
	1 & 1 & 1 & 1 & 1 & 1 & 0 \\ 
	0 & 0 & 0 & 0 & 0 & 1 & 1 \\ 
\end{bmatrix*}} & 
(2,0) & (4, 2) & 480 \\ \midrule

\label{cand:V-1-10} & {\tiny \begin{bmatrix*}[r]
	1 & 1 & 1 & 1 & 1 & 2 & 0 \\ 
	0 & 0 & 0 & 0 & 0 & 1 & 1 \\ 
\end{bmatrix*}} & 
(2,0) & (5, 2) & 624 \\ \bottomrule

\end{tabular}
\end{minipage}
\end{center}

\bigskip
\noindent
Any two smooth Fano fourfolds of Picard number
two with specifying data from distinct items of the table
are not isomorphic to each other.
Moreover, each of the items~1 to~\thecandidateno \
even defines a general smooth Fano hypersurface Cox ring and
thus provides the specifying data for a whole family of smooth
Fano fourfolds.
\end{theorem}

Let us compare the result with existing classifications.
Wi\'{s}niewski classified in~\cite{Wi} the smooth 
Fano fourfolds of Picard number and Fano index 
at least two,
where the Fano index is the largest 
integer $\iota$ such that $-\KKK = \iota H$
holds with an ample divisor~$H$.

\begin{remark}
In eight cases, the families listed 
in Theorem~\ref{thm:candidates}
consist of varieties of Fano 
index two and in all other cases,
the varieties are of Fano index 
one.
The conversion between
Theorem~\ref{thm:candidates} 
and Wi\'{s}niewski's results
as presented in the 
table~\cite[12.7]{AG5}
is as follows:
\bigskip

\begin{center}
{\small
\begin{tabular}{l|l|l|l|l|l|l|l|l}
Thm.~\ref{thm:candidates} 
& 
\ref{cand:deginrelint-ia-2} 
&
\ref{cand:deginrelint-ia-7}  
&
\ref{cand:II-3-1}  
&
\ref{cand:III-3-1}  
&
\ref{cand:V-1-2}  
&
\ref{cand:V-1-4}  
&
\ref{cand:V-1-6}  
&
\ref{cand:V-1-9}  
\\
\hline
\cite[12.7]{AG5} 
&
5
& 
12   
& 
4
& 
10
& 
1
& 
2
& 
3
& 
13
\end{tabular}
}
\end{center}
\end{remark}

Theorem~\ref{thm:candidates} has no overlap
with Batyrev's classification~\cite{Ba} of 
smooth toric Fano fourfolds.
Indeed, toric varieties have polynomial rings 
as Cox rings which are by definition
no hypersurface Cox rings.
However, there is some interaction with 
the case of torus actions of complexity 
one.

\begin{remark}
Eleven of the families of Theorem~\ref{thm:candidates}
admit small degenerations to smooth Fano fourfolds 
with an effective action of a three-dimensional torus.
Here are these families and the corresponding 
varieties  from~\cite[Thm.~1.2]{FaHaNi}.

\bigskip

\begin{center}
{\small
\begin{longtable}{l|l}
Thm.~\ref{thm:candidates} & \cite[Thm.~1.2]{FaHaNi} 
\\
\hline
\ref{cand:deginrelint-ia-1}  & 4.A:  $m = 1$,  $c = 0$ 
\\
\ref{cand:deginrelint-ia-4}  & 4.C:  $m = 1$  
\\
\ref{cand:deginrelint-ia-7}  & 2   
\\
\ref{cand:deginrelint-ib-1}  & 5:   $m = 1$ 
\\
\ref{cand:deginrelint-BI1-3}  & 4.A: $m = 1$, $c=-1$ 
\\
\ref{cand:III-1-1}  & 1  
\\
\ref{cand:1contr-1_2_1_2_2-1}  & 10:  $m = 2$ 
\\
\ref{cand:2contr-1}  & 7:  $m = 1$ 
\\
\ref{cand:V-1-8}  & 12: $m = 2$, $a=b=c=0$ 
\\
\ref{cand:V-1-9} & 11: $m = 2$, $a_2 = 1$   
\\
\ref{cand:V-1-10}  & 11:   $m=2$, $a_2=2$ 
\end{longtable}
}
\end{center}
Moreover, observe that for the families 1, 20, 48, 53, 65, 66 and 67 
of Theorem~\ref{thm:candidates} the degeneration process
gives a Fano smooth intrinsic quadric; compare~\cite[Thm.~1.3]{FaHa}.
\end{remark}

\begin{remark}
Coates, Kasprzyk and Prince classified 
in~\cite{CKP} the smooth Fano fourfolds
that arise as complete intersections of ample
divisors in smooth toric Fano varieties
of dimension at most eight. 
Comparing anticanonical self-intersection numbers
as well as the first six coefficients of the
Hilbert series yields that at least the~17
families 14, 15, 24, 25, 28, 30, 32, 33,
38, 44, 45, 46, 47, 51, 52, 57 and 58 
of Theorem~\ref{thm:candidates} do not show up~\cite{CKP}.
\end{remark}

In Sections~\ref{sec:geometry} to~\ref{sec:defaut},
we investigate the 
geometry of the Fano varieties from Theorem~\ref{thm:candidates}.
We take a look at the elementary contractions, 
see Proposition~\ref{prop:contractions}, we determine the
Hodge numbers, see Propositions~\ref{prop:hodgediamond}
and~\ref{prop:hodgenum} and in many cases
the inifinitesimal deformations, see
Corollary~\ref{cor:infdef}.

\tableofcontents

\section{Factorial gradings}

Here we provide the first part of the 
algebraic and combinatorial tools used 
in our classification.
We recall the basic concepts on factorially 
graded algebras and, as a new result, 
prove Proposition~\ref{prop:hypmov}, 
locating the relation degrees of a factorially 
graded complete intersection algebra.
Moreover, we recall and discuss the GIT-fan 
of the quasitorus action associated with a 
graded affine algebra.

For the moment, $\KK$ is any field.
Let~$R$ be a \emph{$K$-graded algebra},
which, in this article, means that~$K$ is 
a finitely generated abelian group and 
$R$ is a $\KK$-algebra coming 
with a direct sum decomposition into 
$\KK$-vector subspaces
$$ 
R \ = \ \bigoplus_{w \in K} R_w
$$
such that $R_wR_{w'} \subseteq R_{w+w'}$ holds 
for all $w,w' \in R$.
An element $f \in R$ is \emph{homogeneous} if 
$f \in R_w$ holds for some $w \in K$; 
in that case, $w$ is the \emph{degree} of $f$,
written $w = \deg(f)$.
We say that $R$ is \emph{$K$-integral} 
if it has no homogeneous zero divisors.

Consider the rational vector space 
$K_\QQ := K \otimes_\ZZ \QQ$ associated 
with $K$.
The \emph{effective cone} of $R$  
is the convex cone generated by 
all degrees admitting a non-zero 
homogeneous element:
$$ 
\Eff(R) 
\ := \ 
\cone(w \in K; \ R_w \ne 0)
\ \subseteq \ 
K_\QQ.
$$
The $K$-grading of $R$ is called \emph{pointed}
if $R_0 = \KK$ holds and the effective cone 
$\Eff(R)$ contains no line.
Note that $\Eff(R)$ is polyhedral, if the 
$\KK$-algebra $R$ is finitely generated.

\begin{lemma}
\label{lem:coarshom}
Let $R$ be a $K$-graded algebra.
Assume that $R$ is $K$-integral and every 
homogeneous unit of $R$ is of degree zero.
\begin{enumerate}
\item
If $R_0 = \KK$ holds, then the $K$-grading 
is pointed and for every 
non-zero torsion element $w \in K$,
we have $R_w = 0$.
\item
The $K$-grading is pointed if and 
only if there is a homomorphism 
$\kappa \colon K \to \ZZ$ defining a 
pointed $\ZZ$-grading with effective cone 
$\QQ_{\ge 0}$.
\end{enumerate}
\end{lemma}

\begin{proof}
We prove~(i). 
It suffices to show that there is no 
non-zero $w \in K$ with $R_w \ne 0$ 
and $R_{-w} \ne 0$.
Consider $f \in R_w$ and $f' \in R_{-w}$,
both being non-zero.
Then $ff'$ is a non-zero element of $R_0$ 
and hence constant. 
Thus, $f$ and $f'$ are both units.
By assumption, we have $w = 0$.

We prove~(ii).
If the $K$-grading is pointed, then we find a 
hyperplane $U \subseteq K_\QQ$ intersecting 
$\Eff(X)$ precisely in the origin.
Let $K_U \subseteq K$ be the subgroup consisting 
of all elements $w \in K$ with $w \otimes 1 \in U$.
Then $K/K_U \cong \ZZ$ holds
and we may assume that the projection 
$\kappa \colon K \to \ZZ$ 
sends the effective cone to the positive ray.
Using~(i), we see that for the induced $\ZZ$-grading 
all homogeneous elements of degree zero
are constant.
The reverse implication is clear according to~(i).
\end{proof}

Let $R$ be a $K$-integral algebra.
A homogeneous non-zero non-unit $f \in R$
is \emph{$K$-irreducible}, if admits no 
decomposition $f = f'f''$ with homogeneous non-zero 
non-units $f',f'' \in R$.
A homogeneous non-zero non-unit $f \in R$ 
is \emph{$K$-prime}, if for any two 
homogeneous $f',f'' \in R$ we have 
that $f \mid f'f''$ implies $f \mid f'$ 
or $f \mid f''$.
Every $K$-prime element is $K$-irreducible.
The algebra $R$ is called  \emph{$K$-factorial},
or the $K$-grading just \emph{factorial},
if $R$ is $K$-integral and every homogeneous 
non-zero non-unit is a product of $K$-primes.
In a $K$-factorial algebra, the $K$-prime 
elements are exactly the $K$-irreducible ones.

An ideal $\mathfrak{a} \subseteq R$ is \emph{homogeneous}
if it is generated by homogeneous elements.
Moreover, an ideal $\mathfrak{a} \subseteq R$ is 
\emph{$K$-prime} if for any two homogeneous 
$f,f' \in R$ we have that 
$ff' \in \mathfrak{a}$ implies $f \in \mathfrak{a}$
or $f' \in \mathfrak{a}$.
A homogeneous ideal $\mathfrak{a} \subseteq R$ is 
$K$-prime if and only if $R /\mathfrak{a}$ is 
$K$-integral.
We say that homogeneous elements
$g_1,\ldots,g_s \in R$ 
\emph{minimally generate} the 
$K$-homogeneous ideal $\mathfrak{a} \subseteq R$
if they generate $\mathfrak{a}$ and no proper 
subcollection of $g_1,\ldots,g_s$ does so.

\begin{lemma}
\label{lem:minimal2prime}
Let $R$ be a $K$-graded algebra
such that the grading is pointed, factorial and 
every homogeneous unit is of degree zero.
If $g_1,\ldots,g_s \in R$ minimally
generate a $K$-prime ideal of $R$,
then each $g_i$ is a $K$-prime element of $R$.
\end{lemma}

\begin{proof}
Assume that $g_1$ is not $K$-prime.
Then $g_1$ is not $K$-irreducible and we 
can write $g_1 = g_1'g_1''$ with homogeneous 
non-zero non-units $g_1', g_1'' \in R$.
As the ideal 
$\bangle{g_1,\ldots,g_s} \subseteq R$
is $K$-prime, it contains one of 
$g_1'$ and $g_1''$, say $g_1'$. 
That means that 
$$ 
g_1' \ = \ h_1g_1 + \ldots + h_sg_s
$$
holds with homogeneous elements $h_i \in R$.
Take a coarsening $K \to \ZZ$ of 
the $K$-grading as provided by 
Lemma~\ref{lem:coarshom}~(ii).
Then the above representation of $g_1'$ 
yields 
$$ 
\deg_\ZZ(g_1')
\ = \ 
\deg_\ZZ(h_1) + \deg_\ZZ(g_1) 
\ = \  
\ldots 
\ = \  
\deg_\ZZ(h_s) + \deg_\ZZ(g_s).
$$
Consequently, $\deg_\ZZ(g_1') \ge \deg_\ZZ(g_1)$ or
$h_1 = 0$.
Since the $\ZZ$-grading of $R$ is pointed, we have 
$\deg_\ZZ(g_1') < \deg_\ZZ(g_1') + \deg_\ZZ(g_1'') 
= \deg_\ZZ(g_1)$. 
Thus, $h_1 = 0$ holds.
This implies $g_1 = g_1'g_1'' \in \bangle{g_2,\ldots,g_s}$.
A contradiction.
\end{proof}

Given a finitely generated abelian group $K$
and $w_1, \ldots, w_r \in K$, there is a unique
$K$-grading on the polynomial algebra 
$\KK[T_1,\ldots, T_r]$ satisfying 
$\deg(T_i) = w_i$ for $i = 1, \ldots, r$.
We call such grading a \emph{linear} grading 
of $\KK[T_1,\ldots, T_r]$.

\goodbreak

\begin{lemma} 
\label{lem:degrel}
Consider a linear $K$-grading on 
$\KK[T_1,\ldots, T_r]$ and a 
$K$-homogeneous $g \in \KK[T_1,\ldots, T_r]$. 
Moreover, let $1 \le i_1, \ldots, i_q \le r$ 
be pairwise distinct.
Assume that $T_{i_1}$ is not a monomial of~$g$
and that $g, T_{i_2}, \ldots, T_{i_q}$ 
minimally generate a $K$-prime ideal in 
$\KK[T_1,\ldots,T_r]$.
Then we have a presentation
$$
\qquad 
\deg(g) 
\ = \ 
\sum
a_j \deg(T_j),
\quad
j \ne i_1, \ldots, i_q,
\
a_j \in \ZZ_{\ge 0}.
$$
\end{lemma}

\begin{proof}
Suppose that $\deg(g)$ allows no
representation as a positive 
combination over the $\deg(T_j)$ with 
$j \not\in \{i_1,\ldots, i_q\}$.
Then each monomial of $g$ must 
have a factor $T_{i_j}$ for some
$j = 1, \ldots, q$. 
Write 
$$ 
g
\ = \ 
g_1 T_{i_1} + g_2 T_{i_2} + \ldots + g_q T_{i_q}
\ = \ 
g_1 T_{i_1} + h
$$
with polynomials $g_j \in \KK[T_1,\ldots, T_r]$
such that $g_1$ depends on none of 
$T_{i_2}, \ldots, T_{i_q}$.
By assumption, $g_1 T_{i_1}$ is non-zero and 
we have a $K$-integral factor ring
$$ 
\KK[T_1, \ldots, T_r] / \bangle{g, T_{i_2}, \ldots, T_{i_q}} 
\ \cong \ 
\KK[T_j; \ j \ne i_2, \ldots, i_r] / \bangle{g_1 T_{i_1}}.
$$
Consequently, $g_1 T_{i_1}$ is a $K$-prime polynomial.
This implies~$g_1 = c \in \KK^*$ and thus we arrive at 
$g = cT_{i_1} + h$; a contradiction to the assumption 
that  $T_{i_1}$ is not a monomial of $g$.
\end{proof}

If $R$ is a finitely generated $K$-graded algebra,
then $R$ admits homogeneous generators 
$f_1,\ldots, f_r$.
Turning the polynomial ring $\KK[T_1, \ldots, T_r]$ 
into a $K$-graded algebra via $\deg(T_i) := \deg(f_i)$,
we obtain an epimorphism of $K$-graded algebras: 
$$
\pi \colon 
\KK[T_1, \ldots, T_r] \ \to \ R,
\qquad
T_i \ \mapsto \ f_i.
$$
Together with a choice of $K$-homogeneous generators 
$g_1, \ldots, g_s$ for the ideal $\ker(\pi)$,
we arrive at \emph{$K$-graded presentation} 
of $R$ by homogeneous generators and relations:
$$ 
R
\ = \ 
\KK[T_1, \ldots, T_r] / \bangle{g_1, \ldots, g_s}.
$$
We call such presentation \emph{irredundant}
if $\ker(\pi)$ contains no
elements of the form $T_i - h_i$ with 
$h_i \in \KK[T_1,\ldots, T_r]$ not depending on $T_i$.

\begin{proposition}
\label{prop:hypmov}
Let $R$ a finitely generated $K$-graded algebra 
such that the grading is pointed, factorial
and every homogeneous unit is of degree zero.
Let 
$$ 
R 
\ = \ 
\KK[T_1, \ldots, T_r] / \bangle{g_1,\ldots,g_s}
$$
be an irredundant $K$-graded presentation
with $\dim(R) = r-s$ such that
$T_1, \ldots, T_r$ define $K$-prime elements 
in $R$. 
Then, for every $l = 1,\ldots, s$, we have 
$$ 
\deg(g_l) 
\ \in \ 
\bigcap_{1 \le i < j \le r}
\cone(\deg(T_k); \ k \ne i, \ k \ne j)
\ \subseteq \
K_\QQ.
$$
\end{proposition}

\begin{proof}
It suffices to show that for any two 
$1 \le i < j \le r$, we can 
represent each $\deg(g_l)$ as a positive 
combination over the $\deg(T_k)$,
where $k \ne i,j$.
For $l = 1, \ldots, s$, set
$$ 
g_{l,j} 
\ := \ 
g_l(T_1, \ldots, T_{j-1},0,T_{j+1}, \ldots, T_r)
\ \in \ 
\KK[T_1, \ldots, T_r].
$$ 
Since $T_j$ defines a $K$-prime element in $R$,
the ideal $\bangle{T_j} \subseteq R$ is $K$-prime
and $\bangle{T_j}$ lifts to a $K$-prime ideal 
$$ 
I_j
\ := \ 
\bangle{g_1,\ldots,g_s,T_j}
\ = \ 
\bangle{g_{1,j},\ldots,g_{s,j},T_j}
\ \subseteq \
\KK[T_1, \ldots, T_r].
$$
Then $\KK[T_1, \ldots, T_r]/I_j$ is isomorphic 
to $R/\bangle{T_j}$. 
The latter algebra is of dimension $r-s-1$ 
due to our assumptions.
Thus, $g_{1,j},\ldots,g_{s,j},T_j$
minimally generate $I_j$. 
By Lemma~\ref{lem:minimal2prime}, 
each $g_{l,j}$ is $K$-prime and hence
defines a $K$-integral factor algebra
$$ 
\KK[T_m; \ m \ne j] / \bangle{g_{l,j}}
\ \cong \ 
\KK[T_1, \ldots, T_r] / \bangle{g_l,T_j}.
$$
We conclude that $g_l,T_j$ minimally 
generate a $K$-prime ideal in 
$\KK[T_1,\ldots,T_r]$.
Thus, we may apply Lemma~\ref{lem:degrel}
and obtain the assertion.
\end{proof}

We turn to the geometric point of view.
So, $\KK$ is now algebraically closed of 
characteristic zero and $R$ an affine 
$K$-graded algebra, where affine means 
that $R$ is finitely generated over $\KK$ 
and has no nilpotent elements.
Then we have the affine 
variety $\bar{X}$ with $R$ as its algebra 
of global functions and the quasitorus $H$ 
with $K$ as its character group:
$$ 
\bar{X} \ = \ \Spec \, R,
\qquad\qquad 
H \ = \ \Spec \, \KK[K].
$$ 
The $K$-grading of $R$ defines an action of 
$H$ on $\bar{X}$, which is uniquely 
determined by the property that each  
$f \in R_w$ satisfies 
$f(h \cdot x)  = \chi^w(h) f(x)$
for all $x \in \bar{X}$ and $h \in H$,
where $\chi^w$ is the character corresponding 
to $w \in K$.
We take a look at the geometric invariant 
theory of the $H$-action on $\bar{X}$;
see~\cite{BeHa1,ADHL}.
The \emph{orbit cone} $\omega_x \subseteq K_\QQ$ 
associated with $x \in \bar{X}$ and the 
\emph{GIT-cone} $\lambda_w \subseteq K_\QQ$
associated with $w \in \Eff(R)$ are defined as  
$$ 
\omega_x \ = \ \cone(w \in K; \ f(x) \ne 0 \text{ for some } f \in R_w),
\qquad
\lambda_w \ := \ \bigcap_{x \in \bar{X}, w \in \omega_x} \omega_x.
$$
Orbit cones as well as GIT-cones are convex 
polyhedral cones and there are only finitely many of them.
The basic observation is that the GIT-cones form 
a fan $\Lambda(R)$ in $K_\QQ$, the \emph{GIT-fan}, having 
the effective cone $\Eff(R)$ as its support.

\begin{remark}
\label{rem:barXfaces} 
Let $K$ be a finitely generated abelian group
and $R$ a $K$-integral affine algebra.
Fix a $K$-graded presentation
$$ 
R 
\ = \ 
\KK[T_1, \ldots, T_r] / \bangle{g_1,\ldots,g_s}.
$$
This yields an $H$-equivariant closed embedding
$ 
\bar{X} 
= 
V(g_1,\ldots,g_s) 
\subseteq 
\KK^r
$
of affine varieties.
Moreover, we have a homomorphism
$$ 
Q \colon 
\ZZ^r \ \to \ K, 
\qquad 
\nu \ \mapsto \ \nu_1 \deg(T_1) + \ldots +  \nu_r \deg(T_r).
$$
An \emph{$\bar{X}$-face} is a face $\gamma_0 \preceq \gamma$ 
of the orthant $\gamma := \QQ_{\ge 0}^r$ admitting
a point $x \in \bar{X}$ such that one has 
$$ 
x_i \ne 0 \ \iff \ e_i \in \gamma_0
$$
for the coordinates $x_1, \ldots, x_r$ of $x$ 
and the canonical basis vectors $e_1, \ldots, e_r \in \ZZ^r$.
Write $\mathfrak{S}(\bar{X})$ for the set of all
$\bar{X}$-faces of $\gamma \subseteq \QQ^r$.
Then we have 
$$
\{Q(\gamma_0); \ \gamma_0 \in \mathfrak{S}(\bar{X})\}
\ =  \
\{\omega_x; \ x \in \bar{X}\}.
$$
That means that the projected $\bar{X}$-faces 
are exactly the orbit cones.
The $\bar{X}$-faces define a 
decomposition into locally closed 
subsets
$$ 
\bar{X}
\ = \ 
\bigcup_{\gamma_0 \in \mathfrak{S}(X)}
\bar{X}(\gamma_0),
\qquad
\bar{X}(\gamma_0) 
\ := \ 
\{x \in \bar{X}; \ x_i \ne 0  \Leftrightarrow  e_i \in \gamma_0\}
\ \subseteq \
\bar{X}.
$$
\end{remark}

\begin{definition}
\label{def:gI}
Let $I = \{i_1, \ldots, i_k\}$ be a subset of 
$\{1, \ldots, r\}$. 
Then the face $\gamma_I$ of the orthant 
$\gamma = \QQ_{\ge 0}^r$ associated with $I$ 
is defined as 
$$ 
\gamma_I 
\ := \ 
\gamma_{i_1, \ldots, i_k}
\ := \ 
\cone(e_{i_1}, \ldots, e_{i_k}).
$$
Moreover, for a polynomial $g \in \KK[T_1,\ldots,T_r]$,
the polynomial $g_I$ associated with $I$ is 
defined as 
$$
g_I \ := \ g(\tilde T_1,\ldots, \tilde T_r),
\qquad
\tilde T_i 
\ :=  \
\begin{cases}
T_i, & i \in I,
\\
0,  & i \not\in I.
\end{cases}
$$
\end{definition}

\begin{remark}
\label{rem:barXfacecrit}
In the setting of Remark~\ref{rem:barXfaces},
let $I = \{i_1, \ldots, i_k\}$ be a subset of 
$\{1, \ldots, r\}$.
\begin{enumerate}
\item
$\gamma_I$ is an $\bar X$-face if and only if
$\bangle{g_{1,I}, \ldots, g_{s,I}}$ 
contains no monomial.
\item
If $\deg(g_j) \not\in \cone(w_i; \ i \in I)$ 
holds for $j = 1, \ldots, s$, then 
$\gamma_I$ is an $\bar X$-face.
\item 
If $(w_i; \ i \in I)$ is linearly independent
in $K$, then $\gamma_I$ is an $\bar X$-face if 
and only if none of $g_{1}, \ldots, g_{s}$
has a monomial $T_{i_1}^{l_1} \cdots T_{i_k}^{l_k}$
with $l_1, \ldots, l_k \in \ZZ_{\ge 0}$.
\end{enumerate}
\end{remark}

\begin{proposition}
\label{prop:degincone}
Let $K$ be a finitely generated abelian 
group and $R$ an affine algebra
with a pointed $K$-grading.
Consider a $K$-graded presentation
$$ 
R 
\ = \ 
\KK[T_1, \ldots, T_r] / \bangle{g_1,\ldots,g_s}
$$
such that $T_1,\ldots,T_r$ define non-constant
elements in $R$.
Assume that there are a GIT-cone $\lambda \in \Lambda(R)$ 
of dimension at least two and an index~$i$ with 
$\deg(T_i) \in \lambda^\circ$.
\begin{enumerate}
\item 
There exists a $j$ such that $g_j$ has a 
monomial $T_i^{l_i}$ with $l_i \in \ZZ_{\ge 0}$.
\item 
There exists a $j$ such that $\deg(g_j) = l_i \deg(T_i)$ 
holds with $l_i \in \ZZ_{\ge 0}$.
\item
If $s=1$ holds, then, $\deg(T_k)$ generates a
ray of $\Lambda(R)$
whenever $k \ne i$. 
\end{enumerate}
\end{proposition}

\begin{proof}
Because of $\deg(T_i) \in \lambda^\circ$, 
the ray $\tau$ generated by $\deg(T_i)$
is not an orbit cone.
Thus, $\QQ_{\ge 0} e_i$ is not an $\bar{X}$-face.
This means that some~$g_j$
has a monomial $T_i^{l_i}$, 
which in particular proves~(i) and~(ii).
To obtain~(iii), first observe that 
$\deg(T_k) \in K_\QQ$ is non-zero and 
thus lies in the relative interior of 
some GIT-cone 
$\varrho \in \Lambda(R)$
of positive dimension.  
Suppose that $\varrho$ is not a ray.
Then~(i) yields that besides $T_i^{l_i}$ 
also $T_k^{l_k}$ is a monomial of the 
relation $g_1$.
We conclude that $\gamma_{i,k}$ is an 
$\bar{X}$-face.
Thus, $\deg(T_i)$ and $\deg(T_k)$ lie on a ray of 
$\Lambda(R)$.
A contradiction.
\end{proof}

\section{Mori dream spaces}

Mori dream spaces, introduced in~\cite{HuKe},
behave optimally with respect to the minimal 
model programme and are characterized as the
normal projective varieties with finitely 
generated Cox ring.
Well known example classes are 
the projective toric or spherical varieties
and, most important for the present article, 
the smooth Fano varieties.
In this section, we provide a brief summary 
of the combinatorial approach~\cite{BeHa2,Ha,ADHL} 
to Mori dream spaces, adapted to 
our needs.
Moreover, as a new observation, we present 
Proposition~\ref{prop:qsmooth2degrees},
locating the relation degrees of a Cox ring 
inside the effective cone of a quasismooth 
Mori dream space.

Let $\KK$ be an algebraically closed 
field of characteristic zero,
$R$ be a $K$-graded affine $\KK$-algebra
and consider the action of 
$H = \Spec \, \KK[K]$ on
variety $\bar X  = \Spec \, R$.
Mori dream spaces are obtained as quotients
of the $H$-action. 
We briefly recall the general framework.
Each cone $\lambda \in \Lambda(R)$ 
of the GIT-fan defines 
an $H$-invariant open set of 
\emph{semistable points}
and a \emph{good quotient}:
$$ 
\bar{X}^{ss}(\lambda) 
\ =  \ 
\{x \in \bar{X}; \ \lambda \subseteq \omega_x\}
\ \subseteq \ 
\bar{X},
\qquad\qquad
\bar{X}^{ss}(\lambda) \ \to \  \bar{X}^{ss}(\lambda) \quot H,  
$$
where $\omega_x \subseteq K_\QQ$ denotes the 
orbit cone of $x \in \bar X$.
Each of the quotient varieties $ \bar{X}^{ss}(\lambda) \quot H$ 
is projective over 
$\Spec \, R_0$ and whenever $\lambda' \subseteq \lambda$ 
holds for two GIT-cones, then we have 
$\bar{X}^{ss}(\lambda) \subseteq \bar{X}^{ss}(\lambda')$
and thus an induced projective morphism 
$\bar{X}^{ss}(\lambda) \quot H \to  \bar{X}^{ss}(\lambda') \quot H$
of the quotient spaces.

The $K$-grading of $R$ is \emph{almost free} if
the (open) set $\bar{X}_0 \subseteq \bar{X}$ 
of points $x \in \bar{X}$ with trivial isotropy 
group $H_x \subseteq H$ has complement of 
codimension at least two in~$\bar{X}$.
Moreover, the \emph{moving cone} of $R$ 
is the convex cone $\Mov(R) \subseteq K_\QQ$ 
obtained as the union over all 
$\lambda \in \Lambda(R)$, where 
$\bar{X}^{ss}(\lambda)$ has a complement of 
codimension at least two in $\bar{X}$.

\begin{remark}
\label{rem:EffRMovR}
Let $R$ be a $K$-graded affine algebra such 
that the grading is factorial and any homogeneous 
unit is constant.
Then $R$ admits a system $f_1, \ldots, f_r$ 
of pairwise non-associated $K$-prime generators.
Moreover, if $f_1, \ldots, f_r$ is such 
a system of generators for $R$, then the following 
holds.
\begin{enumerate}
\item
The $K$-grading is almost free if and only if 
any $r-1$ of $\deg(f_1), \ldots, \deg(f_r)$ 
generate $K$ as a group.
\item
If the $K$-grading is almost free,
then the orbit cones $\omega_x$, 
where $x \in \bar{X}$,
and the moving cone are given by 
\begin{eqnarray*} 
\omega_x & = & \cone(\deg(f_i); \ f_i(x) \ne 0),
\\
\Mov(R) 
& = &
\bigcap_{i=1}^r
\cone(\deg(f_j); \ j \ne i).
\end{eqnarray*}
\end{enumerate}
\end{remark}

We say that a $K$-graded affine $\KK$-algebra~$R$
is an \emph{abstract Cox ring} if it is 
integral, normal, has only constant homogeneous 
units, the $K$-grading is almost free, 
pointed, factorial and the moving cone $\Mov(R)$ 
is of full dimension in $K_\QQ$.

\begin{construction} 
\label{constr:mds}
Let $R$ be an abstract Cox ring and 
consider the action of the quasitorus 
$H = \Spec \, \KK[K]$ on the affine 
variety $\bar X  = \Spec \, R$.
For every GIT-cone $\lambda \in \Lambda(R)$ 
with
$\lambda^\circ \subseteq \Mov(R)^\circ$, 
we set
$$ 
X(\lambda) 
\ := \ 
\bar X^{ss}(\lambda) \quot H .
$$
\end{construction}

The following proposition tells us
in particular that Construction~\ref{constr:mds} 
delivers Mori dream spaces;
see~\cite[Thm.~3.2.14, Prop.~3.3.2.9 and Rem.~3.3.4.2]{ADHL}.

\begin{proposition}
\label{prop:mdsprops}
Let $X = X(\lambda)$ arise from
Construction~\ref{constr:mds}.
Then $X$ is normal, projective 
and of dimension $\dim(R)-\dim(K_\QQ)$.
The divisor class group and 
the Cox ring of $X$ are given as
$$ 
\Cl(X) \ = \ K,
\qquad\qquad
\mathcal{R}(X) 
\ = \ 
\bigoplus_{\Cl(X)} \Gamma(X,\mathcal{O}_X(D))
\ = \ 
\bigoplus_{K} R_w 
\ = \ 
R.
$$
Moreover, the cones of effective, 
movable, semiample and ample 
divisor classes of~$X$ are given 
in $\Cl_\QQ(X) = K_\QQ$ as 
$$ 
\Eff(X) \ = \ \Eff(R), 
\qquad
\Mov(X) \ = \  \Mov(R), 
$$
$$
\SAmple(X) = \lambda,
\qquad
\Ample(X) = \lambda^\circ.
$$
\end{proposition}

By~\cite[Cor.~3.2.1.11]{ADHL},
all Mori dream space arise from 
Construction~\ref{constr:mds}.
For the subsequent work, we have to get 
more concrete, meaning that we will 
work in terms of generators and relations.

\begin{construction} 
\label{rem:Coxconst}
Let $R$ be an abstract Cox ring 
and $X = X(\lambda)$ be as in
Construction~\ref{constr:mds}.
Fix a $K$-graded presentation
$$ 
R 
\ = \ 
\KK[T_1, \ldots, T_r] / \bangle{g_1,\ldots,g_s}
$$
such that the variables $T_1, \ldots, T_r$ 
define pairwise non-associated $K$-primes
in $R$. 
Consider the orthant $\gamma = \QQ_{\ge 0}^r$
and the projection
$$ 
Q \colon \ZZ^r \ \to \ K,
\qquad
e_i \ \mapsto \ w_i := \deg(T_i).
$$
An \emph{$X$-face} is 
an $\bar{X}$-face $\gamma_0 \preceq \gamma$ 
with $\lambda^{\circ} \subseteq Q(\gamma_0)^{\circ}$.
Let $\rlv(X)$ be the set of all $X$-faces 
and $\pi \colon \bar{X}^{ss}(\lambda) \to X$ 
the quotient map.
Then we have a decomposition
$$ 
X 
\ = \ 
\bigcup_{\gamma_0 \in \rlv(X)}
X(\gamma_0)
$$
into pairwise disjoint locally closed sets
$X(\gamma_0) := \pi(\bar{X}(\gamma_0))$,
which we also call the \emph{pieces} of~$X$.
\end{construction}

Recall that~$X$ is 
\emph{$\QQ$-factorial} if for every 
Weil divisor on $X$ some non-zero multiple 
is locally principal.
Moreover, $X$ is \emph{locally factorial} 
if every stalk~$\mathcal{O}_x$, 
where $x \in X$ is a (closed)
point, is a unique factorization domain.
Finally, $X$ is \emph{quasismooth}
if the open set 
$\bar X^{ss}(\lambda) \subseteq \bar X$ 
of semistable points is a smooth variety.

\begin{proposition} 
\label{prop:locprops}
Consider the situation of
Construction~\ref{rem:Coxconst}.
\begin{enumerate}
\item 
The variety $X$ is $\QQ$-factorial, if and 
only if $\dim(\lambda) = \dim(K_\QQ)$ holds
for $\lambda = \SAmple(X)$.
\item
The variety $X$ is locally factorial if and 
only if for every $X$-face 
$\gamma_0 \preceq \gamma$,
the group $K$ is generated by 
$Q(\gamma_0 \cap \ZZ^r)$.
\item 
The variety $X$ is quasismooth if and 
only if every $\bar X(\gamma_0)$ 
consists of smooth points of 
$\bar X$ for every $X$-face
$\gamma_0 \preceq \gamma$.
\item 
The variety $X$ is smooth if and only if $X$ 
is locally factorial and quasismooth.
\end{enumerate}
\end{proposition}

We refer to~\cite[Cor.~1.6.2.6, Cor.~3.3.1.8, Cor.~3.3.1.9]{ADHL}
for the above statements.
Next we describe the impact of quasismoothness
has an impact the position of the relation
degrees.

\begin{proposition}
\label{prop:qsmooth2degrees}
In the situation of
Construction~\ref{rem:Coxconst},
assume $\dim(R) = r -s$ 
and let~$X$ be quasismooth.
Then, for every $j = 1, \ldots, s$,
we have 
\[
\deg(g_j)
\ \in \ 
\bigcap_{\gamma_0 \in \rlv(X)} 
\left( 
Q(\gamma_0 \cap \ZZ^r) 
\cup 
\bigcup_{i=1}^r w_i + Q(\gamma_0 \cap \ZZ^r) 
\right).
\]
\end{proposition}

\begin{proof}
Consider any $X$-face $\gamma_I$, where 
$I \subseteq \{1,\ldots,r\}$,
and choose a point $x \in \bar X(\gamma_I)$.
Then $x_i \ne 0 $ holds if and only if 
$i \in I$.
For any monomial $T^{\nu}$, we have 
$$
\frac{\partial T^{\nu}}{\partial T_k} (x)
 \ne 
0
\ \Rightarrow \
\nu \ \in \ \gamma_I \cup \gamma_I + e_k
\ \Rightarrow \
\deg(T^{\nu}) 
 = 
Q(\nu) 
 \in  
Q(\gamma_I) \cup Q(\gamma_I) + w_k.
$$
Now, since $X$ is quasismooth, we have 
$\grad_{g_j}(x) \ne 0$ for all $j = 1,\ldots, s$.
Thus, every $g_j$ must have a monomial $T^{\nu_j}$ 
with non-vanishing gradient at $x$.
\end{proof}

Finally, in case of a complete intersection
Cox ring, we have an explicit description
of the anticanonical class;
see~\cite[Prop.~3.3.3.2]{ADHL}.

\begin{proposition} 
\label{prop:anticanclass}
In the situation of
Construction~\ref{rem:Coxconst},
assume that $\dim(R) = r -s$ holds.
Then the anticanonical class of $X$ is
given in $K = \Cl(X)$ as 
$$ 
-\KKK_X
\ = \ 
\deg(T_1) +  \ldots + \deg(T_r)
-
\deg(g_1) -  \ldots - \deg(g_s).
$$
\end{proposition}

\section{General hypersurface Cox rings}

First, we make our concept of a general hypersurface 
Cox ring precise.
Then we present the toolbox to be used 
in the proof of Theorem~\ref{thm:candidates}
for verifying that given specifying data,
that means a collection of the generator 
degrees and a relation degree,
allow indeed a smooth general hypersurface 
Cox ring.
We will have to deal with the following 
setting.

\begin{construction}
\label{constr:hypersurf}
Consider a linear, pointed, almost free
$K$-grading on the polynomial ring 
$S := \KK[T_1, \ldots, T_r]$ and the 
quasitorus action 
$H \times \bar Z \to \bar Z$,
where 
$$ 
H 
\ := \ 
\Spec \, \KK[K],
\qquad\qquad
\bar Z 
\ := \ 
\Spec \, S 
\ = \ 
\KK^r.
$$
As earlier, we write $Q \colon \ZZ^r \to K$,
$e_i \mapsto w_i := \deg(T_i)$ for the degree
map.
Assume that $\Mov(S) \subseteq K_\QQ$ is of 
full dimension and fix $\tau \in \Lambda(S)$ 
with $\tau^\circ \subseteq \Mov(S)^\circ$.
Set
$$ 
\hat {Z}
\ := \ 
\bar Z^{ss}(\tau),
\qquad \qquad 
Z
\ := \ 
\hat {Z} \quot H.
$$ 
Then $Z$ is a projective toric variety with 
divisor class group $\Cl(Z) = K$ 
and Cox ring $\mathcal{R}(Z) = S$.
Moreover, fix $0 \ne \mu \in K$, and 
for $g \in S_\mu$ set
$$
R_g  := S / \langle g \rangle,
\qquad
\bar X_g  :=  V(g)  \subseteq  \bar Z,
\qquad
\hat X_g  :=  \bar X_g \cap \hat Z,
\qquad
X_g  :=  \hat X_g \quot H  \subseteq  Z.
$$
Then the factor algebra $R_g$ inherits a 
$K$-grading from $S$ and the quotient  
$X_g \subseteq Z$ is a closed subvariety.
Moreover, we have
$$ 
X_g
\ \subseteq \
Z_g
\ \subseteq \ 
Z
$$
where $Z_g \subseteq Z$ is the 
minimal ambient toric variety
of $X_g$, that means the (unique) minimal 
open toric subvariety containing $X_g$. 
\end{construction}

\begin{remark}
\label{rem:embcox}
In the situation of Construction~\ref{constr:hypersurf},
there is a (unique) GIT-cone $\lambda \in \Lambda(R_g)$
such that we have  
$$
\hat X_g \ = \ \bar X_g^{ss}(\lambda),
\qquad\qquad
X_g \ = \ \bar X_g^{ss}(\lambda) \quot H.
$$
Thus, if $R_g$ is an abstract Cox ring 
and $T_1, \ldots, T_r$ define pairwise
non-associated $K$-primes in $R_g$, 
then $X_g$ is as in 
Construction~\ref{rem:Coxconst}. 
In particular
$$\Cl(X) \ = \ K,
\qquad\qquad
\RRR(X_g) \ = \ R_g
$$
hold for the  divisor class group 
and the Cox ring of $X_g$.
Moreover, in $K_\QQ$ we have the 
following
\[
\tau^\circ 
\ = \ 
\Ample(Z) 
\ \subseteq \ 
\Ample(Z_g) 
\ = \ 
\Ample(X_g)
\ = \ 
\lambda^\circ.
 \]
\end{remark}

We are ready to formulate the precise definitions for our
notions around hypersurface Cox rings.

\begin{definition}
Consider the situation of Construction~\ref{constr:hypersurf}.
\begin{enumerate}
\item
We call $R_g$ a \emph{hypersurface Cox ring}
if $T_1,\ldots, T_r$ define a minimal system
of $K$-homogeneous generators for $R_g$.
\item
We say that $R_g$ is \emph{spread} if
every monomial $T^\nu \in \KK[T_1,\ldots, T_r]$
of degree $\mu = \deg(g) \in K$ is a convex
combination of monomials of $g$.
\end{enumerate}
\end{definition}

Here, we tacitly identify a monomial
$T^\nu = T_1^{\nu_1} \cdots T_r^{\nu_r}$
with its exponent vector
$\nu = (\nu_1, \ldots, \nu_r) \in \QQ^r$
when we speak about convex combinations
of monomials.

\begin{remark}
In the setting of Construction~\ref{constr:hypersurf},
assume that $R_g$ is a hypersurface Cox ring.  
\begin{enumerate}
\item
Since $T_1,\ldots, T_r$ define a minimal system
of $K$-homogeneous generators, $R_g$ is not a
polynomial ring.
\item
As the $K$-grading is pointed, the $T_i$ 
define pairwise non-associated $K$-prime
elements in $R_g$.
\item
$R_g$ is spread if and only if the Newton polytope
of $g$ equals the convex hull over all
monomials of degree $\mu = \deg(g) \in K$.
\end{enumerate}
\end{remark}

\begin{definition}
\label{def:genhypcr}
Consider the situation of
Construction~\ref{constr:hypersurf}
and denote by
$S_\mu \subseteq S = \KK[T_1,\ldots, T_r]$
the homogeneous component
of degree $\mu \in K$.
\begin{enumerate}
\item
A \emph{general hypersurface Cox ring}
is a family $R_g$, where $g \in U$
with a non-empty open $U \subseteq S_\mu$,
such that each $R_g$ is a hypersurface Cox ring.
\item
We say that a general hypersurface Cox ring  
$R_g$ is \emph{spread} if each $R_g$, where 
$g \in U$, is spread.
\item
We say that a general hypersurface Cox ring  
$R_g$ is \emph{smooth (Fano)} if
for some $\tau \in \Lambda(S)$ all
the resulting $X_g$, where $g \in U$,
are smooth (Fano).
\end{enumerate}
\end{definition}

We turn to the toolbox for verifying 
that given specifying data $w_1, \ldots, w_r \in K$ 
and $\mu \in K$ as in Construction~\ref{constr:hypersurf}
lead to a smooth Fano general hypersurface Cox ring~$R_g$ 
in the above sense.

\begin{remark}
\label{rem:genCRFano}
In the notation of Construction~\ref{constr:hypersurf},
a general hypersurface Cox ring~$R_g$ is Fano
if and only if the generator and relation
degrees satisfy
$$ 
- \mathcal{K}
\ = \
w_1 + \ldots + w_r - \mu
\ \in \ 
\Mov(R_g)^\circ.
$$  
In this case, the unique cone $\tau \in \Lambda(S)$
with $-\mathcal{K} \in \tau^\circ$ defines Fano 
varieties $X_g$ for all $g \in U$; 
see Proposition~\ref{prop:anticanclass} 
and Remark~\ref{rem:embcox}.
\end{remark}

In the notation of Construction~\ref{constr:hypersurf},
we denote by $U_\mu \subseteq S_\mu$ the non-empty
open set of polynomials $f \in S$ of degree
$\mu \in K$ such that each monomial of $S_\mu$
is a convex combination of monomials of $f$.

\begin{remark}
If $R_g$, where $g \in U$, is a general hypersurface
Cox ring, then $R_g$, where $g \in U \cap U_\mu$
is spread general hypersurface Cox ring.
In particular, we can always assume a general
hypersurface Cox ring to be spread.
\end{remark}

\begin{remark}
\label{rem:genCRirredundant}
In the situation of Construction~\ref{constr:hypersurf},
consider the rings $R_g$ for $g \in U_\mu$.
Then the following statements are equivalent.
\begin{enumerate}
\item
The variables $T_1, \ldots, T_r$ form a minimal 
system of generators for all $R_g$, where $g \in U_\mu$.
\item
The variables $T_1, \ldots, T_r$ form a minimal 
system of generators for one $R_g$ with $g \in U_\mu$.
\item
We have $\mu \ne w_i$ for $i = 1, \ldots, r$.
\item
The polynomial $g \in U_\mu$ is not of the form
$g = T_i + h$ with $h \in S_\mu$ not depending on $T_i$.
\end{enumerate}
\end{remark}

\begin{lemma} 
\label{lem:genirred}
Consider a linear, pointed $K$-grading
on $S := \KK[T_1, \ldots, T_r]$. 
Then, for any $0 \neq \mu \in K$ the irreducible
polynomials $g \in S_\mu$ form an open subset of $S_\mu$.
\end{lemma}

\begin{proof}
Lemma~\ref{lem:coarshom}~(ii) provides us with 
a coarsening homomorphism $\kappa \colon K \to \ZZ$ 
that turns $S$ into a pointed $\ZZ$-graded algebra.
Then $S_\mu$ is a vector subspace of the 
(finite dimensional) vector space $S_{\kappa(\mu)}$ 
of $\kappa(\mu)$-homogeneous polynomials
and we may assume $K = \ZZ$ for the proof.
Since the $K$-grading of $S$ is pointed, we 
have $S^* = S_0 \setminus \{0\}$.
Thus, a polynomial $g \in S_\mu$ is reducible 
if and only if it is a product of homogeneous 
polynomials of non-zero $K$-degree. 

Now, let $u, v \in \ZZ$ with $u + v = \mu$ and 
$S_u \ne \{0\} \ne S_v$.
Then the set of $\mu$-homogeneous polynomials $g$ 
admitting a factorization 
$g = fh$ with $f \in S_u$, $h \in S_u$
is exactly the affine cone over the image
of the projectivized multiplication map
$$
\PP(S_u) \times \PP(S_v) \ \to \ \PP(S_\mu), 
\qquad \qquad
([f],[h]) \ \mapsto \ [fh]
$$
and thus is a closed subset of $S_\mu$.
As there are only finitely many such
presentations $u + v = \mu$,
the reducible $g \in S_\mu$ 
form a closed subset of $S_\mu$.
\end{proof}

\begin{proposition} 
\label{prop:varclassprime}
Consider the setting of Construction~\ref{constr:hypersurf}.
For $1 \le i \le r$ denote by $U_i \subseteq S_\mu$ 
the set of all $g \in S_\mu$ such that $g$ is prime 
in $S$ and $T_i$ is prime in~$R_g$.
Then $U_i \subseteq S_\mu$ is open.
Moreover, $U_i$ is non-empty if and only if 
there is a $\mu$-homogeneous prime polynomial 
not depending on $T_i$.
\end{proposition}

\begin{proof}
By Lemma~\ref{lem:genirred}, the $g \in S_\mu$ being 
prime in $S$ form an open subset $U \subseteq S_\mu$.
For any $g \in U$, the variable $T_i$ defines a prime
in $R_g$ if and only if the polynomial
$g_i := g(T_1, \ldots, T_{i-1}, 0, T_{i+1}, \ldots, T_n)$
is prime in $\KK[T_j; \ j \ne i]$.
Thus, using again Lemma~\ref{lem:genirred}, we 
see that the $g \in U$ with $T_i \in R_g$ prime 
form the desired open subset $U_i \subseteq U$. 
The supplement is clear.
\end{proof}

Checking the normality and $K$-factoriality
of $R_g$ amounts, in our situation, to 
proving factoriality.
We will use Dolgachev's criterion,
see~\cite[Thm.~1.2]{Do80} and~\cite{Do81},
which tells us that a polynomial 
$g = \sum a_\nu T^\nu$ in $\KK[T_1, \ldots, T_r]$
defines a unique factorization domain 
if the Newton polytope $\Delta \subseteq \QQ^r$ 
of $g$ satisfies the following conditions:
\begin{enumerate}
\item 
$\dim(\Delta) \geq 4$,
\item 
each coordinate hyperplane of $\QQ^r$ 
intersects $\Delta$ non-trivially,
\item 
the dual cone of $\cone(\Delta - u; \ u \in \Delta_0)$ 
is regular for each one-dimensional face 
$\Delta_0 \preceq \Delta$,
\item
for each face $\Delta_0 \preceq \Delta$ the zero locus of
$\sum_{\nu \in \Delta_0} a_\nu T^\nu$ is smooth
along the torus $\mathbb{T}^r = (\mathbb{K}^*)^r$.
\end{enumerate}
We will call for short a convex polytope 
$\Delta \subseteq \QQ_{\geq 0}^r$ with
properties (i)--(iii) from above a \emph{Dolgachev polytope}.

\begin{proposition} 
\label{prop:ufdcrit}
In the situation of Construction~\ref{constr:hypersurf},
suppose that one of the following conditions 
is fulfilled:
\begin{enumerate}
\item 
$K$ is of rank at most $r-4$ and torsion free,
there is a $g \in S_\mu$ such that
$T_1, \ldots, T_r$ define primes
in $R_g$, we have $\mu \in \tau^\circ$ 
and $\mu$ is base point free on~$Z$.
\item
The set 
$\conv(\nu \in \ZZ_{\geq 0}^r; \ Q(\nu) = \mu)$
is a Dolgachev polytope.
\end{enumerate}
Then there is a non-empty open subset of polynomials 
$g \in S_\mu$ such that the ring $R_g$ is factorial.
\end{proposition}

\begin{proof}
Assume that~(i) is satisfied.
If $\mu = \deg(T_i)$ holds for some~$i$,
then, as the grading is pointed,
we have a non-empty open set of polynomials
$g = T_i + h$ in $S_\mu$ with $h$
not depending
on $T_i$.
The corresponding $R_g$ are all factorial.
Now assume $\mu \ne \deg(T_i)$
for all $i$.
By Proposition~\ref{prop:varclassprime}, 
the set $U \subseteq S_\mu$ of all prime 
$g \in S_\mu$ such that $T_1, \ldots, T_r$ 
define primes in $R_g$ is open and, 
by assumption, $U \subseteq S_\mu$ 
is non-empty.
Remark~\ref{rem:genCRirredundant}
yields that $T_1, \ldots, T_r$ form a 
minimal system of generators for~$R_g$.
We conclude that for all $f \in U$,
the complement of $\hat X_g$ in $\bar X_g$
is of codimension at least two. 
Since $\mu$ is base point free and ample on $Z$,
we can apply~\cite[Cor.~2.3]{ArLa}, 
telling us that after suitably shrinking,
$U$ is still non-empty and~$R_g$ is the 
Cox ring of $X_g$ for all $g \in U$.
In particular, $R_g$ is $K$-factorial. 
Since~$K$ is torsion free,
$R_g$ is a unique factorization domain.

Assume that~(ii) holds.
As $\Delta := \conv(\nu \in \ZZ_{\geq 0}^r; \ Q(\nu) = \mu)$
is a Dolgachev polytope,
we infer from~\cite[§2, Thm.~2]{Ho} that 
there is a non-empty open subset 
of polynomials $g \in S_\mu$ with 
Newton polytope $\Delta$ satisfying 
the above conditions~(i) to~(iv).
Thus, Dolgachev's criterion
shows that $R_g$ is a factorial ring.
\end{proof}

\begin{proposition} 
\label{prop:Xgsmooth}
In the setting of Construction~\ref{constr:hypersurf},
assume that $Z_g$ and $\hat X_g$ both are smooth.
Then $X_g$ is smooth.
\end{proposition}

\begin{proof}
Consider the quotient map $p \colon \hat{Z} \to Z$.
Since $Z_g$ is smooth, $H$ acts freely on 
$p^{-1}(Z_g)$.
Thus, $X_g$ inherits smoothness from 
$\hat X_g = p^{-1}(X_g)$.
\end{proof}

\begin{lemma}
\label{lem:gensmooth}
Consider a linear, pointed $K$-grading on 
$S := \KK[T_1, \ldots, T_r]$.
Let $\lambda \in \Lambda(S)$ and set
$W := (\KK^r)^{ss}(\lambda)$.
Then, for any $\mu \in K$, the polynomials 
$g \in S_\mu$ such that $\grad(g)$ has no zeroes
in $W$
form an open subset of $S_\mu$.
\end{lemma}

\begin{proof}
Consider the morphism 
$\varphi \colon S_\mu \times W \to \KK^r$
sending $(g,z)$ to $\grad_z(g)$
and the projection 
$\pr_1 \colon S_\mu \times W \to S_\mu$
onto the first factor.
Then our task is to show that 
$S_\mu \setminus \pr_1(\varphi^{-1}(0))$ 
is open in $S_\mu$.
We make use of the action of 
$H = \Spec \, \KK[K]$ on~$W$
given by the $K$-grading 
and the commutative diagram
$$ 
\xymatrix{
S_\mu \times W
\ar[rr]
\ar[dr]_{\pr_1}
&
&
S_\mu \times W \quot H
\ar[dl]^{\pr_1}
\\
&
S_\mu
&
}
$$ 
where the horizontal arrow is the good 
quotient for $H$, acting trivially on 
$S_\mu$ and on $W$ as indicated above. 
Since $\varphi^{-1}(0) \subseteq S_\mu \times W$
is invariant under the $H$-action, the image of 
$\varphi^{-1}(0)$ in $S_\mu \times W \quot H$
is closed.
Since $W \quot H$ is projective, the image 
$\pr_1(\varphi^{-1}(0))$ is closed in~$S_\mu$.
\end{proof}

\begin{proposition} 
\label{prop:Xgqsmooth}
Consider the situation of Construction~\ref{constr:hypersurf}.
Then the polynomials $g \in S_\mu$ 
such that $g \in S$ is prime and $\hat X_g$ 
is smooth form an open subset $U \subseteq S_\mu$.
Moreover, $U$ is non-empty if and only if
there are $g_1, g_2 \in S_\mu$ such that $g_1 \in S$ 
is prime and
$\grad(g_2)$ has no zeroes in $\hat{Z}$.
\end{proposition}

\begin{proof}
By Lemma~\ref{lem:genirred},
the set $V_1$ of all prime polynomials 
of $S_\mu$ is open.
Moreover, by Lemma~\ref{lem:gensmooth},
the set of all polynomials of $S_\mu$
such that $\grad(g)$ has no zeroes 
in~$\hat{Z}$ is open.
The assertion follows from 
$U = V_1 \cap V_2$.
\end{proof}

\begin{corollary}
\label{cor:shmooth2gensmooth}
Let $X$ be a variety with a general hypersurface 
Cox ring $R$.
If $X$ is smooth, then $R$ is a smooth general 
hypersurface Cox ring.
\end{corollary}

\begin{proposition} 
\label{prop:bertini}
Consider the situation of Construction~\ref{constr:hypersurf}.
If $\mu \in \Cl(Z)$ is base point free,
then there is a non-empty open subset of $g \in S_\mu$ such that 
$X_g \cap Z^{\rm reg}$
is smooth.
\end{proposition}

\begin{proof}
Observe that $\PP(S_\mu)$ is the complete linear system
associated with the divisor class $\mu \in \Cl(Z)$.
If $\mu$ is a base point free class on $Z$, we can apply
Bertini's first theorem~\cite[Thm.~4.1]{Kl} stating that
there is a non-empty open subset $U \subseteq S_\mu$ such that
for each $g \in U$ the singular locus of $X_g$
is precisely $X_g \cap Z^{\mathrm{sing}}$.
In particular, $X_g \cap Z^{\mathrm{reg}}$ is smooth for
all $g \in U$.
\end{proof}

\begin{remark}
\label{rem:minamb}
In the situation of Construction~\ref{constr:hypersurf},
let $N(g)$ be the Newton polytope of $g$.
For $I \subseteq \{1, \ldots,r\}$, let 
$\gamma_I \preccurlyeq \gamma$ and 
$g_I \in \KK[T_1,\ldots, T_r]$ be as 
in Definition~\ref{def:gI}.
Then~\cite[Prop.~3.1.1.12]{ADHL} yields
the equivalence of the following statements.
\begin{enumerate}
\item
We have $X_g \cap Z(\gamma_I) \ne \emptyset$.
\item  
We have$\bar X_g \cap \bar Z(\gamma_I) \ne \emptyset$.
\item  
The polynomial $g_I$ is not a monomial.
\item
The number of vertices of $N(g)$ contained
$\gamma_I$ differs from one.  
\end{enumerate}
In particular, for the non-empty open subset 
$U_\mu \subseteq S_\mu$ of polynomials of $S_\mu$ 
of polynomials $f \in S$ of degree
$\mu = \deg(g) \in K$ such that each monomial
of $S_\mu$ is a convex combination of monomials
of $f$, we obtain $Z_g = Z_{g'}$ for
all $g, g' \in U_\mu$.
\end{remark}

\begin{definition}
In the setting of Remark~\ref{rem:minamb},
we call $Z_\mu := Z_g$, where $g \in U_\mu$,
the \emph{$\mu$-minimal ambient toric variety}.
\end{definition}

\begin{corollary} 
\label{cor:rk2bertini}
In the setting of in Construction~\ref{constr:hypersurf},
assume $\rk(K) = 2$ and that  $Z_{\mu} \subseteq Z$ is smooth.
If $\mu \in \tau$ holds, then $\mu$ is base point free.
Moreover, then there is a non-empty open subset of
polynomials $g \in S_\mu$ such that
$X_g$ is smooth.
\end{corollary}

\begin{proof}
According to~\cite[Prop.~3.3.2.8]{ADHL}, the class 
$\mu \in \Cl(Z)$ is base point free on $Z$ 
if and only if the following holds:
$$
\mu
\ \in \ 
\bigcap_{\gamma_0 \in \rlv(Z)} Q(\gamma_0 \cap \mathbb{Z}^r).
$$
To check the latter, 
let $\gamma_0 \in \rlv(Z)$. 
As $K_\mathbb{Q}$ is two-dimensional, 
we find $1 \le i, j \le r$ with $e_i, e_j \in \gamma_0$ 
and $\lambda^\circ \subseteq \cone(w_i, w_j)^\circ$. 
If $w_i, w_j$ generate $K$ as a group, then 
$K$ is torsion-free, $w_i, w_j$ form a Hilbert basis 
for $\cone(w_i, w_j)$ and thus
$\mu$ is a positive combination of $w_i, w_j$. 
Otherwise, the toric orbit $Z(\gamma_{i,j})$ is not smooth, hence
not contained in $Z_\mu$.
The latter means $V(g) \cap \bar{Z}(\gamma_{i,j}) = \emptyset$,
which in turn shows that~$g$ has a monomial
of the form $T_i^{l_i} T_j^{l_j}$ where $l_i + l_j > 0$.
Thus, $\mu$ is a positive combination of $w_i$ and $w_j$.

Knowing that $\mu$ is base point free, we obtain 
the supplement as a direct consequence of smoothness 
of $Z_\mu$ and Propositon~\ref{prop:bertini}.
\end{proof}

\section{Proof of Theorem~\ref{thm:candidates}}

We work in the combinatorial framework for 
Mori dream spaces provided in the preceding 
sections.
The ground field is now $\KK=\CC$,
due to the references we use;
see Remark~\ref{rem:smoothfanoprops}.
The major part of proving Theorem~\ref{thm:candidates},
is to figure out the candidates for specifying 
data of smooth general hypersurface Cox rings
of Fano fourfolds of Picard number two.
Having found the candidates, the remaining 
task is to verify them, that means to show 
that the given specifying data indeed define 
a smooth general hypersurface Cox ring 
of a Fano fourfold.
The precise setting for the elaboration 
is the following.

\begin{setting}
\label{rem:effrho2}
Consider a $K$-graded algebra $R$
and $X = X(\lambda)$,
where $\lambda \in \Lambda(R)$ with 
$\lambda^\circ \subseteq \Mov(R)^\circ$,
as in Construction~\ref{constr:mds}.
Assume that $\dim(K_\QQ) = 2$ holds
and that we have an irredundant 
$K$-graded presentation
$$ 
R 
\ = \ 
R_g 
\ = \ 
\CC[T_1, \ldots, T_r] / \bangle{g}
$$
such that the $T_i$ define pairwise 
nonassociated $K$-primes in $R$.
Write $w_i \coloneqq \deg(T_i)$,
$\mu \coloneqq \deg(g)$ for the degrees 
in $K$, also when regarded in $K_\QQ$.
Suitably numbering $w_1, \ldots, w_r$,
we ensure counter-clockwise ordering,
that means that we always have
$$ 
i \le j \ \implies \ \det(w_i,w_j) \ge 0.
$$
Note that each ray of $\Lambda(R)$ is of 
the form $\varrho_i  = \cone(w_i)$, but not 
vice versa.
We assume $X$ to be $\QQ$-factorial.
According to Proposition~\ref{prop:locprops}~(i)
this means $\dim(\lambda) = 2$.
Then the  effective cone of $X$ is uniquely 
decomposed into three convex sets,
$$
\Eff(X) 
\ = \ 
\lambda^- \cup \lambda^\circ \cup \lambda^+,
$$
where~$\lambda^-$ and~$\lambda^+$ are convex polyhedral 
cones not intersecting
$\lambda^\circ = \Ample(X)$ and 
$\lambda^- \cap \lambda^+$ consists of the origin.
By Remark~\ref{rem:EffRMovR}~(ii)
and Proposition~\ref{prop:mdsprops}, 
each of $\lambda^-$ and~$\lambda^+$ 
contains at least two of the degrees $w_1,\ldots,w_r$.
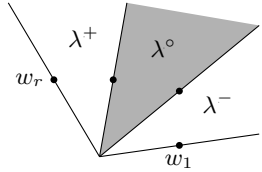
\begin{center}
    \begin{tikzpicture}[scale=0.6]
    \path[fill=gray!60!] (0,0)--(3.5,2.9)--(0.6,3.4)--(0,0);
    \path[fill, color=black] (1.75,1.45) circle (0.5ex)  node[]{};
    \path[fill, color=black] (1.4,2.4) circle (0.0ex)  node[]{\small{$\lambda^\circ$}};
    \path[fill, color=black] (0.3,1.7) circle (0.5ex)  node[]{};
    \draw (0,0)--(0.6,3.4);
    \draw (0,0) --(-2,3.4);
  \path[fill, color=black] (-1,1.7) circle (0.5ex)  node[left]{\small{$w_r$}};
    \path[fill, color=black] (-0.35,2.65) circle (0.0ex)  node[]{\small{$\lambda^+$}};
    \draw (0,0)  -- (3.5,2.9);
    \draw (0,0)  -- (3.5,0.5);
  \path[fill, color=black] (1.75,0.25) circle (0.5ex)  node[below]{\small{$w_1$}};
    \path[fill, color=black] (2.6,1.2) circle (0.0ex)  node[]{\small{$\lambda^-$}};
    \path[fill, color=white] (4,1.9) circle (0.0ex);
  \end{tikzpicture}   
\end{center}
Note that $\lambda^-$ as well as $\lambda^+$ might be 
one-dimensional.
As a GIT-cone in $K_\QQ \cong \QQ^2$,
the closure $\lambda = \SAmple(X)$ 
of $\lambda^\circ = \Ample(X)$ is the
intersection of two projected 
$\bar X$-faces
and thus we find at least one of the 
$w_i$ on each of its bounding rays.
\end{setting}

\begin{remark} 
Setting~\ref{rem:effrho2}
is respected by orientation preserving 
automorphisms of~$K$.
If we apply an orientation reversing 
automorphism of $K$, 
then we regain Setting~\ref{rem:effrho2}
by reversing the numeration of 
$w_1, \ldots, w_r$.
Moreover, we may interchange the numeration 
of $T_i$ and $T_j$ if $w_i$ and $w_j$
share a common ray without 
affecting Setting~\ref{rem:effrho2}.
We call these operations \emph{admissible
coordinate changes}. 
\end{remark}

\begin{remark} 
\label{rem:majorconsts}
In Setting~\ref{rem:effrho2}, consider the 
rays $\varrho_i := \cone(w_i) \subseteq \QQ^2$,
where $i = 1, \ldots, r$,
and the degree $\mu = \deg(g)$ of the relation.
Set
$$
\Gamma
\ := \
\varrho_1 \cup \ldots \cup \varrho_r,
\qquad\qquad
\Gamma^\circ 
\ := \
\Gamma \cap \Eff(R)^\circ.
$$
Then a suitable admissible coordinate change
turns the setting into one of the following
{\small
$$
\begin{array}{c}
\begin{tikzpicture}[scale=0.5]
\coordinate(o) at (0,0);
\coordinate(t) at (-1,-1);
\coordinate(r1) at (2,0);
\coordinate(r2) at (2,1);
\coordinate(r3) at (1,2);
\coordinate(r4) at (-1,2);
\coordinate(w1) at (1,0);
\coordinate(w2) at (1,0.5);
\coordinate(w3) at (0.5,1);
\coordinate(w4) at (-0.5,1);
\coordinate(mu) at (1.25,1.25);
\path[fill=gray!60!] (o)--(r1)--(r2)--(r3)--(r4)--(o);
\draw[thick] (o)--(r1);
\draw[thick] (o)--(r2);
\draw[thick] (o)--(r3);
\draw[thick] (o)--(r4);
\path[fill, color=black] (w1) circle (0.6ex);
\path[fill, color=black] (w2) circle (0.6ex);
\path[fill, color=black] (w3) circle (0.6ex);
\path[fill, color=black] (w4) circle (0.6ex);
\filldraw[fill=white, draw=black] (mu) circle (0.6ex);
\end{tikzpicture}   
\\
\\[-3pt]
{\rm{(I)}~\mu \not\in \Gamma}
\\[-3pt] \
\end{array}
\quad
\begin{array}{c}
\begin{tikzpicture}[scale=0.5]
\coordinate(o) at (0,0);
\coordinate(t) at (-1,-1);
\coordinate(r1) at (2,0);
\coordinate(r2) at (2,1);
\coordinate(r3) at (1,2);
\coordinate(r4) at (-1,2);
\coordinate(w1) at (1,0);
\coordinate(w1a) at (0.5,0);
\coordinate(w2) at (1,0.5);
\coordinate(w3) at (0.5,1);
\coordinate(w4) at (-0.5,1);
\coordinate(w4a) at (-0.25,.5);
\coordinate(mu) at (0.75,1.5);
\path[fill=gray!60!] (o)--(r1)--(r2)--(r3)--(r4)--(o);
\draw[thick] (o)--(r1);
\draw[thick] (o)--(r2);
\draw[thick] (o)--(r3);
\draw[thick] (o)--(r4);
\path[fill, color=black] (w1) circle (0.6ex);
\path[fill, color=black] (w1a) circle (0.6ex);
\path[fill, color=black] (w2) circle (0.6ex);
\path[fill, color=black] (w3) circle (0.6ex);
\path[fill, color=black] (w4) circle (0.6ex);
\path[fill, color=black] (w4a) circle (0.6ex);
\filldraw[fill=white, draw=black] (mu) circle (0.6ex);
\end{tikzpicture}
\\
\hphantom{\rm{(IIa)}~}\mu \in \Gamma^\circ
\\[-3pt]
\rm{(IIa)}~\varrho_1 = \varrho_2
\\[-3pt]
\hphantom{\rm{(IIa)}~}\varrho_{r-1} = \varrho_r
\end{array}
\quad
\begin{array}{c}
\begin{tikzpicture}[scale=0.5]
\coordinate(o) at (0,0);
\coordinate(t) at (-1,-1);
\coordinate(r1) at (2,0);
\coordinate(r2) at (2,1);
\coordinate(r3) at (1,2);
\coordinate(r4) at (-1,2);
\coordinate(w1) at (1,0);
\coordinate(w1a) at (0.5,0);
\coordinate(w2) at (1,0.5);
\coordinate(w3) at (0.5,1);
\coordinate(w4) at (-0.5,1);
\coordinate(w4a) at (-0.25,.5);
\coordinate(mu) at (0.75,1.5);
\path[fill=gray!60!] (o)--(r1)--(r2)--(r3)--(r4)--(o);
\draw[thick] (o)--(r1);
\draw[thick] (o)--(r2);
\draw[thick] (o)--(r3);
\draw[thick] (o)--(r4);
\path[fill, color=black] (w1) circle (0.6ex);
\path[fill, color=black] (w2) circle (0.6ex);
\path[fill, color=black] (w3) circle (0.6ex);
\path[fill, color=black] (w4) circle (0.6ex);
\path[fill, color=black] (w4a) circle (0.6ex);
\filldraw[fill=white, draw=black] (mu) circle (0.6ex);
\end{tikzpicture}
\\
\hphantom{\rm{(IIb)}~}\mu \in \Gamma^\circ
\\[-3pt]
\rm{(IIb)}~\varrho_1 \ne \varrho_2
\\[-3pt]
\hphantom{\rm{(IIb)}~}\varrho_{r-1} = \varrho_r
\end{array}
\quad
\begin{array}{c}
\begin{tikzpicture}[scale=0.5]
\coordinate(o) at (0,0);
\coordinate(t) at (-1,-1);
\coordinate(r1) at (2,0);
\coordinate(r2) at (2,1);
\coordinate(r3) at (1,2);
\coordinate(r4) at (-1,2);
\coordinate(w1) at (1,0);
\coordinate(w1a) at (0.5,0);
\coordinate(w2) at (1,0.5);
\coordinate(w3) at (0.5,1);
\coordinate(w4) at (-0.5,1);
\coordinate(w4a) at (-0.25,.5);
\coordinate(mu) at (0.75,1.5);
\path[fill=gray!60!] (o)--(r1)--(r2)--(r3)--(r4)--(o);
\draw[thick] (o)--(r1);
\draw[thick] (o)--(r2);
\draw[thick] (o)--(r3);
\draw[thick] (o)--(r4);
\path[fill, color=black] (w1) circle (0.6ex);
\path[fill, color=black] (w2) circle (0.6ex);
\path[fill, color=black] (w3) circle (0.6ex);
\path[fill, color=black] (w4) circle (0.6ex);
\filldraw[fill=white, draw=black] (mu) circle (0.6ex);
\end{tikzpicture}
\\
\hphantom{\rm{(IIc)}~}\mu \in \Gamma^\circ
\\[-3pt]
\rm{(IIc)}~\varrho_1 \ne \varrho_2
\\[-3pt]
\hphantom{\rm{(IIc)}~}\varrho_{r-1} \ne \varrho_r
\end{array}
\quad
\begin{array}{c}
\begin{tikzpicture}[scale=0.5]
\coordinate(o) at (0,0);
\coordinate(t) at (-1,-1);
\coordinate(r1) at (2,0);
\coordinate(r2) at (2,1);
\coordinate(r3) at (1,2);
\coordinate(r4) at (-1,2);
\coordinate(w1) at (1,0);
\coordinate(w2) at (1,0.5);
\coordinate(w3) at (0.5,1);
\coordinate(w4) at (-0.5,1);
\coordinate(mu) at (1.5,0);
\path[fill=gray!60!] (o)--(r1)--(r2)--(r3)--(r4)--(o);
\draw[thick] (o)--(r1);
\draw[thick] (o)--(r2);
\draw[thick] (o)--(r3);
\draw[thick] (o)--(r4);
\path[fill, color=black] (w1) circle (0.6ex);
\path[fill, color=black] (w2) circle (0.6ex);
\path[fill, color=black] (w3) circle (0.6ex);
\path[fill, color=black] (w4) circle (0.6ex);
\filldraw[fill=white, draw=black] (mu) circle (0.6ex);
\end{tikzpicture}   
\\
\\[-3pt]
{\rm{(III)}~\mu \in \varrho_1}
\\[-3pt] \
\end{array}
$$
}%
where the figures exemplarily sketch
the case $r=5$, the black dots indicate
the generator degrees and the white
dot stands for the relation degree.
\end{remark}

Our proof of Theorem~\ref{thm:candidates} will be
split into Parts I, IIa, IIb, IIc and III
according to the constellations of
Remark~\ref{rem:majorconsts}.
We exemplarily present Parts I, IIa and III.
The remaining parts use analogous arguments
and will be made available in~\cite{Ma}.
The reason why we restrict 
Theorem~\ref{thm:candidates}
to the ground field $\KK = \CC$ 
is that we use the following 
references on complex Fano 
varieties.

\begin{remark} 
\label{rem:smoothfanoprops}
Let $X$ be a smooth complex Fano variety.
Then the divisor class group $\Cl(X)$ 
of $X$ is torsion free; 
see for instance~\cite[Prop.~2.1.2]{AG5}.
Moreover, if $\dim(X) = 4$ holds,
then~\cite[Rem.~3.6]{Ca} tells us 
that any $\QQ$-factorial projective variety
being isomorphic in codimension one 
to~$X$ is smooth as well.
In terms of Construction~\ref{constr:mds}, 
the latter means that all varieties $X(\eta)$
are smooth, where $\eta \in \Lambda(R)$ 
is full-dimensional with 
$\eta^\circ \subseteq \Mov(R)^\circ$.
\end{remark}

We treat Case~\ref{rem:majorconsts}~I
that means that the degree of 
the defining relation is not proportional 
to any of the Cox ring generator degrees.
Here are first constraints on the possible
specifying data in this situation.

\begin{proposition} 
\label{prop:deginrelint-collaps}
In Setting~\ref{rem:effrho2},
assume that $r = 7$, $K \cong \ZZ^2$ holds,
every two-dimensional
$\lambda \in \Lambda(R)$ with 
$\lambda^\circ \subseteq \Mov(R)^\circ$
defines a locally factorial 
$X(\lambda)$
and $\mu$ doesn't lie on any of
the rays $\varrho_1, \ldots, \varrho_7$. 
Then, after a suitable 
admissible coordinate change, we have 
$\mu \in \cone(w_4, w_5)^\circ$
and one of the following holds:

\smallskip

\begin{minipage}[t]{0.5\textwidth}
\begin{enumerate}
\item $w_1 = w_2$ and $w_5 = w_6$,
\item $w_1 = w_2$ and $w_6 = w_7$,
\item $w_2 = w_3$ and $w_5 = w_6$,
\end{enumerate}
\end{minipage} %
\begin{minipage}[t]{0.5\textwidth}
\begin{enumerate}
\setcounter{enumi}{3}
\item $w_2 = w_3$ and $w_6 = w_7$,
\item $w_3 = w_4$ and $w_5 = w_6$,
\item $w_3 = w_4$ and $w_6 = w_7$.
\end{enumerate}
\end{minipage}
\end{proposition}

\begin{lemma} 
\label{lem:twofaces} 
Consider a locally factorial $X = X(\lambda)$ 
arising from Construction~\ref{rem:Coxconst}
with only one relation, i.e., $s=1$. 
Let $i, j$ with $\lambda \subseteq \cone(w_i,w_j)$.
Then either $w_i,w_j$ generate $K$ as a group,
or $g_1$ has precisely one monomial of the form 
$T_i^{l_i} T_j^{l_j}$, where $l_i+l_j > 0$.
\end{lemma}

\begin{proof}
If $\gamma_{i,j}$ is an $X$-face, then 
Proposition~\ref{prop:locprops}~(ii) tells
us that $w_i$ and $w_j$ generate~$K$ as a group.
Now consider the case that  $\gamma_{i,j}$ 
not an $X$-face.
Then we must have 
$\lambda^\circ \not\subseteq Q(\gamma_{i,j})^\circ$ 
or $\gamma_{i,j}$ is not an $\bar X$-face.
Proposition~\ref{prop:locprops}~(i) excludes
the first possibility.
Thus, the second one holds, which in turn means 
that $g_1$ has precisely one monomial of the form 
$T_i^{l_i} T_j^{l_j}$, where $l_i+l_j > 0$.
\end{proof}

\begin{lemma} 
\label{lem:threegenerate}
Let $X = X(\lambda)$ be as in Setting~\ref{rem:effrho2}
and let $1 \leq i < j < k \leq r$.
If $X$ is locally factorial, 
then $w_i, w_j, w_k$ generate $K$ as a group
provided  that one of the following holds:
\begin{enumerate}
\item 
$w_i, w_j \in \lambda^-$, $w_k \in \lambda^+$ and $g$ 
has no monomial of the form $T_k^{l_k}$,
\item 
$w_i \in \lambda^-$, $w_j, w_k \in \lambda^+$ and $g$ 
has no monomial of the form $T_i^{l_i}$.
\end{enumerate}
\end{lemma}

\begin{proof}
Assume that~(i) holds.
If $K$ is generated by $w_i, w_k$ or by $w_j, w_k$, 
then we are done. 
Consider the case that none of the pairs
$w_i, w_k$ and $w_j, w_k$ generates~$K$. 
Applying Lemma~\ref{lem:twofaces} to 
each of the pairs shows that $g$ has 
precisely one monomial of the form 
$T_i^{l_i} T_k^{l_k}$ with $l_i+l_{k} > 0$
and precisely one monomial of the form 
$T_j^{l_j} T_k^{l_k'}$ with $l_j+l_k' > 0$.
By assumption, we must have $l_i,l_j > 0$.
We conclude that $\gamma_{i,j,k}$ is an
$X$-face. 
Since $X$ is locally factorial, 
Proposition~\ref{prop:locprops}~(ii)
yields that $w_i, w_j, w_k$ generate~$K$. 
If~(ii) holds, then a suitable admissible 
coordinate change leads to~(i).
\end{proof}

\begin{lemma} 
\label{lem:2on1ray}
Assume $u, w_1, w_2$ generate the abelian group $\mathbb{Z}^2$.
If $w_i = a_i w$ holds with a primitive $w \in \mathbb{Z}^2$
and $a_i \in \mathbb{Z}$,
then $(u, w)$ is a basis for $\mathbb{Z}^2$ 
and $u$ is primitive. 
\end{lemma}

\begin{lemma} 
\label{lem:4det}
Let $w_1, \dotsc, w_4 \in \mathbb{Z}^2$ such that 
$\det(w_1, w_3)$, $\det(w_1, w_4)$,
$\det(w_2, w_3)$ and $\det(w_2, w_4)$
all equal one. 
Then $w_1 = w_2$ or $w_3 = w_4$ holds.
\end{lemma}

\begin{proof}[{Proof of Proposition~\ref{prop:deginrelint-collaps}}]
The assumption $\mu \not\in \varrho_i$
implies $\varrho_i \in \Lambda(R)$ for 
$i = 1, \ldots, 7$,
see Remark~\ref{rem:barXfacecrit}~(ii). 
Proposition~\ref{prop:hypmov} gives
$\mu \in \cone(w_3, w_5)$.
The latter cone is the union of
$\cone(w_3, w_4)$ and $\cone(w_4, w_5)$;
both are GIT-cones, one of them 
is two-dimensional and hosts $\mu$ in 
its relative interior. 
A suitable admissible coordinate
change yields $\mu \in \cone(w_4, w_5)^\circ$.

First we show that if $w_i \in \varrho_j$ 
holds for some $1 \le i < j \le 4$, 
then two of $w_5,w_6,w_7$ coincide.
Consider the case $w_5, w_6 \in \varrho_5$. 
By assumption $X(\lambda)$ is locally 
factorial for $\lambda = \cone(w_4,w_5)$.
Thus, we can apply 
Lemma~\ref{lem:threegenerate} 
to $w_i, w_j, w_5$ and also to $w_i, w_j, w_6$
and obtain that each of the triples 
generates~$K$ as a group.
Lemma~\ref{lem:2on1ray} yields that $w_5$ and~$w_6$ 
are primitive and hence, lying on a common
ray, coincide.
Now, assume $w_6 \not\in \varrho_5$.
Then we consider $X = X(\lambda)$ for 
$\lambda = \cone(w_5,w_6)$.
Using Lemma~\ref{lem:threegenerate}
as before, see that $w_i, w_j, w_6$ as well as 
$w_i, w_j, w_7$ generate $K$ as a group.
For the primitive generator~$w$ of 
$\varrho_i = \varrho_j$, we infer 
$\det(w,w_6) = 1$ and $\det(w,w_7) = 1$ 
from Lemma~\ref{lem:2on1ray}.
Moreover, $\gamma_{5,6}$ and $\gamma_{5,7}$ 
are $X$-faces due to 
Remark~\ref{rem:barXfacecrit}~(ii).
Thus, Proposition~\ref{prop:locprops}~(ii)
yields $\det(w_5,w_6) = 1$ and 
$\det(w_5,w_7) = 1$.
Lemma~\ref{lem:4det} yields $w_6 = w_7$.

We conclude the proof by 
showing that at least 
two of $w_1, \ldots, w_4$ coincide.
Consider the case $w_2 \in \varrho_3$. 
Then, by the first step, there are $5 \leq i < j \leq 7$ 
with $w_i = w_j$. 
Taking $X(\lambda)$ for $\lambda = \cone(w_4,w_5)$ 
and applying Lemma~\ref{lem:threegenerate} 
to $w_2, w_i, w_j$ as well as to $w_3, w_i, w_j$,
we obtain that each of these triples generates $K$. 
Because of $w_i = w_j$, we directly see that~$w_2$ and~$w_3$,
each being part of a $\ZZ$-basis, are primitive and 
hence coincide.
We are left with the case that 
$\lambda' = \cone(w_2,w_3)$ 
is of dimension two.
By assumption, the variety $X'$ defined by $\lambda'$ 
is locally factorial.
Moreover, Remark~\ref{rem:barXfacecrit}~(ii) 
provides us with the $X'$-faces 
$\gamma_{1,3},\gamma_{2,3},\gamma_{1,4}$ and $\gamma_{2,4}$.
By Proposition~\ref{prop:locprops}~(ii), all corresponding 
determinants $\det(w_k,w_m)$ equal one.
Lemma~\ref{lem:4det} shows that at least two of 
$w_1, \ldots, w_4$ coincide.
\end{proof}

We are ready to enter Part~I of the proof
of Theorem~\ref{thm:candidates}.
The task is to work out further the degree 
constellations left by 
Proposition~\ref{prop:deginrelint-collaps}.
This leads to major multistage case distinctions.
We demonstrate how to get through 
for two of the constellations of 
Proposition~\ref{prop:deginrelint-collaps},
chosen in a manner that basically 
all the necessary arguments of 
Part~I of the proof show up. 
For the full elaboration of all cases we 
refer to~\cite{Ma}.

\begin{proof}[Proof of Theorem~\ref{thm:candidates}, Part I]
This part of the proof treats
the case that $\mu = \deg(g)$ doesn't lie 
on any of the rays $\varrho_i = \cone(w_i)$.
In particular, by Remark~\ref{rem:barXfacecrit}~(ii),
all rays $\varrho_1,\ldots,\varrho_7$ belong to
the GIT-fan $\Lambda(R)$.
By Remark~\ref{rem:smoothfanoprops},
every two-dimensional $\eta \in \Lambda(R)$
with $\eta^\circ \subseteq \Mov(R)^\circ$
produces a smooth variety $X(\eta)$.
Thus, we can apply 
Proposition~\ref{prop:deginrelint-collaps},
which leaves us with $\mu \in \cone(w_4, w_5)^\circ$
and the six possible constellations for 
$w_1, \ldots, w_7$ given there.
Again by Remark~\ref{rem:smoothfanoprops},
the divisor class group of $X$ is torsion
free, that means that we have $K = \ZZ^2$.

\medskip
\noindent
\emph{Constellation~\ref{prop:deginrelint-collaps}~(i)}.
We have $w_1=w_2$ and $w_5 = w_6$.
Lemma~\ref{lem:threegenerate} applied to $w_1, w_2, w_5$
shows that $w_1, w_5$ form a basis of $\mathbb{Z}^2$.
Thus, a suitable admissible coordinate change 
gives $w_1 = (1,0)$ and $w_6 = (0,1)$.
Applying Lemma~\ref{lem:threegenerate} also to
$w_1, w_2, w_7$ and 
$w_i, w_5, w_6$ where $i = 1, \dots, 4$ yields 
the first coordinate of $w_1, \ldots, w_4$ 
and the second coordinate of $w_7$ 
equal one. 
Thus, the degree matrix has the form
\[
Q 
\ = \
[w_1, \dotsc, w_7 ] 
\ = \
\left[
\begin{array}{rrrrrrr}
1 & 1 &   1 &   1 & 0 & 0 & -a_7 
\\
0 & 0 & b_3 & b_4 & 1 & 1 &    1 
\end{array}
\right],
\qquad
b_3,b_4,a_7 \in \ZZ_{\ge 0}.
\]
We determine the possible values of $b_3$ and $b_4$.
If $b_3 > 0$ holds, then $\eta = \cone(w_2, w_3)$
is two-dimensional and satisfies
$\eta^\circ \subseteq \Mov(R)^\circ$. 
Because of $\mu \in \cone(w_4, w_5)^\circ$, 
none of the monomials of $g$ is of the 
form $T_1^{l_1}T_j^{l_j}$ with $j = 3,4$.
Lemma~\ref{lem:twofaces} applied to
$X(\eta)$ gives $b_j = \det(w_1, w_j) = 1$
for $j = 3,4$.  
If $b_3 = 0$ and $b_4 > 0$ hold, we argue similarly
with $\eta = \cone(w_2, w_4)$ and obtain $b_4 = 1$.
Altogether, we arrive at the three cases
$$
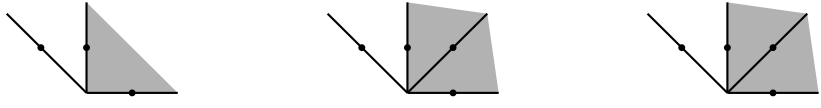
\begin{array}{c}
\begin{tikzpicture}[scale=0.6]
\path[fill=gray!60!] (0,0)--(2,0)--(0,2)--(0,0);
\draw[thick] (0,0)--(2,0);
\draw[thick] (0,0)--(0,2);
\draw[thick] (0,0)--(-1.75,1.75);
\path[fill, color=black] (1,0) circle (0.5ex);
\path[fill, color=black] (0,1) circle (0.5ex);
\path[fill, color=black] (-1,1) circle (0.5ex);
\end{tikzpicture}   
\\[1em]
\text{\ref{prop:deginrelint-collaps}~(i-a): $b_3 = b_4 = 0$,}
\end{array}
\quad
\begin{array}{c}
\begin{tikzpicture}[scale=0.6]
\path[fill=gray!60!] (0,0)--(2,0)--(1.75,1.75)--(0,2)--(0,0);
\draw[thick] (0,0)--(2,0);
\draw[thick] (0,0)--(0,2);
\draw[thick] (0,0)--(-1.75,1.75);
\draw[thick] (0,0)--(1.75,1.75);
\path[fill, color=black] (1,0) circle (0.5ex);
\path[fill, color=black] (0,1) circle (0.5ex);
\path[fill, color=black] (-1,1) circle (0.5ex);
\path[fill, color=black] (1,1) circle (0.5ex);
\end{tikzpicture}   
\\[1em]
\text{\ref{prop:deginrelint-collaps}~(i-b): $b_3 = 0$, $b_4 = 1$,}
\end{array}
\quad
\begin{array}{c}
\begin{tikzpicture}[scale=0.6]
\path[fill=gray!60!] (0,0)--(2,0)--(1.75,1.75)--(0,2)--(0,0);
\draw[thick] (0,0)--(2,0);
\draw[thick] (0,0)--(0,2);
\draw[thick] (0,0)--(-1.75,1.75);
\draw[thick] (0,0)--(1.75,1.75);
\path[fill, color=black] (1,0) circle (0.5ex);
\path[fill, color=black] (0,1) circle (0.5ex);
\path[fill, color=black] (-1,1) circle (0.5ex);
\path[fill, color=black] (1,1) circle (0.5ex);
\end{tikzpicture}   
\\[1em]
\text{\ref{prop:deginrelint-collaps}~(i-c): $b_3 = b_4 = 1$.}
\end{array}
$$
\emph{Case~\ref{prop:deginrelint-collaps}~(i-a)}. 
Here, the semiample cone $\lambda$ of $X = X(\lambda)$
must be the positive orthant. 
Thus, $X$ being Fano just means that both 
coordinates of the anticanonical class
$-\KKK_X \in K = \ZZ^2$ are strictly positive. 
According to Proposition~\ref{prop:anticanclass},
we have 
$$
-\KKK_X 
\ = \ 
(4 - a_7 - \mu_1, \, 3 - \mu_2).
$$
We conclude $1 \le \mu_2 \le 2$
and $1 \le \mu_1 < 4 - a_7$
which implies in particular 
$0 \le a_7 \le 2$. 
Thus, the weights $w_1,\ldots,w_7$ and 
the degree $\mu$ must be as in 
Theorem~\ref{thm:candidates}, 
Numbers~\ref{cand:deginrelint-ia-1}
to~\ref{cand:deginrelint-ia-12}.
We exemplarily verify the candidate Number~12;
the others are settled analogously.
We have to deal with the specifying data 
\[
Q \ = \
[w_1, \dotsc ,w_7]
\ = \ 
\left[
\begin{array}{ccccccc}
1 & 1 & 1 & 1 & 0 & 0 & -2
\\
0 & 0 & 0 & 0 & 1 & 1 & 1   
\end{array}
\right],
\qquad\qquad
\mu \ = \ (1,2).
\]
We run Construction~\ref{constr:hypersurf}
with $\tau \in \Lambda$ such that 
$-\mathcal{K} = (1,1) \in \tau^\circ$ 
and show that the result is a smooth
general hypersurface Cox ring $R_g$.
First, one directly checks that the
convex hull over the $\nu \in \ZZ^7$ 
with $Q(\nu) = \mu$ is Dolgachev polytope.
Thus, Proposition~\ref{prop:ufdcrit}~(i) 
delivers a non-empty open set $U \subseteq S_\mu$ 
such that $R_g$ is factorial 
for all $g \in U$.
Since $\mu \neq w_i$ holds for all $i$,
Remark~\ref{rem:genCRirredundant} ensures
that $T_1, \dotsc, T_7$ are a minimal system 
of generators for $R_g$, whenever
$g \in U$.
For $i \neq 5,6$ the degree $w_i$ of $T_i$ 
is indecomposable in the monoid $\Eff(R_g) \cap K$. 
We conclude that~$T_i$ is irreducible 
and thus prime in $R_g$, whenever
$g \in U$.
To see primality of $T_5$ and~$T_6$,
we use Proposition~\ref{prop:varclassprime},
where we can take $T_1 T_i^2 - T_2^5 T_7^2$
for $i = 6,5$
as the required $\mu$-homogeneous 
polynomial in both cases.
The ambient toric variety $Z$ is smooth  
due to Proposition~\ref{prop:locprops}~(iv).
Thus, also $Z_\mu$ is smooth.
Because of $\mu \in \tau^\circ$,
Corollary~\ref{cor:rk2bertini}
applies and, suitably shrinking $U$, 
we achieve that $X_g$ 
is smooth for all $g \in U$.

\medskip
\noindent
\emph{Case~\ref{prop:deginrelint-collaps}~(i-b)}. 
Here, either $\lambda = \cone(w_3, w_4)$  
or $\lambda = \cone(w_4, w_5)$ holds.
In any case, the anticanonical class is given as
$$
-\KKK_X 
\ = \ 
(4 - a_7 - \mu_1, \, 4 - \mu_2).
$$
First assume that $\lambda = \cone(w_3, w_4)$
holds.
Then, $X$ being Fano, we have 
$-\KKK_X \in \lambda^\circ$.
The latter is equivalent to the inequalities   
$$
4 - \mu_2 \ > \ 0, 
\qquad\qquad
\mu_2 - \mu_1 -a_7 \ > \ 0.
$$
Using $\mu \in \cone(w_4, w_5)^\circ$,
we conclude $1 \le \mu_1 < \mu_2 \le 3$
and $0 \le a_7 \le 1$.
Thus, we end up with 
$$
a_7 = 0 \text{ and } \mu = (1,2), (1,3), (2,3),
\qquad
a_7 = 1 \text{ and } \mu = (1,3).
$$
Note that in all cases, $\gamma_{1,2,3,4}$
is an $X$-face according to 
Remark~\ref{rem:barXfacecrit}~(ii).
Since $X$ is quasismooth,
Proposition~\ref{prop:qsmooth2degrees}
yields
$$ 
\mu 
\ \in  \
Q(\gamma_{1,2,3,4}) \, \cup \, w_7 + Q(\gamma_{1,2,3,4}) .
$$
This excludes $a_7 = 0$ and $\mu = (1,3)$.
The remaining three cases are 
Numbers~\ref{cand:deginrelint-ib-1}
to~\ref{cand:deginrelint-ib-4}
of Theorem~\ref{thm:candidates}.
All these candidates can be verified.
Indeed, take $\tau = \cone(w_3, w_4)$
for all three cases and observe
$-\KKK \in \cone(w_3, w_4)^\circ$.
As in Case~\ref{prop:deginrelint-collaps}~(i-a),
we find a non-empty open subset $U$ of 
polynomials $g \in S_\mu$
such that $R_g$ admits unique factorization,
see that $T_1, \dotsc, T_7$ define pairwise 
non-associated primes in~$R_g$
and observe that $Z$ and thus also $Z_\mu$ 
are smooth.
For smoothness of $X_g$, it suffices to 
show that $\hat{X}_g$ is smooth; 
see~Proposition~\ref{prop:Xgsmooth}.
By Proposition~\ref{prop:Xgqsmooth},
it suffices to find some $g \in S_\mu$ 
such that $\grad(g)$ has no zeroes 
in $\hat{Z}$, then shrinking $U$ suitably
yields that $\hat{X}_g$ is smooth for all 
$g \in U$.
We just chose a random $g$ of degree $\mu$ and 
verified this using~\cite{MDSpackage}. 
For instance, for $a_7 = 0$ and $\mu = (1,2)$,
that means Number~\ref{cand:deginrelint-ib-1},
the following $g$ does the job:
\begin{align*}
8T_{1}T_{5}^2 
& + 7T_{1}T_{5}T_{6} + 7T_{1}T_{5}T_{7}
+ 6T_{1}T_{6}^2 + 4T_{1}T_{6}T_{7} + T_{1}T_{7}^2
+ 7T_{2}T_{5}^2 + 7T_{2}T_{5}T_{6}
\\
& + 3T_{2}T_{5}T_{7} + 8T_{2}T_{6}^2 + 5T_{2}T_{6}T_{7} + 8T_{2}T_{7}^2  
+ 5T_{3}T_{5}^2 + 4T_{3}T_{5}T_{6} + 9T_{3}T_{5}T_{7} 
\\
&  + 2T_{3}T_{6}^2 + 9T_{3}T_{6}T_{7} + T_{3}T_{7}^2
+ 8T_{4}T_{5} + 3T_{4}T_{6} + 6T_{4}T_{7}.
\end{align*}
Now, assume that $\lambda = \cone(w_4, w_5)$ holds.
The condition that $X = X(\lambda)$ is Fano means 
$-\KKK_X \in \lambda^\circ$, which
translates into the inequalities 
$0 < 4 - a_7 - \mu_1 < 4 - \mu_2$.
Moreover, $\mu \in \lambda^\circ$ implies
$\mu_1 < \mu_2$ and we conclude
$$ 
1 
\ \le \ 
\mu_1 
\ < \ 
\mu_2 
\ < \ 
\mu_1 + a_7 
\ \le \ 
3.
$$
This is only possible for $a_7 = 2$ and $\mu = (1,2)$. 
Then we have $w_4 = (1,1)$ and $w_7 = (-2,1)$.
In particular, $g$ admits no monomial of the form 
$T_4^{l_4} T_7^{l_7}$.
Lemma~\ref{lem:twofaces} tells us that $w_4$ and $w_7$ 
generate $K = \mathbb{Z}^2$ as a group.
A contradiction.

\medskip
\noindent
\emph{Case~\ref{prop:deginrelint-collaps}~(i-c)}.
Applying Remark~\ref{rem:smoothfanoprops}
and Lemma~\ref{lem:threegenerate} to
$X(\eta)$ with $\eta = \cone(w_4, w_5)$
and $w_3, w_4, w_7$ yields $\det(w_4, w_7) = 1$. 
From this we infer $a_7 = 0$.
Thus, either $\lambda = \cone(w_2, w_3)$
or $\lambda = \cone(w_4, w_5)$ holds.
In any case, the anticanonical class is
$$
-\KKK_X \ = \ (4 - \mu_1, \, 5 - \mu_2).
$$
Assume $\lambda = \cone(w_2, w_3)$. 
Then the Fano condition $-\KKK_X \in \lambda^\circ$ 
implies $\mu_1 + 1 < \mu_2$. 
Remark~\ref{rem:barXfacecrit}~(ii) says that 
$\gamma_{1,2,3,4}$ is an $X$-face.
As before, Proposition~\ref{prop:qsmooth2degrees} 
gives
$$ 
\mu 
\ \in  \
Q(\gamma_{1,2,3,4}) \, \cup \, w_7 + Q(\gamma_{1,2,3,4}) .
$$
We conclude $\mu_1 +1 \ge \mu_2$. A contradiction.
Now, assume $\lambda = \cone(w_4, w_5)$.
Then $-\KKK_X \in \lambda^\circ$ 
yields $\mu_1  \ge \mu_2$.
But we have $\mu \in \cone(w_4, w_5)^\circ$,
hence $\mu_1 < \mu_2$.
A contradiction.

\bigskip
\noindent
\emph{Constellation~\ref{prop:deginrelint-collaps}~(ii)}.
We have $w_1 = w_2$ and $w_6 = w_7$.
Lemma~\ref{lem:threegenerate} applied to $w_1, w_6, w_7$
shows that $w_1, w_7$ generate $\mathbb{Z}^2$.
Hence, a suitable admissible coordinate change 
yields $w_1 = w_2 = (1,0)$ and $w_6 = w_7 = (0,1)$.
Applying Lemma~\ref{lem:threegenerate} to $w_3, w_6, w_7$
and $w_4, w_6, w_7$, we obtain that the first 
coordinates of $w_3$ and $w_4$ both equal one. 
Thus, the degree matrix has the form
\[
Q \ = \
[w_1, \dotsc ,w_7]
\ = \ 
\left[
\begin{array}{ccccccc}
1 & 1 &   1 &   1 & a_5 & 0 & 0
\\
0 & 0 & b_3 & b_4 & 1 & 1 & 1 
\end{array}
\right],
\qquad\qquad 
a_5, b_3, b_4 \in \ZZ_{\geq 0}.
\]
By assumption $w_4$ and $w_5$ don't lie on a common ray. 
Consequently, $b_4 = 0$ or $a_5 = 0$ holds. 
If $a_5 = 0$ holds, then we are in 
Constellation~\ref{prop:deginrelint-collaps}~(i)
just treated. 
So, assume $a_5 > 0$. Then $b_3 = b_4 = 0$ holds. 
Taking $X(\eta)$ for $\eta = \cone(w_5, w_6)$ and 
applying Lemma~\ref{lem:twofaces} to $w_5, w_6$
yields $a_5 = 1$.
We arrive at the degree matrix
\[
Q \ = \
[w_1, \dotsc ,w_7]
\ = \ 
\left[
\begin{array}{ccccccc}
1 & 1 & 1 & 1 & 1 & 0 & 0
\\
0 & 0 & 0 & 0 & 1 & 1 & 1
\end{array}
\right]. 
\]
Observe that either $\lambda = \cone(w_4, w_5)$
or $\lambda = \cone(w_5, w_6)$ holds. In any case, the
anticanonical class of $X = X(\lambda)$ is given as
$$
-\KKK_X \ = \ (5 - \mu_1, \, 3 - \mu_2).
$$
First, assume $\lambda = \cone(w_4, w_5)$. 
Then $X$ being Fano means
$0 < 3 - \mu_2 < 5 - \mu_1$.
We conclude $\mu_2 \le 2$ and $\mu_1 \le \mu_2+1$.
Moreover, $\mu \in \cone(w_4, w_5)^\circ$ gives  
$0 < \mu_2 < \mu_1$. 
Thus, we have $\mu_1 = \mu_2 + 1$ and arrive at 
the possibilities $\mu = (2, 1), (3,2)$,
which are Numbers~\ref{cand:deginrelint-ii-1}
and~\ref{cand:deginrelint-ii-2} in Theorem~\ref{thm:candidates}.
Showing that these constellations indeed define smooth
Fano varieties runs exactly as in 
Case~\ref{prop:deginrelint-collaps}~(i-a).
Now, let $\lambda = \cone(w_5, w_6)$. 
Then~$X$ being Fano gives $0 < 5-\mu_1<3-\mu_2$.
We conclude $\mu = (4, 1)$.
Remark~\ref{rem:barXfacecrit}~(ii) provides 
us with the $X$-face $\gamma_{5,6,7}$.
Proposition~\ref{prop:qsmooth2degrees} says
that $\mu$ should lie in $Q(\gamma_{5,6,7})$ or in
$w_1 + Q(\gamma_{5,6,7})$. A contradiction.
\end{proof}

We treat Case~\ref{rem:majorconsts}~IIa
that means that the degree of the relation 
lies in the interior of the effective cone,
is proportional to some Cox ring generator degree 
and $\varrho_1 = \varrho_2$ as well as 
$\varrho_{r-1} = \varrho_r$ hold.

\begin{lemma} 
\label{lem:combminray}
In Setting~\ref{rem:effrho2} assume that
$\Mov(R) = \Eff(R)$ and $\mu \in \Eff(R)^\circ$
hold.
Let $\Omega$ denote the set of two-dimensional 
$\lambda \in \Lambda(R)$ with
$\lambda^\circ \subseteq \Mov(R)^\circ$.
\begin{enumerate}
\item 
If $X(\lambda)$ is locally factorial 
for some $\lambda \in \Omega$, 
then $\Eff(R)$ is a regular cone and 
every $w_i$ on the boundary of $\Eff(R)$ 
is primitive.

\item 
If $X(\lambda)$ is locally factorial for 
each $\lambda \in \Omega$, 
then, for any $w_i \in \Eff(R)^\circ$, 
we have $w_i = w_1 + w_r$ or $g$ has a 
monomial of the form $T_i^{l_i}$.
\end{enumerate}
\end{lemma}

\begin{proof}
We show~(i).
Let $w_i \in \varrho_r$. 
Due to $\mu \in \Eff(R)^\circ$, 
the relation $g$ has no monomial of the 
form $T_i^{l_i}$.
Thus, Lemmas~\ref{lem:threegenerate} 
and~\ref{lem:2on1ray} applied to the 
triple $w_1,w_2,w_i$ show that $w_i$ is 
primitive.
Analogously, we see that any 
$w_i \in \varrho_1$ is primitive.
In particular, we have $w_1=w_2$.
Thus, applying Lemma~\ref{lem:threegenerate} 
to $w_1,w_2,w_r$, we 
obtain that $\Eff(R)$ is a regular cone.

We turn to~(ii).
By~(i), we may assume $w_1 = w_2 = (1,0)$
and $w_{r-1} = w_r = (0,1)$.
Consider $w_i \in \Eff(R)^\circ$ 
such that $T_i^{l_i}$ is not
a monomial of $g$.
Then we find GIT-cones 
$\lambda_1 \subseteq \cone(w_1, w_i)$
and $\lambda_2 \subseteq \cone(w_i, w_r)$
defining locally factorial 
varieties~$X(\lambda_1)$ and $X(\lambda_2)$
respectively.
Lemma~\ref{lem:threegenerate}, applied
to $w_1,w_2,w_i$ together with~$X(\lambda_1)$ 
and to $w_i, w_{r-1}, w_r$ together with 
$X(\lambda_2)$
shows $w_i = (1,1) = w_1 + w_r$. 
\end{proof}

\begin{lemma} 
\label{lem:raypower}
In Setting~\ref{rem:effrho2},
assume that $X = X(\lambda)$ is locally factorial
and $R_g$ a spread hypersurface Cox ring.
\begin{enumerate}
\item
If $w_i$ lies on the ray through $\mu$,
then $g$ has a monomial of the form $T_i^{l_i}$
where $l_i \ge 2$.
\item
If $w_i, w_j$, where $i \ne j$, 
lie on the ray through $\mu$, 
then $\varrho_i = \varrho_j \in \Lambda(R_g)$ 
holds.
\end{enumerate}
\end{lemma}

\begin{proof}
We show~(i).
Suppose that $g$ has no monomial of the 
form~$T_i^{l_i}$ where $l_i \geq 2$.
As $R_g$ is a hypersurface Cox ring,
also $T_i$ is not a monomial of $g$.
Then, on one of the extremal rays of $\Eff(R)$,
we find a $w_j$ such that
$\gamma_{i,j}$ is a $X$-face;
see Remark~\ref{rem:barXfacecrit}~(i).
Proposition~\ref{prop:locprops}~(ii) yields 
that $w_i, w_j$ generate $\mathbb{Z}^2$
as a group.
In particular, $w_i$ is primitive.
Hence $\mu = k w_i$ holds for some
$k \in \mathbb{Z}_{\ge 1}$.
As $R_g$ is spread,
$T_i^k$ must be a monomial of $g$. 
In addition, we obtain $k \geq 2$.
A contradiction.

We prove~(ii).
Assertion~(i) just proven and 
Remark~\ref{rem:barXfacecrit}~(i) 
tell us that $\gamma_{i,j}$ is an 
$\bar X$-face.
Thus, being a ray, 
$Q(\gamma_{i,j}) = \varrho_i = \varrho_j$
belongs to the GIT-fan $\Lambda(R_g)$. 
\end{proof}

\begin{proof}[Proof of Theorem~\ref{thm:candidates}, Part IIa]
We deal with the specifying data of a
smooth general hypersurface Cox ring $R$ 
as in Remark~\ref{rem:majorconsts}~IIa
defining a smooth Fano fourfould $X = X(\lambda)$.
By Proposition~\ref{prop:hypmov}, 
the relation degree~$\mu$ lies on 
$\varrho_3$, $\varrho_4$ or $\varrho_5$.
We claim that we can't have 
$\varrho_3 = \varrho_4 = \varrho_5$.
Otherwise Lemma~\ref{lem:raypower} shows
$\eta = \cone(w_1, w_3) \in \Lambda(R)$.
Since $X(\eta)$ is smooth by Remark~\ref{rem:smoothfanoprops},
we may apply Lemma~\ref{lem:threegenerate} to
the triple $w_1, w_3, w_4$. According to
Lemma~\ref{lem:2on1ray} we obtain
$\det(w_1, v) = 1$ where $v$ denotes
the primitive generator of the ray $\varrho_3$.
Analogous arguments
yield $\det(v, w_7) = 1$. Using both determinantal equations
we conclude that $v$ and $w_1 + w_7$ are collinear. 
In particular $w_1 + w_7$ generates $\varrho_3 = \varrho_4 = \varrho_5$.
Lemma~\ref{lem:combminray}~(i) tells us $w_1 = w_2$
and $w_6 = w_7$.
As a result, Proposition~\ref{prop:anticanclass} gives
$-\KKK_X \in \varrho_3$.
Moreover, Lemma~\ref{lem:raypower}~(ii)
tells us $\varrho_3 \in \Lambda(R_g)$
and thus $\lambda = \varrho_3$,
which contradicts $\QQ$-factoriality, 
see Proposition~\ref{prop:locprops}~(i).
A suitable admissible coordinate change yields
$\mu \notin \varrho_5$ and we are left 
with the following three constellations:
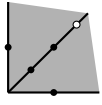
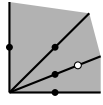
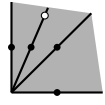
\[
\begin{array}{c}
\begin{tikzpicture}[scale=0.6]
\path[fill=gray!60!] (0,0)--(2,0)--(1.75,1.75)--(0,2)--(0,0);
\draw[thick] (0,0)--(2,0);
\draw[thick] (0,0)--(0,2);
\draw[thick] (0,0)--(1.75,1.75);
\path[fill, color=black] (1,0) circle (0.5ex);
\path[fill, color=black] (0,1) circle (0.5ex);
\path[fill, color=black] (0.5,0.5) circle (0.5ex);
\path[fill, color=black] (1,1) circle (0.5ex);
\filldraw[fill=white, draw=black] (1.5,1.5) circle (0.5ex);
\end{tikzpicture}   
\\[1em]
\text{(i)}~\varrho_3 = \varrho_4,\, \mu \in \varrho_3
\end{array}
\quad
\begin{array}{c}
\begin{tikzpicture}[scale=0.6]
\path[fill=gray!60!] (0,0)--(2,0)--(2,0.8)--(1.75,1.75)--(0,2)--(0,0);
\draw[thick] (0,0)--(2,0);
\draw[thick] (0,0)--(0,2);
\draw[thick] (0,0)--(1.75,1.75);
\draw[thick] (0,0)--(2,0.8);
\path[fill, color=black] (1,0) circle (0.5ex);
\path[fill, color=black] (0,1) circle (0.5ex);
\path[fill, color=black] (1,0.4) circle (0.5ex);
\filldraw[fill=white, draw=black] (1.5,0.6) circle (0.5ex);
\path[fill, color=black] (1,1) circle (0.5ex);
\end{tikzpicture}   
\\[1em]
\text{(ii)}~\varrho_3 \neq \varrho_4,\, \mu \in \varrho_3
\end{array}
\quad
\begin{array}{c}
\begin{tikzpicture}[scale=0.6]
\path[fill=gray!60!] (0,0)--(2,0)--(1.75,1.75)--(0,2)--(0,0);
\draw[thick] (0,0)--(2,0);
\draw[thick] (0,0)--(0,2);
\draw[thick] (0,0)--(1.75,1.75);
\draw[thick] (0,0)--(0.8,1.87);
\path[fill, color=black] (1,0) circle (0.5ex);
\path[fill, color=black] (0,1) circle (0.5ex);
\path[fill, color=black] (1,1) circle (0.5ex);
\path[fill, color=black] (0.43,1) circle (0.5ex);
\filldraw[fill=white, draw=black] (0.73,1.7) circle (0.5ex);
\end{tikzpicture}   
\\[1em]
\text{(iii)}~\varrho_3 \neq \varrho_4,\, \mu \in \varrho_4
\end{array}
\]
By Lemma~\ref{lem:combminray}~(i), we can 
assume $w_1 = w_2 = (1,0)$ and $w_6 = w_7 = (0,1)$.
We show $w_5 = (0,1)$.
Otherwise, by Lemma~\ref{lem:combminray}~(ii),
we must have $w_5 = (1,1)$.
Consider $\lambda' = \cone(w_5, w_6)$.
Then $\mu \not\in \lambda'$ holds.
Remark~\ref{rem:barXfacecrit}~(ii)
tells us that $\gamma_{5,6}$ is an $X'$-face
and hence $\lambda'$ is a GIT-cone.
The associated variety $X'$ is smooth 
according to Remark~\ref{rem:smoothfanoprops}.
Thus, Proposition~\ref{prop:qsmooth2degrees}
yields $\mu \in w_i + \lambda'$ for some 
$1 \leq i \leq 7$.
By the geometry of the possible 
degree constellations,
only $i = 1,2$ come into consideration.
We conclude 
$\mu = (e + 1, e + f)$ with $e, f \in \mathbb{Z}_{\geq 0}$.
Positive orientation of $(\mu, w_5)$ gives $f = 0$.
Hence, $\mu$ is primitive.
By Lemma~\ref{lem:raypower}~(i), this contradicts
$R_g$ being a spread hypersurface ring.

\medskip 
\noindent
\emph{Constellation~(i)}.
Let $v = (v_1, v_2)$ be the primitive generator 
of $\varrho_3 = \varrho_4$.
Due to Lemma~\ref{lem:raypower}~(ii), we have 
$\varrho_3 \in \Lambda(R_g)$ and thus also  
$\lambda' = \cone(w_3, w_7)$ is a GIT-cone.
The associated variety $X'$ is smooth
by Remark~\ref{rem:smoothfanoprops}.
Applying Lemmas~\ref{lem:threegenerate}
and~\ref{lem:2on1ray} to the triple 
$w_3, w_4, w_7$ yields $v_1 = 1$ and 
that the first coordinates 
of $w_3$, $w_4$ are coprime. 
Arguing similarly with $w_1, w_3, w_4$ 
gives $v_2 = 1$. 
So, the degree matrix has the form
\[
Q \ = \
[w_1, \dotsc ,w_7]
\ = \ 
\left[
\begin{array}{ccccccc}
1 & 1 & a & b & 0 & 0 & 0 
\\
0 & 0 & a & b & 1 & 1 & 1
\end{array}
\right], 
\qquad
a, b \in \mathbb{Z}_{\geq 1},
\quad
\gcd(a,b) = 1.
\]
We may assume $a \leq b$.
By Lemma~\ref{lem:raypower}~(i),
the relation $g$ has 
monomials of the form~$T_3^{l_3}$ 
and $T_4^{l_4}$. 
Since $\gcd(a, b) = 1$ holds,
we conclude $\mu_1 = \mu_2 = dab$ 
with $d \in \mathbb{Z}_{\geq 1}$.
In particular $\mu_1 \ge ab$ holds.
By Proposition~\ref{prop:anticanclass},
the anticanonical class is given as
\[
-\KKK_X 
\ = \ 
(2 + a + b - \mu_1, \, 3 + a + b - \mu_2).
\]
From $X$ being Fano we deduce $-\KKK_X \in \Eff(R)^\circ$,
that means that each coordinate of $-\KKK_X$ 
is positive. Thus, we obtain
\[
2 + a + b \ > \ dab \ \ge \ ab.
\]
This implies $a = 1$ or $a = 2, b = 3$.
In the case $a = 1$,
using the inequality again leads to
$3 + (1-d)b > 0$ and we end up with
possibilities
\[
b = 1, \ d = 2, 3, 
\qquad\qquad 
b = 2, \ d = 2,
\]
leading to the specifying data of 
Numbers~\ref{cand:II-1-1} to~\ref{cand:II-1-3}
of Theorem~\ref{thm:candidates}.
The constellation $a = 2, b = 3$ immediately
implies $d = 1$, which gives the specifying data 
of Number~\ref{cand:II-1-4} of Theorem~\ref{thm:candidates}.

It remains to show that these specifying data
yield Fano smooth general hypersurface Cox rings.
We work in the setting of 
Construction~\ref{constr:hypersurf} 
and treat exemplarily 
Number~\ref{cand:II-1-4}.
From Remark~\ref{rem:genCRirredundant}
we infer that for all $g \in U_\mu$
the algebra $R_g$ has $T_1, \ldots, T_7$ 
as a minimal system of generators $R_g$.
Moreover, Proposition~\ref{prop:varclassprime} 
provides us with a non-empty open subset 
$U \subseteq S_\mu$ such that $T_1, \ldots, T_7$ 
define primes in $R_g$ for all $g \in U$,
provided we deliver for each $i$ 
a $\mu$-homogeneous prime polynomial 
not depending on $T_i$. Here they are:
\[
T_3^3 - T_4^2 \enspace \text{for} \enspace i = 1,2,5,6,7, 
\qquad\qquad
T_1^6 T_5^6 - T_2^6 T_6^5 T_7 \enspace \text{for} \enspace i = 3, 4.
\]
In Construction~\ref{constr:hypersurf}, 
consider $\lambda = \cone(w_3) \in \Lambda(R_g)$.
Then $\lambda^\circ \subseteq \Mov(S)^\circ$
holds and we have $\mu \in \lambda^\circ$. 
One directly verifies that $\mu$ is 
basepoint free for $Z$.
Thus, Proposition~\ref{prop:ufdcrit}~(i) shows
that after shrinking $U$ suitably, $R_g$ 
admits unique factorization in $R_g$
for all $g \in U$.
Since $X_g$ should be a Fano variety,
$-\KKK = (1, 2)$ has to be ample
and thus we have to take 
$\lambda = \cone(w_4, w_5)$.
Then $Z_\mu$ is smooth and $\mu \in \lambda$ 
holds.
Thus, Proposition~\ref{cor:rk2bertini} shows
that after possibly shrinking $U$ again,
$X_g$ is smooth for all $g \in U$.

\medskip 
\noindent
\emph{Constellation~(ii)}.
Here we obtain $w_4 = (0,1)$ by the same arguments 
used for showing $w_5 = (0,1)$.
Write $w_3 = (a_3,b_3)$ and let $k$ be the unique 
positive integer with $\mu = k w_3$.
Then $k \ge 2$ as $R_g$ is spread and
$T_1, \ldots, T_7$ form a minimal system of
generators.
By Proposition~\ref{prop:anticanclass}, the anticanonical
class of $X = X(\lambda)$ is given as
$$
-\KKK_X
\ = \
(2 + (1-k)a_3,\; 4 + (1-k)b_3).
$$
Moreover, we have $\varrho_3 \not\in \Lambda(R_g)$ 
due to Lemma~\ref{lem:raypower}~(i) and 
Remark~\ref{rem:barXfacecrit}~(i),
the defining GIT-cone $\lambda$ of $X$ 
is the positive orthant. Thus the
Fano condition $-\KKK_X \in \lambda^\circ$ simply
means that both coordinates of $-\KKK_X$ are positive.
This leads to $a_3 = 1$, $k = 2$ and
$b_3 \leq 3$.
These are Numbers~\ref{cand:II-2-1} to~\ref{cand:II-2-3}
of Theorem~\ref{thm:candidates}.

The verification of these candidates as specifying 
data of a general smooth Fano hypersurface
Cox ring is done by the same arguments
as in Case~\ref{prop:deginrelint-collaps}~(i-a) from Part~I,
except that for Numbers~\ref{cand:II-2-2}
and~\ref{cand:II-2-3} one has to verify smoothness
of $Z_\mu$ explicitly.

\medskip \noindent
\emph{Constellation~(iii)}.
We obtain $w_3 = (1,0)$
by analogous arguments as used 
for showing $w_5 = (0,1)$ before.
The degree $w_4 = (a_4, b_4)$
has to be determined.
A suitable admissible coordinate change 
yields $a_4 \geq b_4$. 
By Proposition~\ref{prop:anticanclass} the anticanonical
class of $X = X(\lambda)$ is given as
\[
	-\KKK_X = (3 + (1-k)a_4,\; 3 + (1-k)b_4),
\]
where $k \in \ZZ_{\ge 0}$ is defined via $\mu = k w_3$.
As in the preceding constellation, we 
see that $\lambda$ is the positive orthant.
Thus, $X(\lambda)$ being Fano just means 
that both coordinates of $-\KKK_X$
are positive. 
We end up with the specifying data 
from Numbers~\ref{cand:II-3-1} to~\ref{cand:II-3-4}
of Theorem~\ref{thm:candidates}.

The verification of these candidates runs with  
the same arguments as in 
Case~\ref{prop:deginrelint-collaps}~(i-a) 
from Part~I, except that we have verify smoothness 
of $Z_\mu$ explicitly.
\end{proof}

We treat Case~\ref{rem:majorconsts}~III
that means that the degree $\mu$ of the relation 
lies in the bounding ray $\varrho_1$ of the effective cone.

\begin{lemma} 
\label{lem:freedegspos}
Let $X = X(\lambda)$ be as in Setting~\ref{rem:effrho2}
and $1 \leq i < j \leq r$
such that $g$ neither depends on $T_i$ nor on $T_j$.
If $X$ is quasismooth,
then $w_i, w_j$ lie either both in $\lambda^-$ 
or both in $\lambda^+$.
\end{lemma}

\begin{proof}
Otherwise, we may assume $w_i \in \lambda^-$ 
and $w_j \in \lambda^+$.
Then $\gamma_{i,j}$ is an $X$-face
and $\bar X(\gamma_{i,j})$ is a singular 
point of $\bar X$.
According to Proposition~\ref{prop:locprops}~(iv),
this contradicts quasismoothness of~$X$.
\end{proof}

\begin{proof}[Proof of Theorem~\ref{thm:candidates}~Part~III] 
We may assume that the ray $\varrho_1$ is generated 
by the vector~$(1,0)$.
Let $m$ be the number with 
$w_1, \ldots, w_m \in \varrho_1$ and 
$w_{m+1}, \ldots, w_7 \not\in \varrho_1$.
Observe that due to $\mu \in \varrho_1$, 
the relation $g$ only depends on 
$T_1, \ldots, T_m$.

The first step ist to show that only 
for $m=5$, the specifying data 
$w_1,\ldots, w_7$ and $\mu$ in 
$K = \ZZ^2$ allow a
hypersurface Cox ring.
Since $\mu \in \varrho_1$,
Proposition~\ref{prop:hypmov} yields
$m \ge 3$.
As $\Mov(X)$ is of dimension two,
we must have $m \le 5$; see
Setting~\ref{rem:effrho2}.
Lemma~\ref{lem:freedegspos} shows
$w_{m+1},\ldots,w_r \in \lambda^+$.
Applying Lemma~\ref{lem:threegenerate}
to triples $w_1,w_2,w_i$ for $i \ge m+1$,
we obtain
$$
\mu = (\mu_1,0),
\qquad
w_i = (a_i,0), \ i = 1,\ldots, m,
\qquad
w_i = (a_i,1), \ i = m+1 ,\ldots, 7,
$$
where, for any two $1 \le i < j  \le m$,
the numbers $a_i$ and~$a_j$ are coprime
and we may assume $a_7 = 0$.
Moreover, we must have $a_{m+1} = \ldots = a_6$,
because otherwise we obtain a GIT 
cone $\lambda \ne \eta \in \Lambda(R)$ 
with $\eta^{\circ} \in \Mov(R)^{\circ}$
and the associated variety $X(\eta)$ 
is not quasismooth by Lemma~\ref{lem:freedegspos},
contradicting Remark~\ref{rem:smoothfanoprops}.
Proposition~\ref{prop:anticanclass}
and the fact that $X$ is Fano give us 
$$ 
(a_1+ \ldots + a_6 - \mu_1, \, 7-m) 
\ = \ 
- \mathcal{K}_X
\ \in \
\lambda^\circ
\ = \ 
\cone( (1,0), \, (a_{m+1}, 1))^\circ.
$$
Since $a_1, \ldots, a_m$ are pairwise coprime,
the component $\mu_1$ of the degree of the 
relation $g$ is greater or equal to 
$a_1 \cdots a_m$.
Using moreover $a_{m+1} = \ldots = a_6$,
we derive from the above Fano condition 
$$
a_1 \cdots a_m
\ \le \ 
\mu_1 
\ < \ 
a_1 + \ldots + a_m - a_{m+1},
$$
where we may assume $a_1 \le \ldots \le a_m$.
We exclude $m=3$: here, $g = g(T_1,T_2,T_3)$,
the above inequality forces  $a_1 = a_2 = 1$,
hence $g(T_1,T_2,0)$ is classically homogeneous
and $T_3$ is not prime in $R$, a contradiction.
Let us discuss $m = 4$.
The above inequality and pairwise coprimeness 
of the $a_i$ leave us with 
$$ 
a_1 = a_2 = a_3 = 1,
\qquad \qquad
a_1 = a_2 = 1, 
\ 
a_3 = 2,
\
a_4 = 3.
$$
In the case $a_3 = 1$, we must have $\mu_1 = ka_4$
with some $k \in \ZZ_{\ge 2}$, because otherwise,
the relation would be redundant or, seen similarly 
as above, one of $T_1,T_2,T_3$ would not be 
prime in $R$. 
The inequality gives $(k-1)a_4 < 3 - a_{m+1}$.
We arrive at the following possibilities:
$$ 
a_{m+1} = a_4 = 1, 
\ 
k = 2,
\quad
a_{m+1} = 0, 
\
a_1 = 1,
\ 
k = 2,3,
\quad
a_{m+1} = 0, 
\
a_1 = k = 2.
$$
The first constellation implies that 
$R$ is not factorial and hence is 
excluded.
In the each of remaining ones,
$X$ is a product of $\PP_2$ and a surface~$Y$
which must be smooth as $X$ is so.  
Moreover, for the Picard numbers,
we have
$$
\rho(X) \ = \ \rho(\PP_2) + \rho(Y).
$$
Thus, $\rho(Y) = 1$.
Finally, being a Mori fiber, $Y$ is a del
Pezzo surface.
We arrive at $Y = \PP_2$ and hence
$X$ is toric. A contradiction to
$X$ having a hypersurface Cox ring.
We conclude that $m = 5$ is the only 
possibility.
In this case, $\lambda = \cone(w_1,w_6)$
holds and our degree matrix is of the form 
$$
Q \ = \
[w_1, \dotsc ,w_7]
\ = \ 
\left[
\begin{array}{ccccc}
a_1 & \ldots & a_5 & a_6 & 0 
\\
0  & \ldots & 0 & 1 & 1 
\end{array}
\right]
\qquad
1 \le a_1 \le  \ldots \le  a_5,
\quad 
0 \le a_6.
$$
As mentioned before, $g$ neither depends
on $T_6$ nor on $T_7$. Consequently,
we can write~$R$ as a polynomial 
ring over a $K$-graded subalgebra $R' \subseteq R$ 
as follows:
$$
R \ = \ R'[T_6,T_7],
\qquad
R' \ := \ \CC[T_1,\ldots,T_5] / \langle g \rangle.
$$
Moreover, $R'$ is $\ZZ$-graded via $\deg(T_i) := a_i$.
We claim that the $\ZZ$-graded algebra~$R'$ is a 
smooth Fano hypersurface Cox ring.
if the $K$-graded algebra $R$ is so.
First observe that $R'$ inherits the properties 
of an abstract Cox ring from $R$.
Moreover, with 
$\bar X ' = V(g) \subseteq \CC^5$,
we have $\bar X = \bar X' \times \CC^2$.
Now, the action of the one-dimensional torus 
$H' = \Spec  \, \CC[\ZZ]$ on $\bar X'$
 admits a unique projective quotient 
in the sense of Construction~\ref{constr:mds},
namely
$$
X' \ =  \ \hat X' \quot H',
\qquad\qquad
\hat X' 
\ = \ 
\bar X ' \setminus \{0\}. 
$$
Propositions~\ref{prop:mdsprops} and~\ref{prop:anticanclass} 
show that $X'$ is a Fano variety.
Observe that each $X'$-face of $\gamma_0' \preccurlyeq \gamma'$ 
of the orthant $\gamma' \subseteq \QQ^5$
defines and $X$-face $\gamma_0 = \gamma_0' + \cone(e_6,e_7)$.
In particular, using Proposition~\ref{prop:locprops}~(ii) and~(iv),
we see that $X'$ is smooth if $X$ is so.
Moreover, $R'$ is a smooth hypersurface Cox ring
if $R$ is so.
The smooth Fano threefolds with 
hypersurface Cox ring are listed in~\cite[Thm.~4.1]{DHHKL},
which gives us the possible values of $a_1,\ldots,a_5$
and from the Fano condition on $X$, we infer 
$a_6 + \mu_1 < a_1 + \ldots + a_5$.
So, we end up with the specifying data as 
in Theorem~\ref{thm:candidates}
Numbers~\ref{cand:V-1-1} to~\ref{cand:V-1-10}.
To show that these data indeed produce
smooth Fano general hypersurface Cox rings,
one proceeds by using our toolbox
in a similar way as in the previously
presented parts of the proof.
\end{proof}

\section{Birational geometry}
\label{sec:geometry}

We begin with a look at the birational geometry
of the Fano fourfolds from Theorem~\ref{thm:candidates}.
Let us briefly recall the necessary background.
Consider any $\QQ$-factorial Mori dream space 
$X = X(\lambda)$ arising from an abstract 
Cox ring $R = \oplus_K R_w$ as in 
Construction~\ref{constr:mds}.
Assume that $K_\QQ = \Cl_\QQ(X)$ is of dimension two.
Then the GIT-fan $\Lambda(R)$ looks as follows:
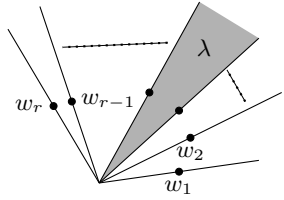
\begin{center}
\begin{tikzpicture}[scale=0.6]
\coordinate(o) at (0,0);
\coordinate(w1) at (1.75,0.25);
\coordinate(rho1) at (3.5,0.5);
\coordinate(w2) at (2,1);
\coordinate(rho2) at (4,2);
\coordinate(w-) at (1.75,1.6);
\coordinate(rho-) at (3.5,3.2);
\coordinate(w+) at (1.1,2);
\coordinate(rho+) at (2.2,4);
\coordinate(wr-1) at (-.6,1.8);
\coordinate(rhor-1) at (-1.3,3.9);
\coordinate(wr) at (-1,1.7);
\coordinate(rhor) at (-2,3.4);
\coordinate(lambda) at (2.3,3);
\coordinate(a) at (3.2,1.8);
\coordinate(b) at (2.8,2.5);
\coordinate(c) at (1.5,3.1);
\coordinate(d) at (-.8,3);
\path[fill=gray!60!] (o)--(rho+)--(rho-)--(o);
\draw[fill, color=black] (w1) circle (0.5ex) node[below]{\small{$w_1$}};
\draw[color=black] (o)--(rho1);
\draw[fill, color=black] (w2) circle (0.5ex) node[below]{\small{$w_2$}};
\draw[color=black] (o)--(rho2);
\draw[fill, color=black] (w-) circle (0.5ex);
\draw[color=black] (o)--(rho-);
\draw[fill, color=black] (w+) circle (0.5ex);
\draw[color=black] (o)--(rho+);
\draw[fill, color=black] (wr-1) circle (0.5ex) node[right]{\small{$w_{r-1}$}};
\draw[color=black] (o)--(rhor-1);
\draw[fill, color=black] (wr) circle (0.5ex) node[left]{\small{$w_{r}$}};
\draw[color=black] (o)--(rhor);
\path[fill, color=black] (lambda) circle (0.0ex)  node[]{\small{$\lambda$}};
\draw[fill, thick, color=black, dotted] (a)--(b);
\draw[fill, thick, color=black, dotted] (c)--(d);
\end{tikzpicture}   
\end{center}
where, as in Setting~\ref{rem:effrho2}, we order the 
generator degrees $w_1, \ldots, w_r \in K$ 
of $R$ counter-clockwise.
The moving cone $\Mov(X)$ is spanned by 
$w_2$ and $w_{r-1}$.
If $w_2 \in \lambda$ holds, then with 
$\tau = \cone(w_2)$ we have 
$$ 
\bar X^{ss}(\lambda) \ \subseteq \ \bar X^{ss}(\tau),
$$
which induces a morphism $\pi \colon X \to Y$ 
from $X = \bar X^{ss}(\lambda) \quot H$ onto 
$Y = \bar X^{ss}(\tau) \quot H$.
Recall that $\pi$ is an elementary 
contraction in the sense of~\cite{Ca}.
In particular, we have the following 
two possibilities:
\begin{itemize}
\item
If $w_2 \not\in \cone(w_1)$ holds, then 
$\pi \colon X \to Y$ is birational and 
contracts the prime divisor $D_1 \subseteq X$ 
corresponding to the ray through $w_1$.
In this case, we write $X \sim Y$ for the morphism
$\pi$ and denote by 
$C \subseteq Y$ the center of the contraction.
\item
$\pi \colon X \to Y$ is a proper fibration 
with $\dim(Y) < \dim(X)$.
In this case, we write $X \to Y$ for the morphism
$\pi$ and denote by 
$F \subseteq X$ the general fiber.
\end{itemize}
Similarly, if $w_{r-2} \in \Mov(X)$ holds,
we use the same notation.
In general, $\lambda$ need not to have
common rays with $\Mov(X)$.
However, given a ray $\varrho \subseteq \Mov(X)$,
we find a small quasimodification 
$X \dasharrow X'$, where $X'$ stems from 
a chamber $\lambda' \in \Lambda(R)$ 
sharing the ray $\varrho$ with $\Mov(X)$.
We then write $X' \sim Y$ or $X' \to Y$ 
etc. accordingly.

\begin{remark}
If~$X$ is as in Theorem~\ref{thm:candidates},
then $X$ admits at least one elementary 
contraction and at most one small 
quasimodification $X \dasharrow X'$.
If there is one, then $X'$ is smooth
due to Remark~\ref{rem:smoothfanoprops}.
\end{remark}

Now assume in addition that $X$ has a 
hypersurface Cox ring and consider 
the toric embedding $X = X_g \subseteq Z$ 
from Construction~\ref{constr:hypersurf}.
Given an elementary contraction of $\pi \colon X_g \to Y$,
a suitable choice of the cone $\tau$ in 
Construction~\ref{constr:hypersurf} leads 
to a commutative diagramm
$$ 
\xymatrix{
X
\ar@{}[r]|\subseteq
\ar[d]_{\pi}
&
Z
\ar[d]^{\pi_Z}
\\
Y 
\ar@{}[r]|\subseteq 
&
W
}
$$
where $\pi_Z \colon Z \to W$ is an elementary 
contraction of the ambient toric variety $Z$.
In particular, we have in this setting
that for every point $y \in Y$, the fiber
$\pi^{-1}(y) \subseteq X$ is contained 
in the fiber $\pi_Z^{-1}(y) \subseteq Z$.
This gives in particular a description 
for the general fiber $F \subseteq X$ 
as a subvariety of the general fiber 
$F_Z \subseteq Z$.

Let us fix the necessary notation to 
formulate the result.
By $Y_{d;a_1^{k_1},\ldots,a_n^{k_n}}$ 
we denote a (not necessarily general) 
hypersurface of degree~$d$ 
in the weighted projective space 
$\PP_{a_1^{k_1}, \ldots, a_n^{k_n}}$,
where, as usual, $a_i^{k_i}$ means that 
$a_i \in \ZZ_{\ge 1}$ is repeated 
$k_i$ times.
For a hypersurface of degree $d$ in 
the classical projective space
$\PP_n$ we just write~$Y_{d;n}$.
In our situation, this notation applies 
to the target spaces $Y \subseteq W$ 
in case of a birational elementary 
contraction and to the general fiber 
$F \subseteq F_Z$ in case of a 
fibration.

\begin{proposition}
\label{prop:contractions}
The subsequent table lists 
the possible elementary contractions
for~$X$ as in Theorem~\ref{thm:candidates},
where~$X$ is not a cartesian product;
the notation $Y^*$ in the context of 
a birational contraction indicates that 
the target space is singular.

\begin{center}
\small
\newcommand{\mycolumnwidth}{.485\textwidth}

\begin{minipage}[t]{\mycolumnwidth}
\begin{tabular}[t]{cCC}
\toprule
\emph{No.} & \emph{Contraction 1} & \emph{Contraction 2} \\
\midrule
\multirow{2}{*}{1} & X \rightarrow \PP_3 &  X \rightarrow \PP_2 \\
& F = Y_{1;2} & F = Y_{1;3}  \\ \midrule

\multirow{2}{*}{2} & X \rightarrow \PP_3 &  X \rightarrow \PP_2 \\
& F = Y_{1;2} & F = Y_{2;3}  \\ \midrule

\multirow{2}{*}{3} & X \rightarrow \PP_3 &  X \rightarrow \PP_2 \\
& F = Y_{1;2} &  F = Y_{3;3} \\ \midrule

\multirow{2}{*}{4} & X \rightarrow \PP_3 &  X \rightarrow \PP_2 \\
& F = Y_{2;2} &  F = Y_{1;3} \\ \midrule

\multirow{2}{*}{5} & X \rightarrow \PP_3 &  X \rightarrow \PP_2 \\
& F = Y_{2;2} &  F = Y_{2;3} \\ \midrule

\multirow{2}{*}{6} & X \rightarrow \PP_3 &  X \rightarrow \PP_2 \\
& F = Y_{2;2} &  F = Y_{3;3} \\ \midrule

\multirow{2}{*}{7} & X \rightarrow \PP_3 &  X \sim Y_{2;5}   \\
& F = Y_{1;2} &  C = \PP_1 \\ \midrule

\multirow{2}{*}{8} & X \rightarrow \PP_3 &  X \sim Y_{3;5}  \\
& F = Y_{2;2} &  C = \PP_1 \\ \midrule

\multirow{2}{*}{9} & X \rightarrow \PP_3 &  X \sim Y_{3;5}^{\ast} \\
& F = Y_{1;2} &  C = \PP_1 \\ \midrule

\multirow{2}{*}{10} & X \rightarrow \PP_3 &  X \sim Y_{4;5}^{\ast} \\
& F = Y_{2;2} &  C = \PP_1  \\ \midrule

\multirow{2}{*}{11} & X \rightarrow \PP_3 &  X \sim Y_{3;1^4,2^2}^{\ast} \\
& F = Y_{1;2} &  C = \PP_1 \\ \midrule

\multirow{2}{*}{12} & X \rightarrow \PP_3 &  X \sim Y_{5;1^4,2^2}^{\ast} \\
& F = Y_{2;2} &  C = \PP_1 \\ \midrule

\multirow{2}{*}{13} & X \rightarrow \PP_2 &  X' \rightarrow \PP_2 \\
& F = Y_{2;3} &  F = Y_{1;3} \\ \midrule

\multirow{2}{*}{14} & X \rightarrow \PP_2 &  X' \rightarrow \PP_2 \\
& F = Y_{3;3} &  F = Y_{2;3} \\ \midrule

\multirow{2}{*}{15} & X \rightarrow \PP_2 &  X' \sim Y_{4;1^5,2}  \\
& F = Y_{3;2} &  C = \PP_1 \\ \bottomrule
\end{tabular}
\end{minipage}
\begin{minipage}[t]{\mycolumnwidth}
\begin{tabular}[t]{cCC}
\toprule
\emph{No.} & \emph{Contraction 1} & \emph{Contraction 2} \\
\midrule

\multirow{2}{*}{16} & X \rightarrow \PP_3 &  X' \rightarrow \PP_1 \\
& F = Y_{1;2} &  F = Y_{2;4} \\ \midrule

\multirow{2}{*}{17} & X \rightarrow \PP_3 &  X' \rightarrow \PP_1 \\
& F = Y_{2;2} &  F = Y_{3;4} \\ \midrule

\multirow{2}{*}{18} & X \sim Y_{4;5} &  X \rightarrow \PP_2  \\
& C = \PP_2 &  F = Y_{3;3} \\ \midrule

\multirow{2}{*}{19} & X \sim Y_{3;5} &  X \rightarrow \PP_2  \\
& C = \PP_2 &  F = Y_{2;3} \\ \midrule

\multirow{2}{*}{20} & X \sim Y_{2;5} &  X \rightarrow \PP_2  \\
& C = \PP_2 &  F = Y_{1;3} \\ \midrule

\multirow{2}{*}{21} & X \sim Y_{4;1^5,2} &  X \rightarrow \PP_1  \\
& C = \PP_2 &  F = Y_{3;4} \\ \midrule

\multirow{2}{*}{22} & X' \rightarrow \PP_1 &  X \rightarrow \PP_2 \\
& F = Y_{2;4} &  F = Y_{2;3} \\ \midrule

\multirow{2}{*}{23} & X' \rightarrow \PP_1 &  X \rightarrow \PP_2 \\
& F = Y_{3;4} &  F = Y_{3;3} \\ \midrule

\multirow{2}{*}{24} & X' \rightarrow \PP_1 &  X \rightarrow \PP_2 \\
& F = Y_{4;1^4,2} &  F = Y_{4;1^3,2} \\ \midrule

\multirow{2}{*}{25} & X' \rightarrow \PP_1 &  X \rightarrow \PP_2 \\
& F = Y_{6;1^3,2,3} &  F = Y_{6;1^2,2,3} \\ \midrule

\multirow{2}{*}{26} & X \rightarrow \PP_1 &  X \rightarrow \PP_3 \\
& F = Y_{2;4} &  F = Y_{2;2} \\ \midrule

\multirow{2}{*}{27} & X \rightarrow \PP_1 &  X \rightarrow \PP_3 \\
& F = Y_{4;1^4,2} &  F = Y_{2;2} \\ \midrule

\multirow{2}{*}{28} & X \rightarrow \PP_1 &  X \rightarrow \PP_3 \\
& F = Y_{6;1^4,3} &  F = Y_{2;2} \\ \midrule

\multirow{2}{*}{29} & X \rightarrow \PP_2 &  X \rightarrow \PP_2 \\
& F = Y_{2;3} &  F = Y_{2;3} \\ \midrule

\multirow{2}{*}{30} & X \rightarrow \PP_2 &  X \rightarrow \PP_2 \\
& F = Y_{3;3} &  F = Y_{3;3} \\ \bottomrule

\end{tabular}
\end{minipage}

\begin{minipage}[t]{\mycolumnwidth}
\begin{tabular}[t]{cCC}
\toprule
\emph{No.} & \emph{Contraction 1} & \emph{Contraction 2} \\
\midrule

\multirow{2}{*}{31} & X \rightarrow \PP_2 &  X \rightarrow \PP_2 \\
& F = Y_{2;3} &  F = Y_{4;1^3,2} \\ \midrule

\multirow{2}{*}{32} & X \rightarrow \PP_2 &  X \rightarrow \PP_2 \\
& F = Y_{4;1^3,2} &  F = Y_{4;1^3,2} \\ \midrule

\multirow{2}{*}{33} & X' \sim Y_{6;1^4,2,3}^{\ast} &  X \rightarrow \PP_2 \\
& C = \{\mathrm{pt}\} &  F = Y_{4;1^3,2} \\ \midrule

\multirow{2}{*}{34} & X \sim Y_{2;5} &  X \rightarrow \PP_1  \\
& C = \PP_1 \times \PP_1 &  F = Y_{2;4}\\ \midrule

\multirow{2}{*}{35} & X \sim Y_{3;5} &  X \rightarrow \PP_1  \\
& C = Y_{3;3} &   F = Y_{3;4} \\ \midrule

\multirow{2}{*}{36} & X \sim Y_{4;5} &  X \rightarrow \PP_1  \\
& C = Y_{4;3} &   F = Y_{4;4} \\ \midrule

\multirow{2}{*}{37} & X \sim Y_{4;1^5,2} &  X \rightarrow \PP_1  \\
& C = Y_{4;1^3,2} &   F = Y_{4;1^4,2}\\ \midrule

\multirow{2}{*}{38} & X \sim Y_{6;1^5,3} &  X \rightarrow \PP_1 \\
& C = Y_{6;1^3,3} &  F = Y_{6;1^4,3} \\ \midrule

\multirow{2}{*}{39} & X \sim Y_{6;1^4,2,3} &  X \rightarrow \PP_1  \\
& C = Y_{6;1^2,2,3} &  F = Y_{6;1^3,2,3} \\ \midrule

\multirow{2}{*}{40} & X \sim Y_{2;5} &  X \rightarrow \PP_2  \\
& C = \PP_1 &  F = Y_{2;3} \\ \midrule

\multirow{2}{*}{41} & X \sim Y_{3;5} &  X \rightarrow \PP_2  \\
& C = Y_{3;2} &  F = Y_{3;3} \\ \midrule

\multirow{2}{*}{42} & X \sim Y_{4;1^3,2^3}^{\ast} & X \rightarrow \PP_2  \\
& C = \PP_1 &   F = Y_{2;3} \\ \midrule

\multirow{2}{*}{43} & X \sim Y_{6;1^2,2^3}^{\ast} & X \rightarrow \PP_2  \\
& C = Y_{3;2} &  F = Y_{3;3} \\ \midrule

\multirow{2}{*}{44} & X \sim Y_{4;1^5,2} &  X \rightarrow \PP_2  \\
& C = Y_{4;1^2,2} &   F = Y_{4;1^3,2}\\ \midrule

\multirow{2}{*}{45} & X \sim Y_{8;1^3,2^2,4}^{\ast} & X \rightarrow \PP_2  \\
& C = Y_{4;1^2,2} &  F = Y_{4;1^3,2} \\ \midrule

\multirow{2}{*}{46} & X \sim Y_{6;1^4,2,3} &  X \rightarrow \PP_2  \\
& C = Y_{6;1,2,3} &  F = Y_{6;1^2,2,3} \\ \bottomrule

\end{tabular}
\end{minipage}
\begin{minipage}[t]{\mycolumnwidth}
\begin{tabular}[t]{cCC}
\toprule
\emph{No.} & \emph{Contraction 1} & \emph{Contraction 2} \\
\midrule

\multirow{2}{*}{47} & X \sim Y_{12;1^3,2,4,6}^{\ast} & X \rightarrow \PP_2 \\
& C = Y_{6;1,2,3} &  F = Y_{6;1^2,2,3} \\ \midrule

\multirow{2}{*}{48} & X \sim \PP_4 &  X' \rightarrow \PP_1  \\
& C = \PP_1 &  F = Y_{2;4} \\ \midrule

\multirow{2}{*}{49} & X \sim Y_{6;1^2,2^3,3}^{\ast} &  X' \rightarrow \PP_1  \\
& C = Y_{6;2^3,4} &  F = Y_{3;4} \\ \midrule

\multirow{2}{*}{50} & X \sim Y_{4;1^5,2}^{\ast} & X \rightarrow \PP_2  \\
& C = \PP_1 &  F = Y_{2;3} \\ \midrule

\multirow{2}{*}{51} & X \sim Y_{6;1^5,3}^{\ast} & X \rightarrow \PP_2 \\
& C = \PP_1 &  F = Y_{4;1^3,2} \\ \midrule

\multirow{2}{*}{52} & X \sim Y_{6;1^4,2,3}^{\ast} & X \rightarrow \PP_1  \\
& C = \PP_1 &   F = Y_{4;1^4,2}\\ \midrule

\multirow{2}{*}{53} & X \sim \PP_4 &  X \sim Q_4    \\
& C = \PP_1 \times \PP_1 &  C = \{\mathrm{pt}\} \\ \midrule

\multirow{2}{*}{54} & X \sim Y_{4;1^5,2} &  X \sim Y_4    \\
& C = Y_{4;3} &  C = \{\mathrm{pt}\} \\ \midrule

\multirow{2}{*}{55} & X \sim \PP_4 &  X \sim Y_{3;1^5,2}^{\ast} \\
& C = Y_{3;3} &  C = \{\mathrm{pt}\} \\ \midrule

\multirow{2}{*}{56} & X \sim \PP_4 &  X \sim Y_{3;1^5,3}  \\
& C = Y_{4;3} &  C = \{\mathrm{pt}\} \\ \midrule

\multirow{2}{*}{57} & X \sim Y_{6;1^4,2,3} &  X \sim Y_{6;1^5,3}  \\
& C = Y_{4;1^3,3} &  C = \{\mathrm{pt}\} \\ \midrule

\multirow{2}{*}{60} & X \rightarrow Y_{6;1^3,2,3} &  X \sim Y_{6;1^4,2,3}^{\ast} \\
& F = \PP_1 &  C = \{\mathrm{pt}\} \\ \midrule

\multirow{2}{*}{62} & X \rightarrow Y_{4;1^4,2} &  X \sim Y_{4;1^5,2}^{\ast} \\
& F = \PP_1 &  C = \{\mathrm{pt}\} \\ \midrule

\multirow{2}{*}{64} & X \rightarrow Y_{3;4} &  X \sim Y_{3;5}^{\ast} \\
& F = \PP_1 &  C = \{\mathrm{pt}\} \\ \midrule

\multirow{2}{*}{66} & X \rightarrow Y_{2;4} &  X \sim Y_{2;5}^{\ast} \\
& F = \PP_1 &  C = \{\mathrm{pt}\} \\ \midrule

\multirow{2}{*}{67} & X \rightarrow Y_{2;4} &  X \sim Y_{2;1^5,2}^{\ast} \\
& F = \PP_1 &  C = \{\mathrm{pt}\} \\ 
\bottomrule
\end{tabular}
\end{minipage}

\end{center}

\bigskip
\noindent
The remaining families of Theorem~\ref{thm:candidates} consist
of cartesian products $Y \times \PP_1$ where
the first factor $Y$ is a smooth threedimensional Fano hypersurface
of Picard number one as displayed in the following table.
\begin{table}[H]
\begin{tabular}{CCCCCC} \toprule
\text{No.} & 58 & 59 & 61 & 63 & 65 \\ \midrule
Y & Y_{6;1^4,3}
& Y_{6;1^3,2,3}
& Y_{4;1^4,2}
& Y_{3;4}
& Y_{2;4} \\ 
\bottomrule
\end{tabular}
\end{table}
\end{proposition}

The proof of this proposition is basically 
a case by case analysis of the contraction 
maps in coordinates.
We restrict ourselves to perform this in the 
subsequent remark for one case, where we even 
go a bit deeper into the matter and specify 
also the singular fibers of the fibration.

\begin{remark}
\label{rem:geomdisc}
We take a closer look at the varieties $X$ 
from No.~9 of Theorem~\ref{thm:candidates}.
In this case the specifying data, that means the degree 
matrix $Q$ and the degee $\mu$ of the relation $g$, are
given by
\[
Q = \begin{bmatrix*}[r]
1 & 1 & 1 & 1 & 0 & 0 & -1 
\\
0 & 0 & 0 & 0 & 1 & 1 &  1
\end{bmatrix*}, 
\qquad
\mu = (2, 1).
\]
Due to $-\KKK = (1,2)$, we have $\lambda = \cone(w_1, w_5)$.
Observe that $\Mov(R)$ and $\lambda$ share the rays
$\varrho_1$ and $\varrho_5$. Thus $X$ admits
two elementary contractions $\pi_1 \colon X \rightarrow Y_1$
and $\pi_2 \colon X \rightarrow Y_2$ associated to $\varrho_1$
resp. $\varrho_5$.
To study $\pi_1$ and $\pi_2$ we make use of
the toric embedding
$X = X_g \subseteq Z$
from Construction~\ref{constr:hypersurf}.

First, we discuss $\pi_1$. Since $w_2 \in \varrho_1$ 
holds, the morphism $\pi_1$ is a fibration.
Moreover, $\pi_1$ is the restriction of
the corresponding ambient toric elementary 
contraction $\pi_{1,Z}$ of $Z$, 
which in turn is explicitly given 
as follows:
$$
\xymatrix@C=1.7em{
\bar{X} 
\ar@{}[r]|{\subseteq} 
\ar@{-->}[d] 
& 
\KK^7 
\ar@{-->}[d] 
\ar[rrrr]^{(z_1, \dotsc, z_7) \mapsto (z_1, \dotsc, z_4)} 
&&&& 
\KK^4 \ar@{-->}[d] 
\\
X \ar@{}[r]|{\subseteq}  
& 
Z \ar[rrrr]^{\pi_{1,Z}} 
&&&& 
\PP_3
} 
$$
Suitably sorting the terms of $g$ yields
a presentation $g = q_1 T_5 + q_2 T_6 + f T_7$
where $q_1, q_2 \in \KK[T_1, \dotsc, T_4]$ both are
quadrics and $f \in \KK[T_1, \dotsc, T_4]$ is a
cubic, each of which is general.
Note that $V(g) \subseteq \KK^7$
projects onto $\KK^4$ thus $Y_1 = \PP_3$.
For any point $y = [y_1, \dotsc, y_4] \in \PP_3$
the fiber $\pi_{1,Z}^{-1}(y)$ of the ambient toric
variety is given by
the equations
\[
y_2 T_1 - y_1 T_2 
\ = \
y_3 T_2 - y_2 T_3 
\ = \
y_4 T_3 - y_3 T_4 
\ = \ 
0.
\]
Besides we have $y_i \neq 0$ for some $i$.
Taking this into account
one directly checks $\pi_{1,Z}^{-1}(y) \cong \PP_2$.
Being homogeneous $g$ is compatible with this isomorphism,
thereby we obtain
\[
\pi_1^{-1}(y) 
\ \cong \ 
V(y_i q_1(y) T_0 + y_i q_2(y) T_1 + f(y) T_2) 
\ \subseteq \ 
\PP_2.
\]
We conclude that the general fiber $\pi_1^{-1}(y)$ is isomorphic
to $\PP_1$. In addition, $V(q_1, q_2, f) \subseteq \PP_3$ 
consists of precisely 12 points
$p_1, \dotsc, p_{12}$,
each of which has fiber $\pi_1^{-1}(p_i) \cong \PP_2$.

We turn to $\pi_2$.
From $w_7 \notin \varrho_5$ follows that 
$\pi_2$ is a birational morphism contracting
the prime divisor $V(T_7) \subseteq X$.
The according elementary contraction $\pi_{2,Z}$ 
of the ambient toric variety $Z$
is the blow-up of $\PP_5$ along 
$C = V_{\PP_5}(T_0, \dotsc, T_3) \cong \PP_1$. 
The situation is as in the subsequent diagram:
\[ 
\xymatrix@C=1.9em{
\bar{X} 
\ar@{}[r]|{\subseteq} 
\ar@{-->}[d] 
& 
\KK^7 
\ar@{-->}[d] 
\ar[rrrrr]^{(z_1, \dotsc, z_7) \mapsto (z_1 z_7, \dotsc, z_4 z_7, z_5, z_6)} 
&&&&& 
\KK^6 
\ar@{-->}[d] 
\\
X \ar@{}[r]|{\subseteq}  
& 
Z \ar[rrrrr]^{\pi_{2,Z}} 
&&&&& 
\PP_5
} 
\]
The target variety $Y_2 \subseteq \PP_5$
of $\pi_2$ is $V(g') \subseteq \PP_5$
where
$g' = g(T_0, \dotsc, T_6, 1)$.
From this we infer $C \subseteq Y_2$,
so $C$ is the center of $\pi_2$ as well.
In particular $\pi_2$ is the blow-up of
$Y_2$ along $C$.
Moreover, the polynomial $g'$ is an irreducible cubic
living in ${\langle T_0, \dotsc, T_3 \rangle}^2$.
Consequently, $Y_2$ is singular
at every point of $C$.
\end{remark}

\section{Hodge numbers}

Here we determine the Hodge numbers of the
Fano fourfolds from Theorem~\ref{thm:candidates}.
First, we note the following simple
observation.

\begin{proposition}
\label{prop:hodgediamond}
Let $X$ be a smooth projective 
Fano fourfold of Picard rank $2$.
Then the Hodge diamond of $X$ 
is the following.
\[
 \begin{array}{ccccccccc}
  &&&&1&&&&\\
  &&&0&&0&&&\\
  &&0&&2&&0&&\\
  &0&&h^{1,2}&&h^{2,1}&&0&\\
  0&&h^{1,3}&&h^{2,2}&&h^{3,1}&&0\\
  &0&&h^{3,2}&&h^{2,3}&&0&\\
  &&0&&2&&0&&\\
  &&&0&&0&&&\\
  &&&&1&&&&
 \end{array}
\]
\end{proposition}

\begin{proof}
Ampleness of $-K_X$ and 
Kawamata-Viehweg vanishing
give $h^{p,0}(X) = 0$ for any $p>0$.
Moreover, plugging
$H^i(X,\mathcal O) = 0$ for $i = 1,2$
into the cohomology sequence associated
with the exponential sequence yields
$H^2(X,\mathbb C) \cong \CC^2$.
The Hodge decomposition together with
$h^{1,0}(X) = h^{0,1}(X) = 0$
shows $h^{1,1}(X) = 2$.
\end{proof}

By symmetry, we are left with computing
the Hodge numbers $h^{2,1}$, 
$h^{3,1}$ and $h^{2,2}$.
Here comes our result.

\begin{proposition}
\label{prop:hodgenum}
The subsequent table lists 
the Hodge numbers $h^{2,1}$, 
$h^{3,1}$ and $h^{2,2}$
for~$X$ as in Theorem~\ref{thm:candidates}.

\begin{table}[H]
\centering \upshape \small
\begin{subtable}[t]{0.32\linewidth}
\centering
\begin{tabular}[t]{c|c|c|c}
{\it No.} & $h^{2,1}$ & $h^{3,1}$ & $h^{2,2}$\\
\hline
1 & 0 & 0 & 3\\
2 & 0 & 0 & 10\\
3 & 0 & 0 & 29\\
4 & 0 & 0 & 3\\
5 & 0 & 3 & 40\\
6 & 0 & 30 & 185\\
7 & 0 & 0 & 4\\
8 & 0 & 1 & 23\\
9 & 0 & 0 & 14\\
10 & 0 & 18 & 126\\
11 & 0 & 0 & 5\\
12 & 0 & 12 & 95\\
13 & 0 & 0 & 4\\
14 & 0 & 6 & 65\\
15 & 0 & 5 & 55\\
16 & 0 & 0 & 6\\
17 & 0 & 9 & 77\\
18 & 0 & 21 & 143\\
19 & 0 & 1 & 22\\
20 & 0 & 0 & 3\\
21 & 0 & 5 & 53\\
22 & 0 & 0 & 10\\
\end{tabular}
\end{subtable}
\begin{subtable}[t]{0.32\linewidth}
\centering
\begin{tabular}[t]{c|c|c|c}
{\it No.} & $h^{2,1}$ & $h^{3,1}$ & $h^{2,2}$\\
\hline
23 & 0 & 13 & 103\\
24 & 0 & 35 & 218\\
25 & 0 & 114 & 591\\
26 & 0 & 0 & 10\\
27 & 0 & 20 & 138\\
28 & 0 & 112 & 570\\
29 & 0 & 1 & 22\\
30 & 0 & 45 & 255\\
31 & 0 & 10 & 94\\
32 & 0 & 100 & 508\\
33 & 0 & 24 & 162\\
34 & 0 & 0 & 4\\
35 & 0 & 1 & 28\\
36 & 0 & 22 & 162\\
37 & 0 & 5 & 60\\
38 & 0 & 71 & 402\\
39 & 0 & 24 & 170\\
40 & 0 & 0 & 4\\
41 & 1 & 1 & 23\\
42 & 0 & 0 & 10\\
43 & 1 & 19 & 131\\
44 & 1 & 5 & 54\\
\end{tabular}
\end{subtable}
\begin{subtable}[t]{0.32\linewidth}
\centering
\begin{tabular}[t]{c|c|c|c}
{\it No.} & $h^{2,1}$ & $h^{3,1}$ & $h^{2,2}$\\
\hline
45 & 1 & 50 & 288\\
46 & 1 & 24 & 163\\
47 & 1 & 159 & 793\\
48 & 0 & 0 & 3\\
49 & 1 & 2 & 31\\
50 & 0 & 3 & 40\\
51 & 0 & 65 & 356\\
52 & 0 & 20 & 139\\
53 & 0 & 0 & 3\\
54 & 0 & 6 & 72\\
55 & 0 & 0 & 8\\
56 & 0 & 1 & 21\\
57 & 0 & 25 & 181\\
58 & 52 & 0 & 2\\
59 & 21 & 0 & 2\\
60 & 21 & 0 & 2\\
61 & 10 & 0 & 2\\
62 & 10 & 0 & 2\\
63 & 5 & 0 & 2\\
64 & 5 & 0 & 2\\
65 & 0 & 0 & 2\\
66 & 0 & 0 & 2\\
67 & 0 & 0 & 2\\
\end{tabular}
\end{subtable}
\end{table}
\end{proposition}

\begin{proof}
We consider the toric embedding
$X = X_g \subseteq Z_g$ as 
provided by Construction~\ref{constr:hypersurf}.
The five-dimensional toric ambient
variety $Z_g$ is smooth and the 
decomposition
$$
X \ = \ \bigcup_{\gamma_0 \in \rlv(X)} X(\gamma_0)
$$
from Construction~\ref{rem:Coxconst} is
obtained by cutting down the toric
orbit decomposition of~$Z_g$.
Now the idea is to compute the Hodge numbers
in question via the Hodge-Deligne
polynomial, being defined for any variety
$Y$ as
\[
e(Y)
\ := \
\sum_{p,q}e^{p,q}(Y)x^p\bar x^q
\ \in\
\ZZ [x,\bar x],
\]
with $e^{p,q}(Y)$ as in~\cite{dk}*{p.~280}.
We also write $e^{p,q}$ instead of $e^{p,q}(Y)$.
Recall that $e^{p,q} = e^{q,p}$ holds. 
Moreover, in case that $Y$ is smooth and 
projective, the $e^{p,q}$ are related to the 
Hodge numbers as follows:
$$
e^{p,q}(Y)
\ = \
(-1)^{p+q}h^{p,q}(Y).
$$
The Hodge-Deligne polynomial
is additive on disjoint unions,
multiplicative on cartesian products.
We list the neccessary steps for 
computing it in low dimensions.
On $Y = \CC^*$, it evaluates to $x \bar x-1$.
For a hypersurface $Y \subseteq (\mathbb C^*)^n$ 
with no torus factors, one has the 
Lefschetz type formula
$$
e^{p,q}(Y)
\ = \
e^{p+1,q+1}((\CC^*)^n),
\qquad
\text{for } p+q > n-1,
$$
see~\cite{dk}*{p.~290}.
Moreover, according to~\cite{dk}*{p.~291},
with the Newton polytope $\Delta$ of the
defining equation of $Y$, one has the
following identity
$$
\sum_{q\geq 0}e^{p,q}(Y)
\ = \ 
 (-1)^{p+n-1}\binom{n}{p+1}
 +(-1)^{n-1}\varphi_{n-p}(\Delta),
$$
where, denoting by $l^*(B)$ the  
number of interior points of a polytope 
$B$, the function~$\varphi_i$ is defined as 
$$
\varphi_0(\Delta)
\ := \ 0,
\qquad
\varphi_i(\Delta)
\ := \ 
\sum_{j=1}^i(-1)^{i+j} \binom{n+1}{i-j}l^*(j\Delta),
$$
This leads to an explicit formula
for all $e^{p,0}(Y)$.
Moreover, for $\dim(Y) \le 3$, all the numbers
$e^{p,q}$ are directly calculated using
the above formulas. 
For $\dim(Y) = 4$, the values of
$e^{1,1}+e^{1,2}+e^{1,3}$
and $e^{2,1}+e^{2,2}$
and $e^{3,1}$ can be
directly computed using the 
above formulas. 
By the symmetry $e^{p,q} = e^{q,p}$
these sums involve just 
four numbers which thus can be 
expressed in terms of one of them, 
say $e^{1,2}$, plus known 
quantities.
To determine the value of $e^{1,2}$ 
one passes to a smooth compactification
$Y'$ of $Y$ for which
$$
e^{1,2}(Y')
\ = \
-h^{1,2}(Y')
\ = \
-h^{3,2}(Y')
\ = \
e^{3,2}(Y')
$$
holds by Serre's duality
and then observes that $e^{3,2}$
can be computed for all the strata
via the Lefschetz formula.
Now, we apply these principles to the
strata $Y = X(\gamma_0)$ that have
no torus factor and compute the
desired $e^{p,q}$.
If  $Y = X(\gamma_0)$ hase torus
factor, then we use multiplicativity
of the Hodge-Deligne polynomial and  
again the above principles.
\end{proof}

Finally, we extend the discussion
of the varieties $X$ from Number~9
of Theorem~\ref{thm:candidates}
started in Remark~\ref{rem:geomdisc}
by some topological aspects.

\begin{remark}
\label{rem:no9again}
Let $X$ be as in Theorem~\ref{thm:candidates}, No.~9.
Recall that we have a fibration
$X \rightarrow \PP_3$ with general
fiber $F = \PP_1$ and
precisely 12 special fibers
$F_1, \dotsc, F_{12}$, lying
over $p_1, \ldots, p_{12} \in \PP_3$,
each of the $F_i$ being isomorphic
to $\PP_2$.
We claim
$$
F_i^2 \ = \ 1
\quad
\text{for }
i = 1, \ldots 12,
\qquad\qquad
F_i \cdot F_j = 0
\quad
\text{for }
1 \le i <  j \le 12.
$$
The second part is clear because of
$F_i$ and $F_j$ do not intersect
for $i < j$.
In order to establish the first part,
we show $F_1^2 = 1$, where we
may assume $p_1 = [1,0,0,0]$.
Consider the zero sets
$L_1, L_2 \subseteq X$
of two general polynomials
in the variables $T_2,T_3,T_4$.
By definition $L_1\cap L_2 = F$  
and $L_1\sim L_2$, that is
the two surfaces are rationally 
equivalent.
Thus $L_i \sim F + S_i$
for some surface $S_i$.
Observe that we have
$$
F\cdot L_i \ = \ 0,
\qquad\qquad
S_i \cdot L_i \ = \ 0
$$
because $L_i$ is rationally equivalent 
to a complete intersection of 
two general polynomials in $T_1,\dots,T_4$, 
which has empty intersection
with $L_i$.
We deduce
\[
F^2 
\  = \ 
-F\cdot S_1
\ = \ 
S_1\cdot S_1
\ = \
S_1\cdot S_2,
\]
using $S_1\sim S_2$ in the last step.
For computing the last intersection 
number, we may assume $L_1 = V(T_2,T_3,g)$
and $L_2 = V(T_2,T_4,g)$.
Then $S_1 = V(T_2,T_3,h_1)$ with
$$
h_1 
\ = \ 
T_4^{-1}(q_1(T_1,0,0,T_4)T_5+q_2(T_1,0,0,T_4)T_6+
f(T_1,0,0,T_4)T_7),
$$ 
where the division by $T_4$ can be performed
because by hypothesis $q_1,q_2$ and
$f$ do not contain a pure power of $T_1$.
Similarly $S_2 = V(T_2,T_4,h_2)$,
where 
$$
h_2 
\ = \ 
T_3^{-1}(q_1(T_1,0,T_3,0)T_5+q_2(T_1,0,T_3,0)T_6+
f(T_1,0,T_3,0)T_7).
$$
It follows that
\begin{align*}
S_1\cap S_2 & = 
V(T_2,T_3,T_4,\alpha_1T_1T_5
+\alpha_2T_1T_6+\alpha_3T_1^2T_7,
\beta_1T_1T_5
+\beta_2T_1T_6+\beta_3T_1^2T_7)\\
&= V(T_2,T_3,T_4,\alpha_1T_5
+\alpha_2T_6+\alpha_3T_1T_7,
\beta_1T_5
+\beta_2T_6+\beta_3T_1T_7).
\end{align*}
Now one directly checks that $S_1 \cap S_2$
is a point and the intersection is transverse.
Thus, we arrive at $S_1 \cdot  S_2 = 1$,
proving the $F_1^2 = 1$.
Now, fix two general linear forms
$\ell_1,\ell_2\in\mathbb C[T_1,\dots,T_4]$
and set 
$$
E 
\ := \ 
V(T_6,T_7,g) 
\ \subseteq \ 
X,
\qquad\qquad
L 
\ := \ 
V(\ell_1,\ell_2,g)
\ \subseteq \
X.
$$
We claim that the classes of
$E,L,F_1,\ldots,F_{12}$
in $H^{2,2}(X) \cap H^4(X,\mathbb Q)$
are linearly independent.
First observe that $F_1,\dots,F_{12}$
are linearly independent:
passing to the self-intersection,
$\sum_i a_iF_i \sim 0$ turns into
$\sum_i a_i^2=0$ and thus, 
being rational numbers,
all $a_i$ vanish.
Now, by definition of $L$ one has 
$L^2 = L\cdot F_i = 0$
for any~$i$, in particular the class 
of $L$ cannot be in the linear span of
the classes of the 12 fibers.
The statement then follows from
$E\cdot L = 2$, which in turn holds
due to
$$
E\cap L 
\ = \
V(\ell_1,\ell_2,T_6,T_7,g)
\ = \
V(\ell_1,\ell_2,T_6,T_7,q_1T_5)
\ = \
V(\ell_1,\ell_2,T_6,T_7,q_1).
$$
Combining linear
independence of $E,L,F_1,\ldots,F_{12} \in
H^{2,2}(X) \cap H^4(X,\mathbb Q)$
with $h^{2,2}(X) = 14$ as provided by
Proposition~\ref{prop:hodgenum},
we retrieve that the varieties $X$ from
Number~9 of Theorem~\ref{thm:candidates}
satisfy the Hodge Conjecture;
which, in this case, is known to hold
also by~\cite{CM}
and~\cite[Proof of Lemma~15.2]{Vo}.
\end{remark}

\section{Deformations and automorphisms}
\label{sec:defaut}

We take a look at the deformations
of the varieties from Theorem~\ref{thm:candidates}.
For any variety $X$, we denote by $\mathcal{T}_X$ its
tangent sheaf.
If $X$ is Fano, then it is unobstructed and thus
its versal deformation space is of dimension
$h^1(X,\mathcal{T}_X)$.
The following observation makes precise how
the problem of determining $h^1(X,\mathcal{T}_X)$
is connected with determining the automorphisms
in our setting.

\begin{proposition}
\label{prop:infdef}
Let $X$ be a smooth Fano variety $X$
with a general hypersurface Cox ring
$\mathcal{R}(X) = \CC[T_1,\ldots,T_r] / \bangle{g}$
and associated minimal toric embedding $X \subseteq Z$.
Assume that $\mu = \deg(g) \in \Cl(Z)$ is
base point free and no
$w_i = \deg(T_i) \in \Cl(Z)$ lies in
$\mu + \ZZ_{\ge 0} w_1 + \ldots + \ZZ_{\ge 0} w_r$.
Then we have
\begin{eqnarray*}
h^1(X,\mathcal{T}_X)
& = &
\dim(\mathcal{R}(Z)_\mu) - 1
+
\rk(\Cl(Z)) - \sum_{i=1}^r \dim(\mathcal{R}(Z)_{w_i}) 
+
h^0(X,\mathcal{T}_X)
\\[.5ex]
& = &
-1 + \dim(\mathcal{R}(Z)_\mu)
- \dim(\Aut(Z))
+ \dim(\Aut(X)).     
\end{eqnarray*}
\end{proposition}

\begin{proof}
First look at
$0 \to \mathcal{T}_X \to \imath^*\mathcal{T}_Z \to \mathcal{N}_{X} \to 0$,
the normal sheaf sequence for the inclusion $\imath \colon X \subseteq Z$.
By assumption, $\mu - \mathcal{K}_X$ is ample and thus we obtain
$$
h^1(X,\mathcal{T}_X)
- h^0(X,\mathcal{T}_X)
\ = \ 
- h^0(X, \imath^*\mathcal{T}_Z)
+ h^0(X,\mathcal{N}_X)
+ h^1(X, \imath^* \mathcal{T}_Z),
$$
according to the Kawamata-Viehweg vanishing theorem.
The task is to evaluate the right hand side.
First, note that we have 
$$
h^0(X,\mathcal{N}_X)
\ = \
\dim(\mathcal{R}(X)_\mu)
\ = \
\dim(\mathcal{R}(Z)_\mu) - 1.
$$
For the remaining two terms, we use the Euler sequence
of $Z$ restricted to $X$ which in our setting is given by
$$
\xymatrix{
0
\ar[r]
&
{\mathcal{O}_X \otimes \Cl(Z)}
\ar[r]
&
{\bigoplus_{i=1}^r \mathcal{O}_X(D_i)}
\ar[r]
&
{\imath^* \mathcal{T}_Z}
\ar[r]
&
0
,
}
$$
where $D_i \subseteq X$ denotes the prime divisor defined
by the Cox ring generator $T_i$.
Since $X$ is Fano, $h^i(X,\mathcal{O}_X)$ vanishes
for all $i > 0$. As first consequence, we obtain
$$
h^0(X, \imath^*\mathcal{T}_Z)
\ = \ 
\sum_{i=1}^r \dim(\mathcal{R}(X)_{w_i})  - \rk(\Cl(Z)) 
\ = \ 
\sum_{i=1}^r \dim(\mathcal{R}(Z)_{w_i})  - \rk(\Cl(Z)),
$$
using $\mathcal{R}(X)_{w_i} \cong H^0(X,D_i)$
and $\mathcal{R}(X)_{w_i} =  \mathcal{R}(Z)_{w_i}$,
where the latter holds by assumption.
Moreover, we can conclude
$$
h^1(X, \imath^*\mathcal{T}_Z)
\ = \ 
\sum_{i=1}^r  h^1(X, D_i).
$$
We evaluate the right hand side.
Since $X$ has a general hypersurface
Cox ring, $Z$ is smooth~\cite[Prop.~3.3.1.12]{ADHL}
and $\mu$ is base point free, we can infer
smoothness of
$$
D_i \ = \ V(g) \cap V(T_i) \ \subseteq \ Z
$$
from Bertini's theorem.
Now choose $\varepsilon > 0$ such that
$\varepsilon D_i - \mathcal{K}_X$ is
nef and big. 
Then, using once more the Kawamata-Viehweg
vanishing theorem, we obtain
$$
h^1(X, D_i)
\ = \
h^1(X, \, \mathcal{K}_X
+ (\varepsilon D_i - \mathcal{K}_X)
+ (1-\varepsilon) D_i)
\ = \
0.
$$
Consequently, $h^1(X, \imath^*\mathcal{T}_Z)$ vanishes.
This gives the first equality of the assertion.
The second one follow from~\cite[Thm.~4.2]{Cox}
and~\cite[Lemma~3.4]{MaOo}.
\end{proof}

Observe that Proposition~\ref{prop:infdef} applies
in particular to all smooth Fano non-degenerate
toric hypersurfaces in the sense of
Khovanskii~\cite{Ho} and~\cite[Def.~4.1]{HaMaWr},
where Lemma~3.3~(v) of the latter guarantees
base point freeness of $\mu \in \Cl(Z)$.
Concerning the varieties from Theorem~\ref{thm:candidates},
we can say the following.

\begin{corollary}
\label{cor:infdef}
For each of the Fano varieties $X$ listed 
in Theorem~\ref{thm:candidates}, except
possibly numbers 13, 14, 15, 33 and 67,
we have
$$
h^1(X,\mathcal{T}_X)
\ = \ 
-1 + \dim(\mathcal{R}(Z)_\mu)
- \dim(\Aut(Z))
+ \dim(\Aut(X)).     
$$ 
\end{corollary}

\begin{proof}
Using~\cite[Prop.~3.3.2.8]{ADHL} one directly checks
that $\mu \in \Cl(X)$ and hence also $\mu \in \Cl(Z)$ are
base point free in all cases except the Numbers 13, 14, 15
and 33. Number~67 violates the assumption on the
generator degrees.
\end{proof}

The only serious task left open by
Proposition~\ref{prop:infdef}
for explicitly computing
$h^1(X, \mathcal{T}_X)$
is to determine the dimension of
$\Aut(X)$.
As general tools,
we mention~\cite[Thm.~4.4]{HaKeWo},
the algorithms presented
thereafter and their
implementation provided by~\cite{Ke}.
The subsequent example discussions
indicate how one might proceed in concrete
cases.

\begin{example}
The variety $X$ from No.~65 is a product
of the smooth projective quadric
$Q_4 \subseteq \PP_4$ and a projective 
line. So, $X$ is known to be infinitesimally
rigid.
Via Proposition~\ref{prop:infdef}, this  
is seen as follows:
\begin{eqnarray*}
h^1(X,\mathcal{T}_X)
& = & 
-1 + \dim(\mathcal{R}(Z)_\mu)
- \dim(\Aut(Z))
+ \dim(\Aut(X))  
\\
& = &   
-1 + 15 - 27 + 13
\\
& = & 
0.
\end{eqnarray*}
All ingredients are classical:
First, by~\cite[Cor.~I.2]{Bl} the unit
component of the automorphism group
of a product is the product of the unit
components of the respective automorphism
groups. 
Second, $\Aut(Q_n) = \mathrm{O}(n)$
is of dimension $n(n-1)/2$.
\end{example}

\begin{example}
For the varieties $X$ from No.~1, the
algorithm~\cite{Ke} is feasible and
tells us that $\Aut(X)$ is of
dimension~12. In particular,
we see that also these varieties
are infinitesimally rigid:
\begin{eqnarray*}
h^1(X,\mathcal{T}_X)
& = & 
-1 + \dim(\mathcal{R}(Z)_\mu)
- \dim(\Aut(Z))
+ \dim(\Aut(X))  
\\
& = &
-1 + 12 - 23 + 12
\\
& = &
0.
\end{eqnarray*}
In suitable linear coordinates 
respecting the grading,
$g = T_1T_5+T_2T_6+T_3T_7$ holds
and the automorphisms on $X$ are
induced by the five-dimensional diagonally
acting torus respecting $g$ and
the group $\GL(3)$ acting on
$\mathcal{R}(X)_{w_1} \oplus \mathcal{R}(X)_{w_5}$
via
$$
A \cdot (T_1,T_2,T_3,T_4 ; T_5,T_6,T_7)
\ := \
(A \cdot (T_1,T_2,T_3),T_4; (A^{-1})^t \cdot (T_5,T_6,T_7)).
$$
\end{example}

The two previous examples fit into the class of
\emph{intrisic quadrics}, that means varieties
having a hypersurface Cox ring with a quadric as
defining relation. The ideas just observed 
lead to the following general observation.

\begin{corollary}
Let $X$ be a variety satisfying all the
assumptions of Proposition~\ref{prop:infdef}
and assume that $\Aut_H(\bar Z)$ acts almost
transitively on $\mathcal{R}(Z)_\mu$.
\begin{enumerate}
\item
The variety $X$ is infinitesimally rigid
and the dimension of its automorphism group
is given by
$$
\qquad\qquad
\dim(\Aut(X))
\ = \
\dim(\Aut_H(\bar Z)) - (\dim(\mathcal{R}(Z)_\mu) - 1) - \rk(\Cl(Z)).
$$
\item
If $X$ is an intrinisc quadric, then
$\Aut_H(\bar Z)$ acts almost transitively 
on $\mathcal{R}(Z)_\mu$ and thus 
the statements from~(i) hold for $X$.
\end{enumerate}
\end{corollary}

\begin{proof}
We take $X \subseteq Z$ as in
Construction~\ref{constr:hypersurf}.
According to~\cite[Thm.~4.4~(iv)]{HaKeWo},
the unit component $\Aut(X)^0$ equals the
stabilizer $\Aut(Z)^0_X$
of $X \subseteq Z$ under the action of
$\Aut(Z)^0$ on $Z$.
Thus, using~\cite[Thm.~4.2.4.1]{ADHL},
we obtain
\begin{eqnarray*}
\dim(\Aut(X))
& = &
\dim(\Aut(Z)^0_X)
\\
& = &   
\dim(\Aut_H(\bar Z)^0_{\bar X}) - \dim(H)
\\
& = & 
\dim(\Aut_H(\bar Z)^0) - (\dim(\mathcal{R}(Z)_\mu) - 1) - \rk(\Cl(Z)), 
\end{eqnarray*} 
where $\mathcal{R}(Z)_\mu$ is the
space of defining equations and ``$-1$''
pops up as we are looking for only the
zero sets of these equations.
Thus, Proposition~\ref{prop:infdef}
gives the first statement.
For the second one, note that $\Aut_H(\bar Z)$
acts almost transitively on $\mathcal{R}(Z)_\mu$
due to~\cite[Prop.~2.1]{FaHa}.
\end{proof}

Let us take up once more the geometric
discussion of the varieties from No.~9
of Theorem~\ref{thm:candidates}
started in Remarks~\ref{rem:geomdisc}
and~\ref{rem:no9again}.
Using geometric properties observed
so far, we see $\Aut(X)$ is trivial.

\begin{remark}
Let $X$ be as in Theorem~\ref{thm:candidates}, No.~9.
We claim that $\Aut(X)$ is finite in this case.
As a consequence, we obtain
\begin{eqnarray*}
h^1(X,\mathcal{T}_X)
& = & 
\dim(\mathcal{R}(Z)_\mu) - 1
+
\rk(\Cl(Z)) - \sum_{i=1}^r \dim(\mathcal{R}(Z)_{w_i})
\\
& = &
40 - 1 + 2 - 29
\\
& = &
12.
\end{eqnarray*}
Look at the fibration $\pi_1 \colon X \to Y_1 = \PP_3$
from Remark~\ref{rem:geomdisc}.
By~\cite[Prop.~I.1]{Bl}, there
is an induced action of the unit componenent
$\Aut(X)^0$ on $Y_1$ turning $\pi_1$ into an
equivariant map.
This means in particular that the induced
action permutes the image points of the
12 singular fibers of $\pi_1$.
By the generality assumption, these 12
points don't lie in a common hyperplane
and thus induced action of $\Aut(X)^0$
on~$Y_1$ must be trivial.
Recall that any point of the fiber
$\pi_1$ over $[y] = [y_1,\ldots,y_4]$
has Cox coordinates
$$
[y,x,z]
\ = \ 
[y_1,\ldots,y_4,x_1,x_2,z],
\qquad
\text{where}
\quad
q_1(y)x_1 + q_2(y)x_2 + f(y)z = 0,
$$
with general quadrics $q_1,q_2$ 
and a general cubic $f$ in the
first four variables.
Let us see in these terms what it
means that the $\pi_1$-fibers are
invariant under $\Aut(X)^0$.
Consider the action of the characteristic
quasitorus $H = \Spec \, \CC[\Cl(Z)]$
on $\bar Z = \CC^r$ given by the
$\Cl(Z)$-grading of $\CC[T_1,\ldots,T_r]$.
The group $\Aut_H( \bar Z)$ of
$H$-equivariant automorphisms is
concretely given as
$$
G \ = \ \GL(4) \times \GL(2) \times \KK^*.
$$
According to~\cite[Thm.~4.4]{HaKeWo},
we obtain $\Aut(X)^0$ as a factor group
of the unit component of the subgroup
$\Aut_H(\bar X)$ of $\Aut_H( \bar Z)$
stabilizing $\bar X \subseteq \bar Z$.
We take a closer look at the action of
an element
$\gamma = \diag(A_1,A_2,\alpha_3)$
of $\Aut_H(\bar X)$ on
$\hat X \subseteq \bar X$.
Given general $y \in \CC^4$ and $x \in \CC^2$,
we find $z \in \CC$ such that $[y,x,z]$
is a point of~$\hat X$. 
In particular, $\gamma \cdot [y,x,z]$
belongs to the fiber of $\pi_1$ over $[y]$.
The latter implies
$A_1 \cdot y = \eta y$ with $\eta \in \KK^*$
and for the matrix $A_2 = (a_{ij})$ it gives
\begin{eqnarray*}
0
& = & 
q_1(y) (a_{11} x_1 + a_{12} x_2)
+
q_2(y) (a_{21} x_1 +  a_{22} x_2 )
+
\alpha_3 f(y) z
\\
& = & 
q_1(y) ((a_{11}-\alpha_3) x_1 + a_{12} x_2)
+
q_2(y) (a_{21} x_1 +  (a_{22}-\alpha_3) x_2 ).
\end{eqnarray*}
Recall that this holds for any general choice
of $y$ and $x$.
As a consequence, we arrive at $a_{11}-\alpha_3 = 0 = a_{12}$,
because otherwise 
$q_1q_2^{-1} \in  \CC(T_1,T_2)$ holds in $\CC(X)$
which is impossible due to the general choice
of $q_1$ and $q_2$.
By the same argument, we see $a_{22}-\alpha_3 = 0 = a_{21}$.
Thus, $\gamma$ acts trivially on each fiber
of $\pi_1$ and we conclude that
$\Aut(X)$ is of dimension zero.
\end{remark}

Proposition~\ref{prop:infdef} suggests that
the infinitesimal deformations of $X$ can be
obtained by varying the coefficients of the
defining equation in the Cox ring.
As a possible approach to turn this impression into a
precise statement, we mention the comparison theorem
of Christophersen and Kleppe~\cite[Thm.~6.2]{ChKl} which
relates in particular deformations of a variety
to deformations of its Cox ring.

\end{document}